\documentclass[11pt]{amsart}
\usepackage{mathptmx, amssymb, amsmath, mathtools, enumerate, colonequals, fullpage}
\usepackage[all]{xy}
\usepackage{xcolor, color}
\definecolor{chianti}{rgb}{0.6,0,0}
\definecolor{meretale}{rgb}{0,0,.6}
\definecolor{leaf}{rgb}{0,.35,0}
\usepackage[colorlinks=true, pagebackref, hyperindex, citecolor=meretale, urlcolor=leaf, linkcolor=chianti]{hyperref}

\newtheorem{theorem}{Theorem}[section]
\newtheorem{lemma}[theorem]{Lemma}
\newtheorem{corollary}[theorem]{Corollary}
\newtheorem{proposition}[theorem]{Proposition}

\theoremstyle{definition}
\newtheorem{definition}[theorem]{Definition}
\newtheorem{construction}[theorem]{Construction}
\newtheorem{notation}[theorem]{Notation}
\newtheorem{example}[theorem]{Example}
\newtheorem{remark}[theorem]{Remark}
\numberwithin{equation}{theorem}

\def\ge{\geqslant}
\def\le{\leqslant}
\def\phi{\varphi}
\def\tilde{\widetilde}
\def\to{\longrightarrow}
\def\gets{\longleftarrow}
\def\mapsto{\longmapsto}
\def\into{\lhook\joinrel\longrightarrow}

\def\lcd{\operatorname{lcd}}
\def\rank{\operatorname{rank}}
\def\Hsing{H_{\mathrm{sing}}}
\def\coker{\operatorname{coker}}
\def\image{\operatorname{image}}
\def\Hom{\operatorname{Hom}}
\def\RHom{\operatorname{RHom}}
\def\Ext{\operatorname{Ext}}
\def\Tor{\operatorname{Tor}}

\def\Spec{\operatorname{Spec}}
\def\Supp{\operatorname{Supp}}
\def\Perv{\operatorname{Perv}}

\def\frakm{\mathfrak{m}}
\def\frakp{\mathfrak{p}}
\def\frakq{\mathfrak{q}}

\def\AA{\mathbf{A}}
\def\CC{\mathbf{C}}
\def\DD{\mathbf{D}}
\def\FF{\mathbf{F}}
\def\FFp{\mathbf{F}_{\!p}}
\def\GG{\mathbf{G}}
\def\LL{\mathbf{L}}
\def\NN{\mathbf{N}}
\def\PP{\mathbf{P}}
\def\QQ{\mathbf{Q}}

\def\ZZ{\mathbf{Z}}

\DeclareMathAlphabet{\mathcal}{OMS}{cmsy}{m}{n}

\def\calC{\mathcal{C}}
\def\calD{\mathcal{D}}
\def\calE{\mathcal{E}}

\def\calH{\mathcal{H}}

\def\calO{\mathcal{O}}
\def\calP{\mathcal{P}}

\def\calS{\mathcal{S}}

\def\calY{\mathcal{Y}}

\def\an{\mathrm{an}}
\def\et{\acute{e}t}
\def\perf{per\!f}
\def\cof{co\!f}
\def\SolC{\operatorname{Sol_\CC}} 
\def\dR{\operatorname{dR}} 
\def\Sol{\operatorname{Sol}} 
\def\RH{\operatorname{RH}} 

\begin{document}
\title{Applications of perverse sheaves in commutative algebra}

\begin{abstract}
The goal of this paper is to explain how basic properties of perverse sheaves sometimes translate via Riemann-Hilbert correspondences (in both characteristic $0$ and characteristic $p$) to highly non-trivial properties of singularities, especially their local cohomology. Along the way, we develop a theory of perverse~$\mathbf{F}_{\!p}$-sheaves on varieties in characteristic $p$, expanding on previous work by various authors, and including a strong version of the Artin vanishing theorem.
\end{abstract}

\author[Bhatt]{Bhargav Bhatt}
\address{School of Mathematics, Institute for Advanced Study, and Department of Mathematics, Princeton University, Princeton, USA}
\email{bhargav.bhatt@gmail.com}

\author[Blickle]{Manuel Blickle}
\address{Institut für Mathematik, Johannes Gutenberg-Universit\"at Mainz, Mainz, Germany}
\email{blicklem@uni-mainz.de}

\author[Lyubeznik]{Gennady Lyubeznik}
\address{School of Mathematics, University of Minnesota, Minneapolis, USA}
\email{gennady@math.umn.edu}

\author[Singh]{Anurag K. Singh}
\address{Department of Mathematics, University of Utah, Salt Lake City, USA}
\email{singh@math.utah.edu}

\author[Zhang]{Wenliang Zhang}
\address{Department of Mathematics, Statistics, and Computer Science, University of Illinois at Chicago, Chicago, USA}
\email{wlzhang@uic.edu}

\thanks{B.B.~was supported by NSF grants DMS~1801689, DMS~1952399, and DMS~1840234, the Packard Foundation, and the Simons Foundation; M.B.~by DFG grant SFB/TRR45 and CRC326 GAUS; G.L.~by NSF grant DMS~1800355, A.K.S.~by NSF grants DMS~1801285, DMS~2101671, and DMS~2349623; and W.Z.~by NSF grant DMS~1752081. The authors are also grateful to the American Institute of Mathematics (AIM) and the Institute for Advanced Study (IAS) for supporting their collaboration.}

\setcounter{tocdepth}{1}
\maketitle
\tableofcontents

\section{Introduction}
\label{section:introduction}

Perverse sheaves \cite{GM1, GM2, BBDG} were invented to study the topology of singular spaces, especially algebraic varieties; in the last $40$ years, they have had a deep impact on many areas of mathematics, especially representation theory and arithmetic geometry. Commutative algebra, to some extent, can also be regarded as a study of singularities of algebraic varieties. The goal of this paper is to use perverse sheaves to give relatively soft ``topological'' proofs, in both characteristic $0$ and characteristic $p$, of several important results in commutative algebra concerning invariants of singularities, especially local cohomology. Our main tool is the Riemann-Hilbert correspondence, which relates perverse sheaves to holonomic $\calD$-modules (over $\CC$) or to Frobenius modules (in characteristic $p$).

We remark that holonomic $\calD$-modules have been used in commutative algebra over $\CC$ for at least 30 years \cite{Lyubeznik:Invent}; however, perverse sheaves seem less common. But already over $\CC$, the topological perspective offers some concrete benefits as the theory of perverse sheaves works perfectly well directly on arbitrarily singular spaces (unlike the theory of holonomic $\calD$-modules, which works most directly using a smooth embedding); see \S\ref{ss:EmbInddRH} for an application where this benefit becomes transparent. Secondly, the topological proofs sometimes work surprisingly uniformly in all characteristics (compare \S\ref{ss:GradingChar0} and \S\ref{ss:GradingCharp}). Finally, the topological perspective suggests approaches in mixed characteristic, where neither the theory of holonomic $\calD$-modules nor the theory of Frobenius modules is natively available; see Remark~\ref{mixedcharinspire} for an example.

\subsection*{Results}

In \S\ref{section:RHC}, we recall the classical Riemann-Hilbert correspondence in characteristic zero, emphasizing the connection between local cohomology and perverse truncations under this correspondence. Using this relation, we give topological proofs of the following results:
\begin{enumerate}[\quad\rm(1)]
\item Fundamental results of Ogus \cite{OgusLocalCoh} connecting local cohomological dimension to de~Rham depth --- (\S\ref{ss:Ogus}).

\item The main result of Ma-Zhang \cite{Ma-Zhang} on possible gradings of certain local cohomology modules of homogeneous ideals in characteristic zero --- (\S\ref{ss:GradingChar0}).

\item The topological interpretation of the Bass numbers of certain local cohomology modules, as given in Lyubeznik-Singh-Walther \cite[Theorem~3.1]{LSW} --- (\S\ref{ss:BassChar0}).

\item The embedding independence (from $E_2$ onward) of the Hodge-to-de-Rham spectral sequence for algebraic de~Rham homology due to Bridgland \cite{Bridgland} --- (\S\ref{ss:EmbInddRH}).

\item The embedding independence of Lyubeznik numbers of (affine cones over) smooth projective varieties due to Switala \cite{Switala} --- (\S\ref{ss:EmbIndepLyuChar0}).
\end{enumerate}
An earlier instance where Riemann-Hilbert techniques have been used to determine local cohomological invariants can be found in \cite{GarciaSabbah}.

We next turn our attention to characteristic $p$. The theory of perverse $\FFp$-sheaves on varieties in characteristic $p$ is not as well behaved as in characteristic $0$, e.g., there is no Verdier duality, and certain functors do not preserve constructibility. Furthermore, the theory is also not as well developed as its characteristic~$0$ counterpart. In \S\ref{sec:RHPervFp}, we develop basic properties of this theory using the Riemann-Hilbert correspondence from \cite{BhattLurieRH}, generalizing earlier works of \cite{EK,Gabbert,BoPi}. One of our basic results here is the following surprisingly strong version (to us, at any rate) of Artin vanishing:

\begin{theorem}[The perverse Artin vanishing theorem for $\FFp$-coefficients in characteristic $p$]
Let $f\colon X \to Y$ be an affine map of varieties over a field of characteristic $p$. Then
\[
f_!\colon D^b_c(X_{\et}, \FFp) \to D^b_c(Y_{\et}, \FFp)
\]
is $t$-exact for the perverse $t$-structure on the source and the target. In particular, for any affine variety $Z$, we have $H^i_c(Z,F) = 0$ for $i \neq 0$, and $F$ a perverse $\FFp$-sheaf.
\end{theorem}

Armed with these basic results, we use the Riemann-Hilbert correspondence to obtain topological proofs of the following results on invariants of singularities in characteristic $p$:

\begin{enumerate}[\quad\rm(1)]
\item A new interpretation of generalized Cohen-Macaulay and rational singularities in terms of the perversity and perverse simplicity of the constant sheaf $\FFp[\dim]$ --- (\S\ref{ss:ICFRat})

\item The embedding independence of Lyubeznik numbers of (affine cones over) arbitrary projective varieties due to W.~Zhang \cite{ZhangLyuIndep} --- (\S\ref{ss:EmbIndepLyuCharp}).

\item A new topological interpretation of Bass numbers, mirroring that in characteristic $p$ as in Lyubeznik-Singh-Walther \cite[Theorem~3.1]{LSW} --- (\S\ref{ss:BassCharp}).

\item The main result of Y. Zhang \cite{YiZhang} on possible gradings of certain local cohomology modules of homogeneous ideals in characteristic $p$ --- (\S\ref{ss:GradingCharp}).

\item Fundamental results of Hochster-Huneke \cite{HHRPlusCM} and Smith \cite{SmRPlusFRat} on the Cohen-Macaulay and $F$-rationality properties of the absolute integral closure --- (\S\ref{ss:HHSm}).

\item The graded version of the Cohen-Macaulay and $F$-rationality properties in (5), originally proven in \cite{HHRPlusCM,SmRPlusFRatGr}, giving Kodaira vanishing up to finite covers --- (\S\ref{ss:HHSmGr}).
\end{enumerate}

\subsection*{Notation}

We follow standard notational conventions regarding cohomology and derived categories. Let us recall some of them here to avoid future confusion. Fix a space $X$, an object $K \in D(X)$, and an integer $i \in \ZZ$.

We shall often say ``$K \in D^{\le i}$'' (without specifying the ambient category where $K$ lives) to indicate that~$\calH^j(K) = 0$ for $j \le i$; similarly for the perverse $t$-structure as well.

We shall write $H^i(X,K)\colonequals H^i(R\Gamma(X,K))$. Thus, when $K = \CC_X$ is the constant sheaf and $i$ is negative, we are setting $H^i(X,\CC_X)$ to be $0$.

\subsection*{Acknowledgments}

This paper had a rather long gestation period (the first version was ready in 2016) during which it was ``unofficially distributed'' in preprint form to various people.\footnote{During the long preparation period of this paper, some of the results/techniques discussed here have been rediscovered independently. These include, for example, \cite{RSW2} for results in characteristic zero, and \cite{CassPerv, BaudinDual, CassLau} for results in positive characteristic.} We apologize for the long delay in its completion. We had conversations with and feedback from a number of colleagues on the material in this article, including Robert Cass, Jacob Lurie, Linquan Ma, Mircea Musta\c t\u a, Mihnea Popa, Karl Schwede, and Jakub Witaszek. We are also most grateful to the anonymous referee for their insightful comments and suggestions, which improved the paper significantly.

\section{Recollections on the Riemann-Hilbert correspondence in characteristic $0$}
\label{section:RHC}

Let $X$ be a smooth connected algebraic variety over $\CC$ of dimension $d$. Write $\calD_X$ for the sheaf of differential operators on $X$, and let $D^b(\calD_X)$ be the bounded derived category of quasi-coherent left $\calD_X$-modules. Let~$D^b_{rh}(\calD_X)$ be the full subcategory of $D^b(\calD_X)$ spanned by complexes whose cohomology sheaves are regular holonomic $\calD_X$-modules \cite[Definition~6.1.1]{HTTDmodules}. If $f\colon X\to Y$ is a morphism of smooth algebraic varieties, then the two corresponding pushforward functors take $D^b_{rh}(\calD_X)$ to $D^b_{rh}(\calD_Y)$, and the two corresponding pullback functors take $D^b_{rh}(\calD_Y)$ to $D^b_{rh}(\calD_X)$ \cite[Theorem~6.1.5\,(ii)]{HTTDmodules}.

Let $X^\an$ be the analytic space associated to $X$, let $\CC_{X^\an}$ denote the constant sheaf of complex numbers $\CC$ on $X^\an$, and let $D^b(\CC_{X^\an})$ be the bounded derived category of $\CC_{X^\an}$-modules. A stratification of the variety~$X$ is a locally finite partition $X=\cup X_\alpha$ by mutually disjoint strata $X_\alpha$, such that each stratum $X_\alpha$ is a smooth locally closed subvariety of $X$, and for every $\alpha$ the closure $\overline{X}_\alpha$ is the union of some subset of the strata. An algebraically constructible sheaf $F$ on $X^\an$ is a $\CC_{X^\an}$-module such that all its stalks are finite-dimensional vector spaces over $\CC$ and there is a stratification $X=\cup X_\alpha$ with the property that the restriction $F|_{X^\an_\alpha}$ is locally constant for every $\alpha$ \cite[Definition~4.5.6]{HTTDmodules}. We write $D^b_c(X^\an)$ for the full triangulated subcategory of $D^b(\CC_{X^\an})$ consisting of bounded complexes of $\CC_{X^\an}$-modules with algebraically constructible cohomology sheaves \cite[Notation~4.5.7]{HTTDmodules}.

According to \cite[page~111]{HTTDmodules}, there is the (contravariant) Verdier duality functor
\[
\DD_{X^\an}(-)=\RHom_{\CC_{X^\an}}(-,\omega^\bullet)\colon D^b(\CC_{X^\an})\to D^b(\CC_{X^\an}),
\]
where $\omega^\bullet\in D^b(\CC_{X^\an})$ is the dualizing complex of $X^\an$ \cite[Definition~4.5.2]{HTTDmodules}; if $X$ is smooth, then
\[
\omega^\bullet\simeq\CC_{X^\an}[2\dim X],
\]
\cite[page~111]{HTTDmodules}. Verdier duality preserves the category $D^b_c(X^\an)$ \cite[Theorem~4.5.8\,(i)]{HTTDmodules}. If~$f\colon X\to Y$ is a morphism of algebraic varieties, then there are the pullback and pushforward functors
\[
f^*,f^!\colon D^b_c(Y^\an)\rightleftarrows D^b_c(X^\an) \colon f_*,f_! 
\]
that satisfy
\[
f^!=\DD\circ f^*\circ \DD \quad\text{ and }\quad f_!=\DD\circ f_*\circ \DD,
\]
i.e., $\DD\circ f^!=f^*\circ \DD$ and $\DD\circ f_!=f_*\circ \DD$ \cite[Theorem~4.5.8\,(ii)]{HTTDmodules}; note that the functor $(-)^*$ is denoted $(-)^{-1}$ in \cite{HTTDmodules} and we generally mean the total derived functors. We briefly recall from \cite[\S8]{HTTDmodules} and \cite{BBDG} the notations and basic properties of the perverse $t$-structure on the bounded derived category $D^b_c(X^\an)$ of constructible sheaves of $\CC$-vector spaces on $X^\an$.

The two subcategories defined as follows 
\begin{align*}
M^\bullet \in {}^pD^{\le 0}_c(X^\an) &\iff i_x^*M^\bullet=M^\bullet_x \in D^{\le -\dim{\overline{\{x\}}}}\text{ for all points $x\in X$} \\
M^\bullet \in {}^pD^{\ge 0}_c(X^\an) &\iff i_x^!M^\bullet=R\Gamma_xM^\bullet_x \in D^{\ge -\dim{\overline{\{x\}}}}\text{ for all points $x\in X$ }
\end{align*}
define a $t$-structure on $D^b_c(X^\an)$. The heart of this $t$-structure is an abelian category
\[
\Perv(X) \colonequals {}^pD^{\le 0}_c(X^\an) \cap {}^pD^{\ge 0}_c(X^\an)
\]
called the perverse constructible sheaves on $X$. 

We use ${ }^p \tau^{\le i}$, ${ }^p \tau^{\ge i}$, and ${ }^p \calH^i$ for the perverse truncation and cohomology functors \cite[Definition~8.1.28]{HTTDmodules}. A functor is called right $t$-exact (resp.\ left $t$-exact) if it preserves $D^{\le 0}$ (resp.\ $D^{\ge 0}$). Let us recall some basic properties of perverse sheaves which will play a role in what follows, see \cite[\S8]{HTTDmodules}:

\begin{enumerate}[\quad\rm(1)]
\label{t-exactness-properties}
\item If $X$ is smooth, then $\CC_X[\dim X]$ is perverse. Note that in general we only have $\CC_X[\dim X] \in {}^pD^{\le 0}_c(X)$.

\item (immersions) If $i\colon Y \to X$ is a locally closed immersion then:
\begin{itemize}
\item $i^*$ is right $t$-exact and $i^!$ is left $t$-exact. In particular, if $Y \subseteq X$ is open, then $i^!=i^*$ is $t$-exact. 
\item $i_*$ is right $t$-exact and $i_!$ is left $t$-exact. In particular, if $Y \subseteq X$ is closed, then $i_*=i_!$ is $t$-exact. 
\end{itemize}
Hence, if $Y$ is closed, then $i_*\CC_X[\dim Y] \in {}^pD^{\le 0}_c(X)$. If $Y$ is also smooth, then $i_*\CC_X[\dim Y]$ is perverse. 

\item (smooth maps) If $f\colon X \to Y$ is smooth of relative dimension $d$, then $f^*(-)[d]=f^![-d]$ is $t$-exact.

\item (Artin vanishing \cite[Theorem~4.1]{BBDG}) If $f\colon X \to Y$ is affine, then $f_*$ is right $t$-exact.

\item (Verdier duality) The Verdier duality $\DD_X$ is exact and $t$-exact. In particular $\DD_X(\Perv(X)) = \Perv(X)$
\end{enumerate}

For a complex $M^\bullet\in D^b(\calD_X)$ let $(M^\bullet)^{\an}\in \calD^b_{X^{\an}}$ be the corresponding complex in the analytic category \cite[p.~120]{HTTDmodules}. The de~Rham complex of $M^\bullet\in D^b(\calD_{X})$ is defined in \cite[p.~122]{HTTDmodules} as
\[
\dR(M^\bullet)\colonequals\Omega_{X^{\an}}\otimes_{\calD_{X^\an}}^\LL (M^\bullet)^{\an},
\]
where $\Omega_{X^{\an}}$ is the canonical sheaf of $X^\an$, i.e., the highest exterior power of the sheaf $\Omega^1_{X^\an/\CC}$ of $\CC$-linear differentials of $X^\an$; the canonical sheaf $\Omega_{X^{\an}}$ has a natural structure of a right $\calD_{X^\an}$-module. There results a de~Rham functor
\[
\dR\colon D^b(\calD_{X})\xrightarrow{\ M^\bullet\mapsto \dR(M^\bullet)\ } D^b(\CC_{X^\an}).
\]
It is clear from the definition that the de~Rham functor respects the translation functor and takes distinguished triangles to distinguished triangles.

The functor $(-)^{\an}\colon D^b_{rh}(\calD_X)\to D^b_{rh}(\calD_{X^{\an}})$ is fully faithful ({\it cf.} \cite[Proposition~7.8]{Brylinski}). Moreover, according to \cite[4.7.4]{HTTDmodules}, 
\[
\dR(M^{\bullet})\colonequals \Omega_{X^{\an}}\otimes_{\calD_{X^\an}}^\LL (M^\bullet)^{\an}\simeq {\rm R}\mathcal Hom_{\calD_{X^\an}}(\calO_{X^\an},(M^\bullet)^\an)[\dim(X)];
\] 
likewise
\[
\Omega_X\otimes_{\calD_X}^{\LL}M^{\bullet}\simeq {\rm R}\mathcal Hom_{\calD_X}(\calO_X,M^{\bullet})[\dim(X)].
\]

We will use the following covariant form of the Riemann-Hilbert correspondence, see \cite[\S7]{HTTDmodules}:

\begin{theorem}
\label{thm:RH}
Restricting to the subcategory of \emph{regular holonomic} $\calD_X$-modules $D^b_{rh}(\calD_X)$, the de~Rham functor induces an equivalence of categories
\[
\dR\colon D^b_{rh}(\calD_X) \to D^b_c(X^\an),
\]
that satisfies the following properties:
\begin{enumerate}[\quad\rm(1)]
\item \label{thm:RH:1} $\dR$ commutes with the duality functors and with the standard pushforward and pullback functors for every morphism $f\colon X\to Y$ of smooth algebraic varieties. In particular, it commutes with restriction to open subsets.

\item \label{thm:RH:2} $\dR$ identifies the standard $t$-structure on $D^b_{rh}(\calD_X)$ with the perverse $t$-structure on~$D^b_c(X^\an)$. In particular, for $K^\bullet \in D^b_{rh}(\calD_X)$, and each integer $i$, we have
\[
\dR(\calH^i(K^\bullet)) \simeq { }^p \calH^i (\dR(K^\bullet)).
\]

\item \label{thm:RH:3} $\dR(\calO_X) \simeq \CC_X[d]$, where $d=\dim X$.

\item \label{thm:RH:4} $\dR$ is functorial with respect to automorphisms of $X$. In particular, if $G$ is an algebraic group acting on $X$, then $\dR$ induces an identification of the $G$-equivariant analogs of~$D^b_{rh}(\calD_X)$ and $D^b_c(X^\an)$, as defined in \cite{BernsteinLunts}.
\end{enumerate}
\end{theorem}

\begin{proof}
The fact that the de~Rham functor is an equivalence is proven in \cite[Theorem~7.2.2]{HTTDmodules}.

(1) Commutativity with the duality functors is proven in \cite[Corollary~4.6.5 and Proposition~4.7.9]{HTTDmodules}. Commutativity with the pushforward and pullback functors is proven in \cite[Theorem~7.1.1]{HTTDmodules}.

(2) This is straightforward from \cite[Theorem~7.2.5]{HTTDmodules} that says that the de~Rham functor $\dR$ induces an equivalence between Mod$_{rh}(\calD_X)$, i.e., the category of regular holonomic $\calD_X$-modules on $X$, which is the core of the standard $t$-structure, and Perv$(\CC_{X^\an})$, i.e., the category of perverse sheaves on $X^\an$, which is the core of the perverse $t$-structure.

(3) There is a locally free resolution of the right $\calD_{X^\an}$-module $\Omega$ described in \cite[Lemma~1.5.27]{HTTDmodules}. This resolution is concentrated in homological degrees $-d$ through 0. Applying $-\otimes_{\calD_{X^\an}}\calO_X$ to this resolution yields a complex whose cohomology in degree $-d$ is the constant sheaf $\CC_{X^\an}$, and whose cohomology in all higher degrees vanish by the holomorphic Poincar\'e lemma. Hence $\dR(\calO_{X^\an})$ is quasi-isomorphic to $\CC_X[d]$.

(4) is a consequence of the commutativity of the de~Rham functor with pullbacks and pushforwards.
\end{proof}

Next we summarize the relevant features of local cohomology in the context of the Riemann-Hilbert correspondence:

\begin{example}[Local cohomology under Riemann-Hilbert]
\label{ex:LocalCohPerverse}
Let $i\colon Z \into X$ be a closed immersion with open complement $j\colon U \into X$. For every $M^\bullet\in D^b_{rh}(\calD_X)$ there is a distinguished triangle
\[
R\underline\Gamma_Z(M^\bullet)\to M^\bullet\to j_*j^*(M^\bullet).
\]
As standard pushforwards and pullbacks preserve regular holonomicity, $j_*j^*(M^\bullet)\in D^b_{rh}(\calO_X)$, so~$R\underline\Gamma_Z(M^\bullet)$ is also in $D^b_{rh}(\calD_X)$. Thus, local cohomology modules of regular holonomic modules are regular holonomic. In particular, as $\calO_X\in D^b_{rh}(\calD_X)$, it follows that $R\underline\Gamma_Z(\calO_X)\in D^b_{rh}(\calD_X)$. Setting $M^\bullet=\calO_X$ in the above distinguished triangle and applying the functor $\dR$ gives the following distinguished triangle in~$D^b_c(\CC_X)$:
\[
\dR(R\underline\Gamma_Z(\calO_X))\to \dR(\calO_X)\to \dR(j_*j^*(\calO_X)).
\]
Since $\dR(\calO_X)\simeq \CC_X[d]$ by Theorem~\ref{thm:RH}.\ref{thm:RH:3}, and $\dR(j_*j^*(\calO_X))=j_*j^*(\CC_X[d])$ by Theorem~\ref{thm:RH}.\ref{thm:RH:1}, comparing this distinguished triangle with the distinguished triangle
\[
i_*i^!(\CC_X[d])\to \CC_X[d]\to j_*j^*(\CC_X[d]),
\]
we conclude that
\begin{equation}
\label{eq:RHLocalCoh1}
\dR(R\underline\Gamma_Z(\calO_X)) \simeq i_* i^! \CC_X[d].
\end{equation}
In particular, we obtain
\begin{equation}
\label{eq:RHLocalCoh2}
\dR(\calH^i_Z(\calO_X)) \simeq { }^p \calH^i(i_* i^! \CC_X[d]).
\end{equation}
As a special case, assume that $X = \Spec R$ is (smooth and) affine of dimension $d$, and that $Z = \{x\}$ is a closed point of $X$ defined by the maximal ideal $\frakm\subseteq R$. Then
\[
R\underline\Gamma_Z(\calO_X) = H^d_\frakm(R)[-d],
\]
and $H^d_{\frakm}(R) \simeq E$ is the injective hull of the residue field of $R$. Using Verdier duality we see that
\[
\DD_{\{x\}}(i^!\CC_X[d]) = i^*\DD_{X^\an} \CC_X[d] = i^*\CC_X[d] = \CC_{\{x\}}[d] = \DD_{\{x\}}(\CC_{\{x\}}[-d])
\]
to conclude:
\[
i_* i^! \CC_X[d] \simeq i_* \CC_{\{x\}}[-d].
\]
The preceding formula then shows that $\dR(H^d_{\frakm}(R)) \simeq \CC_{\{x\}}$. In other words, the regular holonomic $\calD_X$-module $H^d_{\frakm}(R)$ corresponds to the skyscraper sheaf with value $\CC$ supported at $x \in X^\an$.

In the sequel, it will be useful to record the Verdier dual form of~\eqref{eq:RHLocalCoh1} and~\eqref{eq:RHLocalCoh2}. Let $\DD$ be the Verdier duality functor $D^b_c(X^\an)$; it is an anti-equivalence, and $t$-exact for the perverse $t$-structure. One defines\footnote{There are concrete formulae for $\Sol\CC$ and $\dR$ but they do not play a role in our applications, see \cite[\S5]{HTTDmodules}. In our treatment of the characteristic $p$ case later, it is the $\Sol$ functor that has an analog in that situation.}
\[
\SolC \colonequals \DD \circ \dR.
\]
Together with the standard formulae $\DD(\CC_X[d]) \simeq \CC_X[d]$, and $\DD \circ i_* \simeq i_* \circ \DD$, and $\DD \circ i^! \simeq i^* \circ \DD$, one obtains the following translation of~\eqref{eq:RHLocalCoh1}:
\begin{equation}
\label{eq:RHLocalCoh1Dual}
\SolC(R\underline\Gamma_Z(\calO_X)) \simeq i_* i^* \CC_X[d] \simeq i_* \CC_Z[d].
\end{equation}
As $i_*$ is $t$-exact (see page~\pageref{t-exactness-properties}),~\eqref{eq:RHLocalCoh2} then becomes
\begin{equation}
\label{eq:RHLocalCoh2Dual}
\SolC(\calH^i_Z(\calO_X)) \simeq i_* \big({ }^p \calH^{d-i}(\CC_Z)\big).
\end{equation}
In particular, as the Riemann-Hilbert correspondence is compatible with restriction to open subsets, questions about the support of local cohomology modules $\calH^i_Z(\calO_X)$ translate to questions about the support of the perverse sheaves ${ }^p \calH^{d-i}(\CC_Z)$.
\end{example}

\section{Applications of the Riemann-Hilbert correspondence in characteristic $0$}
\label{char0}

Our goal in this section is to exploit the Riemann-Hilbert correspondence (RHC) to study local cohomology, recovering the results mentioned in \S\ref{section:introduction}. A particular feature of this approach is that the proofs on the topological side (as explained below) involve far fewer calculations than those on the algebraic side, and use nothing beyond the basic formalism of perverse sheaves and standard pullback/pushforward functors.

\subsection{Recovering results of Ogus}
\label{ss:Ogus}

We recover some fundamental results of Ogus~\cite{OgusLocalCoh} using the perverse sheaf perspective. For the rest of this section, fix a smooth affine algebraic variety $X$ of dimension $d$ over $\CC$, as well as a closed subvariety $i\colon Y \into X$.
Ogus studies the local cohomological dimension
\[
\lcd(Y,X) \colonequals \max\{i \mid \calH^i_Y(\calO_X) \neq 0\},
\]
which, via Riemann-Hilbert as noted in equation~\eqref{eq:RHLocalCoh2Dual}, satisfies
\[
d - \lcd(Y,X) = \min\{ j \mid {}^p\calH^j(\CC_Y)\neq 0\}.
\]
In other words, $R\underline\Gamma_Y(\calO_X) \in D^{\le d-r}_{rh}(\calD_X)$ if and only if $\CC_Y \in { }^p D^{\ge r}_c(Y^\an)$. In \cite[Theorem~2.13]{OgusLocalCoh}, Ogus relates $\lcd(Y,X)$ to the de~Rham depth of $Y$ to show:

\begin{theorem}[Ogus]
Fix an integer $r$. Then the following are equivalent:
\begin{enumerate}[\quad\rm(1)]
\item $\calH^i_Y(\calO_X) = 0$ for $i > d-r$, i.e., $\lcd(Y,X) \le d-r$.

\item There exists a stratification $\{k_i\colon Z_i \into Y\}$ of $Y$ such that for all closed points $z \in Z_i$, we have $H^i_{z}(Y^\an,\CC_Y) = 0$ for $i < r+d_i$, where $d_i = \dim Z_i$.
\end{enumerate}
\end{theorem}

Here, $H^i_{z}(Y,\CC_Y)$ is defined to be the $i$-th cohomology of the homotopy-kernel $R\Gamma(k^!\CC_Y)$ of the map
\[
R\Gamma(Y^\an,\CC_Y) \to R\Gamma((Y - \{z\})^\an,\CC_{Y - \{z\}})
\]
and we point out that by Hartshorne \cite[\S IV, Theorem~1.1]{Hartshorne:AlgdR} this can be expressed in terms of algebraic de~Rham cohomology of $Y$, see \cite[\S III]{Hartshorne:AlgdR}. We have also used \cite[Theorem~6.2]{Hartshorne:AlgdR} to simplify the condition of de~Rham depth occurring in \cite{OgusLocalCoh}, avoiding all mention of non-closed points.

\begin{proof}
Assume (1), so $R\underline\Gamma_Y(\calO_X) \in D^{\le d-r}_{rh}(\calD_X)$. By~\eqref{eq:RHLocalCoh2Dual}, this translates to $\CC_Y \in { }^p D^{\ge r}_c(X^\an)$. By constructibility of the Verdier dual $\DD_X(i_*\CC_Y)$, choose a stratification $\{k_i\colon Z_i \into Y\}$ such that $Z_i$ is smooth of dimension $d_i$ and $k_i^*\DD_X(i_*\CC_Y)$ is lisse, i.e., has locally constant cohomology sheaves. As Verdier duality on a smooth scheme preserves lisse complexes, dualizing again shows that $k_i^!\CC_Y = \DD_{Z_i} k_i^*\DD_X(i_*\CC_Y)$ is lisse.

As $k_i^!$ is left $t$-exact (see page~\pageref{t-exactness-properties}), we obtain $k_i^! \CC_Y \in { }^p D^{\ge r}$. Moreover, as the perverse and standard $t$-structure on lisse complexes are the same up to a shift by the dimension of the variety \cite[Corollary~8.1.23]{HTTDmodules}, we get~$k_i^! \CC_Y \in D^{\ge r-d_i}$. If $k_z\colon \{z\} \into Z_i$ is the inclusion of a closed point, then $k_z^! = k_z^*[-2d_i]$ on lisse complexes \cite[Theorem~4.5.8\,(ii)]{HTTDmodules}, so
\[
k_z^! k_i^! \CC_Y \simeq k_z^\ast \big(k_i^! \CC_Y\big)[-2d_i] \in D^{\ge r-d_i}[-2d_i] = D^{\ge r+d_i},
\]
where we use that $k_z^*$ is exact for the standard $t$-structure. Now $k_i \circ k_z$ is simply the inclusion $h\colon\{z\} \into Y$. If $\ell\colon Y - \{z\} \to Y$ denotes the open complement, then there is an exact triangle
\[
h_* h^! \CC_Y \to \CC_Y \to \ell_* \CC_{Y - \{z\}}.
\]
As $\{z\}$ is a point, the complex $h^! \CC_Y$ is identified (via the previous triangle) with the homotopy-kernel of
\[
R\Gamma(Y^\an, \CC) \to R\Gamma((Y - \{z\})^\an,\CC).
\]
Thus, the condition $h^! \CC_Y \in D^{\ge r+d_i}$ translates to $H^i_{z}(Y^\an,\CC_Y) = 0$ for $i < r+d_i$, which gives~(2).

By definition of the perverse $t$-structure, the condition $\CC_Y \in { }^p D^{\ge r}_c(Y^\an)$ is equivalent to: For any locally closed immersion $h\colon Z \into Y$ one has $h^!\CC_Y \in D^{\ge r - \dim Z}_c(Z^\an)$. By passing to a finer stratification (that also respects $Z$), the above reasoning is essentially reversible, so we get the converse as well.
\end{proof}

As an application of this, we can prove another result of Ogus, which, in some ways, is a specialization of the former (see \cite[Theorem~2.8]{OgusLocalCoh}):

\begin{theorem}[Ogus]
Fix a closed point $x \in Y$. If $\calH^i_Y(\calO_X) = 0$ for all $i > d-r$, then $H^i_{x}(Y^\an,\CC_Y) = 0$ for~$i < r$. The converse is also true if $\calH^i_Y(\calO_X)$ is known to be supported at $x$ for $i > d-r$.
\end{theorem}

\begin{proof}
The forward direction follows from the previous result as $x$ must lie in some stratum. For the converse, consider the inclusion $k\colon\{x\} \into Y$. If $H^ik^!\CC_Y = H^i_{x}(Y^\an,\CC_Y) = 0$ for $i < r$, then $k^! \CC_Y \in D^{\ge r}$. If we additionally know that $\tau^{> d-r} R\underline\Gamma_Y(\calO_X)$ is supported at $x$, we learn, by~\eqref{eq:RHLocalCoh2Dual}, that ${ }^p \tau^{< r} \CC_Y$ is supported at $x$, i.e., is of the form $k_* A$ for suitable $A \in D^{< r}$. On the other hand, applying $k^!$ to the exact triangle
\[
k_* A \colonequals { }^p \tau^{< r} \CC_Y \to \CC_Y \to { }^p \tau^{\ge r} \CC_Y
\]
then gives a triangle
\[
A \simeq k^! k_* A \to k^! \CC_Y \to B
\]
where $B \in D^{\ge r}$ by the left $t$-exactness of $k^!$. Since $k^! \CC_Y \in D^{\ge r}$ by our assumption, we must have $A = 0$, so~$\CC_Y \in D^{\ge r}$. This proves the claim by~\eqref{eq:RHLocalCoh2Dual}.
\end{proof}

Results using the same techniques in the analytic context have independently appeared in \cite{RSW2}. In~\cite{MuPo} a further interesting characterization of $\lcd(Y,X)$ in terms of resolutions of singularities is obtained.

\subsection{Gradings}
\label{ss:GradingChar0}

We now explain how to use RHC to understand gradings of certain local cohomology modules with finite support. For this, note that, by RHC, it is trivial to classify regular holonomic $\calD_X$-modules supported at a point: it amounts to the classification of finite dimensional vector spaces. After incorporating group actions, this lets us recover the main consequence of \cite{Ma-Zhang} in the regular holonomic case (which is the only one relevant to the prequel); from this optic, the main idea is that $\GG_m$ is connected.

\begin{proposition}
\label{prop:BassGradedRH}
Let $X = \Spec R$ be a smooth affine variety of dimension $d$ over $\CC$. Let~$M$ be a regular holonomic $\calD_X$-module that vanishes away from a closed point $x \in X$ corresponding to a maximal ideal $\frakm$. Then~$M \simeq H^d_\frakm(R)^{\oplus \mu}$, i.e., $M$ is isomorphic to a finite direct sum of copies of the injective hull of $R/\frakm$.

Assume, moreover, that $X$ admits a $\GG_m$-action that fixes $x$, and that $M$ is $\GG_m$-equivariant. Then any $\calD_X$-module isomorphism $M \simeq H^d_{\frakm}(R)^{\oplus \mu}$ is automatically $\GG_m$-equivariant.
\end{proposition}

\begin{proof}
Since the support is preserved under the Riemann-Hilbert correspondence, regular holonomic $\calD_X$-modules that vanish away from $x$ correspond to perverse sheaves that vanish away from $x$. Since $\{x\}$ is $0$-dimensional, the latter category is the same as that of constructible sheaves supported at $x$, i.e., that of finite dimensional vector spaces (via taking stalks). Thus, $\dR(M) \simeq \CC_{\{x\}}^{\oplus \mu}$. By Example~\ref{ex:LocalCohPerverse}, we know that~$\dR(H^d_{\frakm}(R)) \simeq \CC_{\{x\}}$. As $\dR$ is an equivalence, we get $M \simeq H^d_{\frakm}(R)^{\oplus \mu}$ as wanted.

For the second part, our hypothesis ensures that $M$ defines a $\GG_m$-equivariant regular holonomic $\calD_X$-module on $X$ supported at $x$, and thus corresponds to a $\GG_m$-equivariant constructible sheaf on $X$ supported at~$x$ (via Theorem~\ref{thm:RH}.\ref{thm:RH:4} and the previous paragraph). The latter category is (essentially by definition, see~\cite{BernsteinLunts}) the category of constructible sheaves on the quotient stack $\calY \colonequals [\Spec k/\GG_m]$. As $|\calY|$ is a point, constructible sheaves on $\calY$ are locally constant. Moreover, since $\GG_m$ is connected, the fundamental group of $\calY$ is trivial, so such sheaves are direct sums of the constant sheaf. Translating back via RHC then gives the claim.
\end{proof}

\begin{remark}
The second half of Proposition~\ref{prop:BassGradedRH} is a special case of the following more general, and well-known fact: if $G$ is a connected algebraic group acting on a smooth variety $X$, then the forgetful functor from $G$-equivariant regular holonomic $\calD_X$-modules to all regular holonomic $\calD_X$-modules is fully faithful. Indeed, this follows from RHC and \cite[Proposition~4.2.5]{BBDG} applied to the universal $G$-torsor $X \to X/G$. In particular, a regular holonomic $\calD_X$-module carries at most one $G$-equivariant structure.
\end{remark}

Upon observing that $\GG_m$-equivariant $R$-modules are synonymous with $\ZZ$-graded $R$-modules, this leads to the following consequence, which roughly corresponds to \cite[Theorem~1.2]{Ma-Zhang}:

\begin{example}
\label{ex:MZviaRH}
Let $R = \CC[x_1,\dots,x_d]$ be a polynomial ring with its standard grading, and~$\frakm$ the homogeneous maximal ideal. Let $J_1,\dots,J_n$ be a collection of homogeneous ideals. The local cohomology module $H^d_\frakm(R)$ is then naturally a graded $R$-module. Proposition~\ref{prop:BassGradedRH} implies that for any sequence of integers $i_0,\dots,i_n$,~if
\[
M \colonequals H^{i_0}_{\frakm}(H^{i_1}_{J_1}(\cdots(H^{i_n}_{J_n}(R))),
\]
then $M \simeq H^d_{\frakm}(R)^{\oplus\mu}$ as a graded $R$-module, for suitable $\mu$, since the support of $M$ is contained in~$\{\frakm\}$ by construction. In particular, if $I$ is a homogeneous ideal such that $H^i_I(R)$ vanishes away from $\{\frakm\}$ for some~$i$, then $H^i_I(R) \simeq H^d_\frakm(R)^{\oplus \mu}$ as a graded $R$-module. Moreover, we learn that $\mu$ is equal to the $\CC$-dimension of the stalk at the closed point of $\dR(H^i_I(R))$.
\end{example}

\subsection{Bass numbers}
\label{ss:BassChar0}

Next, we want to recover \cite[Theorem~3.1]{LSW} via RHC, which gives a topological interpretation of the number $\mu$ appearing in Proposition~\ref{prop:BassGradedRH} in certain situations arising from local cohomology (this number $\mu$ is the Bass number of $M=H^d_\frakm(R)^{\oplus\mu}$; recall that more generally, the Bass number of a module $M$ supported at the maximal ideal $\frakm$ is by definition the dimension of the socle of $M$). For this, recall that if a constructible sheaf has finite support, then its global sections are identified with the direct sum of all its stalks. The next lemma provides a generalization of this fact to complexes, except that only some of the (perverse) cohomology sheaves are required to have finite support:

\begin{lemma}
\label{lem:LocalToGlobalFiniteSupport}
Assume that $X$ is a smooth affine connected variety over $\CC$. Fix a complex $K^\bullet \in D^b_c(X^\an)$ such that ${ }^p \tau^{> 0} K^\bullet$ is supported at finitely many points. Then, for $i > 0$, one has
\[
H^i(R\Gamma(X, K^\bullet)) \simeq \bigoplus_{x \in X} H^i(X,K^\bullet_x)\simeq \bigoplus_{x\in X}{}^pH^i(K^\bullet_x).
\]
\end{lemma}

\begin{proof}
We have a canonical exact triangle
\[
{ }^p \tau^{\le 0} K^\bullet \to K^\bullet \to { }^p \tau^{> 0} K^\bullet.
\]
By definition of perversity, the stalks of ${ }^p \tau^{\le 0} K^\bullet \in { }^p D^{\le 0}_c(X^\an)$ lie in $D^{\le 0}_c$. Hence the triangle implies that for $i > 0$ the stalks of~$H^i(K^\bullet)$ and $H^i({ }^p \tau^{> 0} K^\bullet)$ are isomorphic. As $X$ is affine, the Artin vanishing theorem \cite[Theorem~4.1.1]{BBDG} shows that $R\Gamma(X,-)$ carries ${ }^p D^{\le 0}_c(X^\an)$ to $D^{\le 0}(\CC)$. Applying this to the previous triangle then shows that
\[
H^i(R\Gamma(X, K^\bullet)) \simeq H^i(R\Gamma(X, { }^p \tau^{> 0} K^\bullet))
\]
for $i > 0$. Since ${ }^p \tau^{> 0} K^\bullet$ is supported at finitely many points, applying $R\Gamma(X,-)$ is the same as taking the direct sum over all the stalks, which gives the claim.
\end{proof}

Using the preceding formalism, we recover the promised theorem:

\begin{theorem}[{\cite[Theorem~3.1]{LSW}}]
Let $X = \Spec R$ be a smooth affine connected variety of dimension~$d$ over $\CC$. Let $I \subset R$ be an ideal, and let $\frakm$ be the maximal ideal corresponding to a closed point $x \in X$. Assume that there exists a positive integer $k_0$ such that $H^k_I(R)$ vanishes away from $x$ for $k \ge k_0$. Then:
\begin{enumerate}[\quad\rm(1)]
\item If $k_0 > 1$, then $H^k_I(R) \simeq H^d_{\frakm}(R)^{\oplus \mu}$ for $\mu \colonequals \rank \Hsing^{d+k-1}(U\,;\,\CC)$, where $U$ is the complement of the vanishing locus of $I$.

\item If $\Hsing^d(X\,;\,\CC_X) = 0$, then the previous conclusion is also valid for $k_0 = 1$.
\end{enumerate}
\end{theorem}

\begin{proof}
By Proposition~\ref{prop:BassGradedRH}, for $k$ as above, we have $H^k_I(R) \simeq H^d_{\frakm}(R)^{\oplus \mu}$, where $\mu$ is the rank of $\dR(H^k_I(R))_x$. To determine $\mu$, set $Z = \Spec(R/I)$, and let $Z \stackrel{i}{\into} X \stackrel{j}{\gets} U$ be the resulting decomposition of $X$ into closed and open sets. Now consider $K^\bullet = \dR(R\underline\Gamma_Z(\calO_X) )$, so $K^\bullet \simeq i_* i^! \CC_X[d]$ by Example~\ref{ex:LocalCohPerverse}. In particular, ${ }^p \calH^k K^\bullet \simeq \dR H^k_I(R)$. Since by assumption $H^k_I(R)$ vanishes away from $x$ for $k \ge k_0$ the same holds for its Riemann-Hilbert dual $\dR H^k_I(R) = { }^p \calH^k(K^\bullet)$. Applying Lemma~\ref{lem:LocalToGlobalFiniteSupport} (shifted by $k_0$) then shows that for $k \ge k_0$ one has
\[
\dR(H^k_I(R))_x \simeq { }^p \calH^k(K^\bullet)_x \simeq H^k(X, K^\bullet).
\]
It remains to identify $H^k(X, K^\bullet)$ with $\Hsing^{d+k-1}(U\,;\,\CC)$, which is just the sheaf cohomology of the constant sheaf $\CC_U$. For this, consider the exact triangle
\[
K^\bullet = i_* i^! \CC_X[d] \to \CC_X[d] \to j_* \CC_U[d].
\]
As $k_0$ is positive, we have $H^k(X, \CC_X[d]) = 0$ for $k \ge k_0$ by Artin vanishing \cite[Corollary~4.1.4]{BBDG}. Now both (1) and (2) follow from the long exact sequence of cohomology.
\end{proof}

\subsection{Hodge-to-de-Rham spectral sequence for de~Rham homology of affine schemes}
\label{ss:EmbInddRH}

We use RHC to reinterpret the Hodge-to-de-Rham spectral sequence for algebraic de~Rham homology of a singular affine space $Y$ embedded in a smooth affine space $X$; this interpretation shows that, from $E_2$ onward, the spectral sequence has finite dimensional terms and is independent of the embedding (up to a shift), recovering a recent result of Bridgland \cite{Bridgland} (we refer the interested reader to \cite{SwitaladR} for analogous results over complete local rings). In fact, we obtain a slightly better result: the spectral sequence depends only on the topological space $Y^\an$ (up to a shift that depends only on the dimension of $X$).

Let $X$ be a smooth affine algebraic variety over $\CC$ of dimension $n$, and let $i\colon Y \into X$ be a closed algebraic subscheme of dimension $d$. Following \cite[\S II.3]{Hartshorne:AlgdR}, one defines $H^{\dR}_\bullet(Y)$, the algebraic de~Rham homology of $Y$, as the cohomology of
\[
R\Gamma_{Y,\dR}(X) \colonequals R\Gamma_Y(X, \Omega^\bullet_X) ,
\]
namely, $H^{\dR}_t(Y)=H^{2n-t}(R\Gamma_Y(X, \Omega^\bullet_X))$.
The resulting spectral sequence 
\begin{equation}
\label{eq:HdRss}
E_1^{a,b}(Y \into X)\colon H^b_Y(X, \Omega^a_X) \Rightarrow H^{a+b}_{Y,\dR}(X) = H^{\dR}_{2n-(a+b)}(Y)
\end{equation}
of the hyper-derived functor $R\Gamma_Y(X,-)$ is known as the Hodge-to-de-Rham spectral sequence.

The $E_1^{a,b}$-term above evidently depends on the embedding $i\colon Y \into X$ and is typically infinite dimensional. Nevertheless, these features disappear after turning a page by a recent theorem of Bridgland \cite{Bridgland}. For a more precise formulation, let $E_2^{a,b}(Y \into X)$ denote the spectral sequence obtained from~\eqref{eq:HdRss} by turning the page once. Then Bridgland shows:

\begin{theorem}[\cite{Bridgland}]
The terms $E_2^{a,b}(Y \into X)$ are finite dimensional for all $a,b$. Moreover, up to bidegree shift, the spectral sequence $E_2^{a,b}(Y \into X)$ depends only on $Y$ and not the choice of embedding $i\colon Y \into X$.
\end{theorem}

We will prove this using RHC. In fact, we can ``explain'' the embedding independence purely in topological terms. For this, recall that for every complex $K^\bullet \in D^b_c(Y^\an)$, there is a canonical $E_2$-spectral sequence
\[
P_2^{a,b}(K^\bullet)\colon H^a(Y^\an, { }^p \calH^b(K^\bullet)) \Rightarrow H^{a+b}(Y^\an, K^\bullet)
\]
associated to the filtration of $K^\bullet$ by its perverse cohomology sheaves. The preceding theorem is then a consequence of the following:

\begin{theorem}
\label{thm:E_2independenceHdR}
The spectral sequence $E_2^{a,b}(Y \into X)$ coincides with the spectral sequence $P_2^{a,b}(K^\bullet)$ where $K^\bullet \colonequals \DD(\CC_Y)[-2n]$. In particular, both spectral sequences have finite dimensional terms, and are independent of the embedding $i\colon Y \into X$, up to a degree shift.
\end{theorem}

\begin{proof}
Let $\tilde E_2^{a,b}(Y^{\an}\into X^{\an})$ be the corresponding spectral sequence in the analytic category, i.e., one obtained from the standard spectral sequence with the $E_1$ term $H^b_Y(X^{\an},\Omega^a_{X^{\an}})$ and abutment $H^{a+b}_{Y,\dR}(X^{\an})$ by turning one page. The standard map $\Omega_X\to \Omega_{X^{\an}}$ induces a morphism of complexes $\phi\colon \Omega_X^\bullet\to \Omega^\bullet_{X^{\an}}$ which in turn induces a morphism of spectral sequences, i.e., morphisms
\[
\phi_r\colon E_r^{a,b}\to \tilde E_r^{a,b}
\]
for every $r$. We claim that $\phi_r$ is an isomorphism for every $r\ge 2$. Since $\phi_r$ being an isomorphism implies that $\phi_{r+1}$ also is an isomorphism, it suffices to prove that $\phi_2$ is an isomorphism.

The group $E_2^{a,b}$ (resp. $\tilde E_2^{a,b})$ is the $a$-th cohomology of the complex of abelian groups $H_Y^b(X,\Omega_X^\bullet)$ (resp. $H_{Y^{\an}}^b(X^{\an},\Omega_{X^{\an}}^\bullet)$). Thus, all we need to show is that the map 
\[
\phi'\colon H_Y^b(X,\Omega_X^\bullet)\to H_{Y^{\an}}^b(X^{\an},\Omega_{X^{\an}}^\bullet)
\]
induced by the map of complexes $\phi\colon \Omega_X^\bullet\to \Omega^\bullet_{X^{\an}}$ is an isomorphism. 

Since $X$ is affine and all $\Omega^a_X$ are quasi-coherent,
\[
H^b_Y(X,\Omega^a)=\Gamma(X,\calH^b_Y(\Omega^a))
\]
(resp. $H^b_{Y^{\an}}(X^{\an},\Omega_{X^{\an}}^a)=\Gamma(X^{\an},\calH^b_{Y^{\an}}(\Omega_{X^{\an}}^a))$. But
\[
\Gamma(X, -)=\Hom_{\calD_X}(\calD_X, -)
\]
(resp. $\Gamma(X^{\an}, -)=\Hom_{\calD_{X^{\an}}}(\calD_{X^{\an}}, -)$). Since $X$ is affine and all involved sheaves are quasi-coherent we have
\[
H^a(\Gamma(X, \calH^b_Y(\Omega^\bullet_X)))=H^a({\Hom}_{\calD_X}(\calD_X, \calH^b_Y(\Omega^\bullet_X)))=\Ext^a_{\calD_X}(\calD_X,\calH^b_Y(\Omega^\bullet_X))\simeq \Hom_{D(\calD_X)}(\calD_X[-a], \calH^b_Y(\Omega^\bullet_X)),
\]
where the last isomorphism holds by \cite[Tag~06XP]{stacks-project}. Similarly,
\[
H^a(\Gamma(X^{\an}, \calH^b_{Y^{\an}}(\Omega^\bullet_{X^{\an}})))={\Hom}_{D(\calD_{X^{\an}})}(\calD_{X^{\an}}[-a],\calH^b_{Y^{\an}}(\Omega^\bullet_X)).
\]
But the map
\[
\Hom_{D(\calD_X)}(\calD_X[-a], \calH^b_Y(\Omega^\bullet_X))\to \Hom_{D(\calD_{X^{\an}})}(\calD_{X^{\an}}[-a],\calH^b_{Y^{\an}}(\Omega^\bullet_{X^{\an}}))
\]
is an isomorphism by a result of Brylinski \cite[Proposition~7.8]{Brylinski} stating that analytification induces an equivalence of the respective $\calD$-module categories. Hence, the map
\[
\phi'\colon H_Y^b(X,\Omega_X^\bullet)\to H_{Y^{\an}}^b(X^{\an},\Omega_{X^{\an}}^\bullet)
\]
is indeed a quasi-isomorphism, and the spectral sequences $E_2^{a,b}(Y \into X)$ and $\tilde E_2^{a,b}(Y^{\an}\into X^{\an})$ are indeed isomorphic.

It remains to show that the analytic spectral sequence $\tilde E_2^{a,b}(Y^{\an}\into X^{\an})$ coincides with the spectral sequence $P^{a,b}_2(K^\bullet),$ where $K^\bullet \colonequals \DD(\CC_Y)[-2n]$. 

For each regular holonomic $M^\bullet \in D^b_{rh}(\calD_X)$, using the canonical filtration of $M^\bullet$ by its cohomology sheaves (via the standard $t$-structure), there is an $E_2$-spectral sequence
\[
F_2^{a,b}(M^\bullet)\colon H^a_{\dR}(X, \calH^b(M^\bullet)) \Rightarrow H^{a+b}_{\dR}(X,M^\bullet),
\]
where $H^*_{\dR}(X,-)=H^*(X,{\dR}(-))$. Concretely, $\dR(M^\bullet)$ can be computed as the total complex of the bicomplex $\Omega_{X^{\an}}^\bullet \otimes_{\calD_{X^{\an}}} (M^\bullet)^{\an}$. Hence $F_2^{a,b}(M^\bullet)$ is the $E_2$-spectral sequence defined by the column filtration on the bicomplex $(a,b) \mapsto \Omega^a_{X^{\an}} \otimes_{\calD_{X^{\an}}} (M^b)^{\an}$. This description also shows that the $E_1$-spectral sequence for the same filtration is the Hodge-to-de-Rham spectral sequence
\[
F_1^{a,b}(M^\bullet)\colon \Omega^a_{X^{\an}}\otimes \calH^b((M^\bullet)^{\an}) \Rightarrow H^{a+b}_{\dR}(X,M^\bullet).
\]
Applying this to $M^\bullet \colonequals R\underline\Gamma_Y(\calO_X) \in D^b_{rh}(\calD_X)$ shows that
\[
F_1^{a,b}(R\underline\Gamma_Y(\calO_X))= \Omega_{X^{\an}}^a \otimes H^b_{Y^{\an}}(X^{\an},\calO_{X^{\an}}) \simeq H^b_{Y^{\an}}(X^{\an},\Omega_{X^{\an}}^a) = \tilde E_1^{a,b}(Y^{\an} \into X^{\an})
\]
and hence the spectral sequence $\tilde E_2^{a,b}(Y^{\an} \into X^{\an})$ coincides with the spectral sequence $F_2^{a,b}(R\underline\Gamma_Y(\calO_X))$. 

On the other hand, the construction of the spectral sequence $F_2^{a,b}(M^\bullet)$ only depends on the knowledge of $M^\bullet \in D^b_{rh}(\calD_X)$ together with the $t$-structure on the latter category. Passing through RHC, and using Theorem~\ref{thm:RH}.\ref{thm:RH:3} and the equality $\dR(R\underline\Gamma_Y(\calO_X)) = i_*i^!\CC_X[n]$~\eqref{eq:RHLocalCoh1} then shows that $F_2^{a,b}(M^\bullet)$ is isomorphic to the spectral sequence
\[
P_2^{a,b}(\dR(R\underline\Gamma_Y \calO_X)[-n])\colon H^a(X^{\an}, { }^p \calH^b(i_*i^!\CC_X)) \Rightarrow H^{a+b}(X^{\an},i_*i^!\CC_X).
\]
This holds since the spectral sequence only depends on the perverse filtration on $i_*i^!\CC_X[n]$ which corresponds via $\dR$ to the usual filtration on $R\underline\Gamma_Y \calO_X$. 

By the exactness of $i_*$ this simplifies to the spectral sequence $P_2^{a,b}(i^! \CC_X)$ on $Y$. Then the computation 
\[
i^! \CC_X \simeq i^! \DD(\CC_X)[-2n] \simeq \DD(i^*\CC_X)[-2n] \simeq \DD(\CC_Y)[-2n]
\]
concludes the first part. The second part is then simply a consequence of the fact that the dualizing complex $\DD(\CC_Y)$ is constructible and intrinsic to $Y$.
\end{proof}

\begin{remark}
Theorem~\ref{thm:E_2independenceHdR} implies that the Hodge filtration on $H^i_{Y, \dR}(X)$ is determined by the topology of the affine variety $Y$: indeed, it is isomorphic to the filtration on $H^{i-2n}(Y, \DD(\CC_Y))$ induced by the spectral sequence $P_2^{a,b}(\DD(\CC_Y)[-2n])$ above, and the latter, being determined by the perverse filtration on $\DD(\CC_Y)[-2n]$, only depends on the underlying topological space of $Y$. This is in stark contrast with the global cohomological situation: if $X$ is a smooth and proper variety over $\CC$, the Hodge filtration on $H^i_{\dR}(X)$ is \emph{not} determined by the topology of $X$. Indeed, Campana found examples of homeomorphic (and even diffeomorphic) smooth projective surfaces over $\CC$ with different Hodge numbers; see \cite{Kuperberg}.
\end{remark}

\subsection{Embedding independence for Lyubeznik numbers}
\label{ss:EmbIndepLyuChar0}

\begin{notation}
Let $Z \into \PP^n$ be a projective variety of dimension $d\ge 1$ over $\CC$, and let $i\colon Y \subset X \colonequals \AA^{n+1}$ be the affine cone on $Z$ with vertex $\{0\}$; write $k\colon \{0\} \into Y$ for the inclusion of the vertex, and let $j\colon U \into Y$ be the open complement. The Lyubeznik numbers in this situation are defined as
\[
\lambda_{i,j} \colonequals \ell(\calH^i_{\{0\}}( \calH^{n+1-j}_Y(\calO_X))),
\]
where $\ell(-)$ denotes the length of the displayed $\calD_X$-module; note that this $\calD_X$-module is regular holonomic and supported at $\{0\}$, so it must have finite length, which is the same as the length of its socle. This holds since any holonomic $\calD_X$-module which is supported at a point, is isomorphic to a finite direct sum of copies of the hull of the residue field a the point (easily checked, for example via RHC).
\end{notation}

It was expected in \cite[p.~133]{Lyubeznik:survey} that the $\lambda_{i,j}$ depend only on $Z$, and not the embedding $Z \into \PP^n$. However, recently there have been counterexamples to this in characteristic 0 by Reichelt, Saito, and Walther \cite{RSW1} for reducible $Z$ and also in the irreducible case by Wang \cite{Wang}. Here we recover a positive result by Switala \cite{Switala} which shows the embedding independence in the case that $Z$ is a smooth complex projective variety:

\begin{theorem}
\label{thm:ICLN}
If $Z$ is smooth, then the $\lambda_{i,j}$ defined above are independent of the embedding $Z \into \PP^n$.
\end{theorem}

\begin{remark}
More generally, the proof of Theorem~\ref{thm:ICLN} given below applies as long as $Z$ is an intersection cohomology manifold, i.e., we have $\CC_Z[d] \simeq \mathrm{IC}_Z$. This condition is satisfied, for example, by varieties with quotient singularities.
\end{remark}

We begin by reinterpreting $\lambda_{i,j}$ via RHC, see also \cite{BlickleBondu, GarciaSabbah}.

\begin{lemma}
\label{LambdasPervCohom}
One has
\[
\lambda_{i,j} = \ell\Big( \calH^{-i} \big( k^* ({ }^p \calH^j(\CC_Y))\big)\Big).
\]
\end{lemma}

\begin{proof}
Moving the definition of $\lambda_{i,j}$ through RHC as in~\eqref{eq:RHLocalCoh2}, we know that $\lambda_{i,j}$ is the length of
\[
\calH^i(k^!({ }^p \calH^{n+1-j}(i^! \CC_X[n+1]))).
\]
Applying duality, and using $\DD(\CC_X[n+1]) \simeq \CC_X[n+1]$, as $X$ is smooth of dimension $n+1$, and $i^* \CC_X \simeq \CC_Y$, this is also the length of
\[
\calH^{-i}(k^*({ }^p \calH^j(\CC_Y))),
\]
as wanted.
\end{proof}

Thus, we must calculate the perverse cohomology sheaves of $\CC_Y$. For this, it will be quite convenient to use the intersection cohomology complex as an intermediary. For $Y$ the affine cone over a \emph{smooth} $d$-dimensional projective variety $Z$ as above, this can be identified as
\begin{equation}
\label{eq:DeligneForm}
\mathrm{IC}_Y \simeq \tau^{\le -1}(j_* \CC_U[d+1]).
\end{equation}
This is a special case of Deligne's formula for intermediate extensions in the case of isolated singularity, see \cite[Proposition~2.1.11]{BBDG}, \cite[Proposition~8.2.11]{HTTDmodules}, and \cite[Claim~12.8,~page 47]{Bhatt}; we recall the statement and proof of this special case for convenience:

\begin{lemma}[Deligne's formula for intermediate extensions]
\label{DeligneFormula}
Let $W$ be an $s$-dimensional variety over $\CC$. Fix a point $i\colon \{x\} \to W$ with open complement $j\colon V=W-\{x\} \to W$. For any perverse sheaf $M$ on $V$ which satisfies $\tau^{\ge 0} M = 0$, there is a natural identification
\[
j_{!*}(M) \simeq \tau^{\le -1} j_* M.
\]
\end{lemma}

Instead of the assumption $\tau^{\ge 0} M = 0$ in Deligne's formula, we could say that the formula holds in a neighborhood of $x$. In the special case where $V$ is smooth and $M=\CC_V[s]$, this gives the promised formula
\[
\mathrm{IC}_W \simeq \tau^{\le -1} (j_* \CC_V[s]).
\]

\begin{proof}
Let $K= \tau^{\le -1} j_* M$. Since $K|_V \simeq M$ is already perverse, it suffices (see \cite[Proposition~8.2.5]{HTTDmodules}) to show the following: (1) $i^* K \in D^{\le -1}$, and (2) $i^! K \in D^{\ge 1}$.

Before proceeding, we remark that weaker versions of (1) and (2) --- merely requiring $i^* K \in D^{\le 0}$ and $i^! K \in D^{\ge 0}$ --- already ensure that $K$ is perverse by the definition of middle perversity in terms of stalks and costalks (see \cite[\S 4.0]{BBDG} or \cite[Proposition~8.1.22]{HTTDmodules}). Part (1) is immediate since $K \in D^{\le -1}$ and $i^*$ is exact for the standard $t$-structure by \cite[\S4.5]{HTTDmodules}. For (2) we shall use the defining triangle
\[
K \to j_* M \to Q \colonequals \tau^{\ge 0} j_* M.
\]
Applying $i^!$ and noting that $i^! j_* = 0$, we learn that $i^! K \simeq i^! Q[-1]$. Now by our assumption on $M$ we see that $Q$ is supported on $\{x\}$, so $Q = i_* i^* Q$ with $i^* Q \in D^{\ge 0}$. Since $i^! i_* = \mathrm{id}$, we conclude that $i^! K = i^* Q[-1]$, which lies in $D^{\ge 1}$ as wanted.
\end{proof}

To proceed further, we need the stalks at $0$ of $\mathrm{IC}_Y$. Thanks to the previous formula, these are calculated in terms of the cohomology of $U$ as follows:

\begin{lemma}
\label{lem:StalkIC}
One has
\[
k^* \mathrm{IC}_Y \simeq \bigoplus_{i \le -1} H^{d+1+i}(U, \CC)[-i].
\]
By Hard Lefschetz on $Z$, this simplifies to
\[
k^* \mathrm{IC}_Y \simeq \bigoplus_{i \le -1} \coker\big(H^{d-1+i}(Z, \CC) \xrightarrow{\ c_1\ } H^{d+1+i}(Z,\CC)\big)[-i],
\]
where $c_1 \in H^2(Z, \CC)$ is the first Chern class of $\calO_X(1)$. In particular,
\[
\ell\big(H^{-j} k^* \mathrm{IC}_Y\big) = \beta_{d+1-j} - \beta_{d-1-j} \qquad \text{for } j \le -1,
\]
and $0$ otherwise, where $\beta_i=\dim H^i(Z,\CC)$ denote the Betti numbers of $Z$.
\end{lemma}

\begin{proof}
The first identification comes directly from Deligne's formula~\eqref{eq:DeligneForm} and the observation that for any sufficiently small ball $\overline{V} \subset Y$ about $0$, the ``Milnor'' ball $\overline{V} - \{0\} = \overline{V} \cap U$ is homotopy-equivalent to $U$. In fact, either of these is homotopy-equivalent to the $S^1$-bundle over $Z$ defined by the $\GG_m$-torsor $\calO_X(1)$. In particular, this gives
\[
k^* \mathrm{IC}_Y \simeq \tau^{\le -1} R\Gamma(U, \CC[d+1]) \simeq \bigoplus_{i \le -1} H^{d+1+i}(U, \CC)[-i],
\]
as asserted. To simplify further, let $\pi\colon U \to Z$ be the projection. By calculating locally on $Z$, we see that there is an exact triangle
\[
\CC_Z \to R\pi_* \CC_U \to \CC_Z[-1]
\]
in $D^b_c(X^{\an})$. Moreover, the boundary map for this triangle is identified (up to a unit) with multiplication by the first Chern class $c_1$ of the ample line bundle $\calO(1)|_Z$: indeed, this can be checked by reducing to the universal case of projective space\footnote{We are using the following base change compatibility (applied to the closed immersion of $Z$ in projective space): if $f\colon V \to P$ is a $\GG_m$-torsor and $i\colon Z \to P$ is any morphism of varieties, then, writing $g\colon U \to Z$ for the base change of $f$, the natural base change map $\alpha\colon i^* Rf_* \CC_V \to Rf_* \CC_U$ is an isomorphism. To prove this, since the map is globally defined, we may assume $f$ is the trivial torsor, i.e., $V = P \times \GG_m$. But then we might as well reduce to the universal case where $P=\Spec\CC$ is a point. In this case, the claim that $R\Gamma(Z,\alpha)$ is an isomorphism follows from the K\"unneth formula; as this also holds true after replacing $Z$ with an open subset, $\alpha$ must be an isomorphism.}, and noting that the boundary map is nonzero (as $H^2(U,\CC) = 0$ in the projective space case) and must thus be a nonzero multiple of $c_1$ since $\Hom_{\PP^n}(\CC_{\PP^n}[-1], \CC_{\PP^n}[1]) = H^2(\PP^n, \CC)$ is a $1$-dimensional vector space. The long exact cohomology sequence of the above triangle gives the exact sequence
\[
0 \to \coker\big(H^{j-2}(Z, \CC) \xrightarrow{\ c_1\ } H^j(Z,\CC)\big)\to H^j(U,\CC) \to \ker\big(H^{j-1}(Z,\CC) \xrightarrow{\ c_1\ } H^{j+1}(Z,\CC)\big) \to 0
\]
for any $j$. Now for $j \le d$, the last term above vanishes by Hard Lefschetz. This then gives the second formula for $k^* \mathrm{IC}_Y$ in the lemma. Finally, the assertion about lengths again follows from Hard Lefschetz.
\end{proof}

Using the preceding lemma, one can completely calculate all the perverse cohomology sheaves of $\CC_Y$, and thus each $\lambda_{i,j}$, explicitly in terms of $Z$. The bookkeeping of indices is a bit messy, but recorded below.

\begin{proof}[Proof of Theorem~\ref{thm:ICLN}]
By Deligne's formula~\eqref{eq:DeligneForm}, the canonical map $\CC_Y \to j_* \CC_U$ induces a map $\CC_Y[d+1] \to \mathrm{IC}_Y$ which is an isomorphism outside $\{0\} \subset Y$, as well as on $\tau^{\le -d-1}$. Since $\CC_Y[d+1]$ itself lies in $D^{\le -d-1}$, this gives an exact triangle
\[
K \to \CC_Y[d+1] \to \mathrm{IC}_Y
\]
with $K[1] \simeq k_* k^* K[1] \simeq k_* k^* \tau^{> -d-1} \mathrm{IC}_Y$. Using Lemma~\ref{lem:StalkIC} and the exactness of $k^*$ for the standard $t$-structure \cite[\S4.5]{HTTDmodules}, this simplifies to
\[
K \simeq \big(k_* \tau^{> -d-1} k^* \mathrm{IC}_Y\big)[-1] \simeq \!\!\!\!\!\! \bigoplus_{-d-1 < i \le -1} \!\!\!\!\!\!\!\!
k_* \coker\big(H^{d-1+i}(Z, \CC) \xrightarrow{\ c_1\ } H^{d+1+i}(Z,\CC)\big)[-i-1].
\]
Note that all complexes in the preceding triangle lie in ${ }^p D^{\le 0}$, and that $\mathrm{IC}_Y$ is perverse. In particular, we have
\[
{ }^p \calH^{j+d+1}(\CC_Y) = { }^p \calH^j(\CC_Y[d+1]) = 0 \qquad \text{for } j > 0,
\]
and therefore that
\begin{equation}
\label{eq:LN1}
\lambda_{i,j+d+1} = 0 \quad \mathrm{ for } \quad j > 0.
\end{equation}
Taking perverse cohomology, the previous exact triangle gives a short exact sequence
\begin{equation}
\label{eq:perverseH0LN}
0 \to k_* \coker\big(H^{d-2}(Z, \CC) \xrightarrow{\ c_1\ } H^{d}(Z,\CC)\big) \to { }^p \calH^0(\CC_Y[d+1]) \to \mathrm{IC}_Y \to 0
\end{equation}
and isomorphisms
\begin{equation}
\label{eq:perverseHjLN}
{ }^p \calH^j(\CC_Y[d+1]) = { }^p \calH^j(K) = k_* \coker\big(H^{d+j-2}(Z, \CC) \xrightarrow{\ c_1\ } H^{d+j}(Z,\CC)\big)
\end{equation}
for $j \le -1$. Applying $k^*$ to~\eqref{eq:perverseHjLN} then shows that
\begin{equation}
\label{eq:perverseHjLNpullback}
k^*( { }^p \calH^{j+d+1}(\CC_Y)) \simeq k^* ({ }^p \calH^j(\CC_Y[d+1])) = \coker\big(H^{d+j-2}(Z, \CC) \xrightarrow{\ c_1\ } H^{d+j}(Z,\CC)\big)
\end{equation}
for $j \le -1$. This shows that
\begin{equation}
\label{eq:LN2}
\lambda_{0,j+d+1} = \ell\big(\coker\big(H^{d+j-2}(Z, \CC) \xrightarrow{\ c_1\ } H^{d+j}(Z,\CC)\big)\big) = \beta_{d+j} - \beta_{d+j-2}
\qquad \text{if } j \le -1,
\end{equation}
and
\begin{equation}
\label{eq:LN3}
\lambda_{i,j+d+1} = 0 \qquad \text{if } i \neq 0, \ j \le -1.
\end{equation}
On the other hand, applying $k^*$ to the sequence~\eqref{eq:perverseH0LN} gives an exact triangle
\[
\coker\big(H^{d-2}(Z, \CC) \xrightarrow{\ c_1\ } H^{d}(Z,\CC)\big) \to k^*({ }^p \calH^0(\CC_Y[d+1])) \to k^*(\mathrm{IC}_Y).
\]
Since the first term is already a vector space placed in degree $0$, this collapses to a $5$-term exact sequence
\begin{multline*}
0 \to H^{-1}(k^*({ }^p \calH^0(\CC_Y[d+1]))) \to H^{-1}(k^* \mathrm{IC}_Y) \xrightarrow{\ \delta\ } \coker\big(H^{d-2}(Z, \CC) \xrightarrow{\ c_1\ } H^{d}(Z,\CC)\big) \\
\to H^0(k^*({ }^p \calH^0(\CC_Y[d+1]))) \to H^0(k^* \mathrm{IC}_Y) \to 0
\end{multline*}
and isomorphisms
\[
H^j(k^* { }^p \calH^0(\CC_Y[d+1])) \simeq H^j(k^* \mathrm{IC}_Y) \qquad \text{for } j \le -2.
\]
One can check that $\delta$ is an isomorphism (essentially by construction, cf. Lemma~\ref{DeligneFormula}). Thus, we see that
\[
H^{-1}(k^* { }^p \calH^{d+1}(\CC_Y)) = H^{-1}(k^* {}^p \calH^0(\CC_Y[d+1])) = 0
\]
and that one has isomorphisms
\[
H^j(k^* { }^p \calH^{d+1}(\CC_Y)) \simeq H^j(k^* { }^p \calH^0(\CC_Y[d+1])) \simeq H^j(k^* \mathrm{IC}_Y)
\]
for $j \le -2$ and $j = 0$. Using the description of $\lambda_{i,j}$ in Lemma~\ref{LambdasPervCohom} and the calculation of $k^* \mathrm{IC}_Y$, this gives
\begin{equation}
\label{eq:LN4}
\lambda_{0,d+1} = \lambda_{1,d+1} = 0
\end{equation}
and
\begin{equation}
\label{eq:LN5}
\lambda_{i,d+1} = \ell\Big(\coker\big(H^{d-1-i}(Z, \CC) \xrightarrow{\ c_1\ } H^{d+1-i}(Z,\CC)\big)\Big) = \beta_{d+1-i} - \beta_{d-1-i}
\qquad \text{for } i \ge 2.
\end{equation}
Combining~\eqref{eq:LN1},~\eqref{eq:LN2},~\eqref{eq:LN3},~\eqref{eq:LN4}, and~\eqref{eq:LN5} then proves the theorem as the right hand side in each formula is embedding independent.
\end{proof}

\section{The Riemann-Hilbert functor in characteristic $p$ and perverse $\FFp$-sheaves}
\label{sec:RHPervFp}

The goal of this section is to recall and study the perverse $t$-structure on constructible $\FFp$-sheaves on a large class of schemes of characteristic $p$. In particular, in \S\ref{ss:PervCohFrob} we define the perverse $t$-structure on constructible $\FFp$-sheaves by transporting the perverse $t$-structure on coherent sheaves (recalled in \S\ref{ss:PervCoh}) across the Riemann-Hilbert correspondence from \cite{BhattLurieRH} (recalled in \S\ref{RecallModpRH}); this construction gives excellent control on the commutative algebra properties of the Riemann-Hilbert partners of perverse sheaves. In \S\ref{ss:GabberPervComp}, we show that our definition coincides with the standard one in terms of support conditions on the stalks and co-stalks as introduced by Gabber \cite{Gabbert} and further studied in \cite{EK, CassPerv}. As in the analytic case, we use the version of a characteristic $p$ Riemann-Hilbert correspondence in \cite{BhattLurieRH} (with precursors in \cite{Lyubeznik:Fmod, EK, BoPi}) to relate quasi-coherent sheaves with various Frobenius actions to perverse constructible $\FFp$-sheaves.

\subsection{Perverse coherent sheaves}
\label{ss:PervCoh}

We define and study a notion of perverse coherent sheaves obtained by applying Grothendieck duality to the usual notion of coherent sheaves. This notion will be useful later in our definition of the perverse $t$-structure on $\FFp$-\'etale sheaves and was previously also considered by Deligne (see for example \cite{AB}) and Gabber \cite{Gabbert}. We do not strive for the most general context to develop this theory here, but stick to a situation that we need in the applications of the following section.

\begin{notation}[Normalized dualizing complexes]
\label{NormalizedNotation}
Fix a noetherian local ring $(A,\frakm,k)$ admitting a dualizing complex; for example, $A$ could be any complete noetherian local ring. Choose a dualizing complex $\omega_A^\bullet$ normalized by the requirement that $R\Gamma_{\frakm}(\omega_A^\bullet) =E[0]$ lives in degree $0$; the module $E$ is identified with the injective hull of the residue field $k$ of $A$.

We shall work with $A$-schemes of finite type and use the duality theory from~\cite[\href{https://stacks.math.columbia.edu/tag/0AU3}{Tag~0AU3}]{stacks-project}. Thus, given a finite type map $f\colon X \to \Spec A$, we write
\[
\omega_X^\bullet \colonequals f^! \omega_A^\bullet \in D^b_{coh}(X)
\] for the normalized dualizing complex, and
\[
\DD_X(-) \colonequals \RHom_X(-, \omega_X^\bullet)
\]
for the resulting Grothendieck duality equivalence $D^b_{coh}(X) \simeq D^b_{coh}(X)^{op}$.
\end{notation}

This notion is compatible with various operations.

\begin{remark}[Reduction to the regular case]
We shall often be in a setting where the ring $A$ is the quotient of a regular local ring $R$; for instance, this always holds true when $A$ is $\frakm$-adically complete, or if $A$ is $F$-finite by \cite{Gabbert}. In this case, there is an isomorphism $\omega_A^\bullet \simeq \RHom_R(A, \omega_R^\bullet)$, where
\[
\omega_R^\bullet \colonequals \bigwedge^{\dim R}\Omega_R[\dim R]
\]
with $\Omega_R$ the module of K\"ahler differentials, is a normalized dualizing complex on $R$. As any finite type $A$-scheme is also a finite type $R$-scheme, this allows one to often reduce to the case where $A = R$ is regular.
\end{remark}

\begin{remark}[Compatibility between local and global duality]
\label{DCNormalizedFinite}
For finite $A$-algebras $B$, the dualizing complex $\omega_B^\bullet \simeq \RHom_A(B,\omega_A^\bullet)$ is normalized; more precisely, the $B$-complex $R\Gamma_{\frakm}(\omega_B^\bullet)$ is concentrated in degree~$0$, and identifies with $E_B \simeq \Hom_A(B,E)$, which is an injective hull of the residue field of $B$. Under this identification, we have the following compatibility of global and local duality: for any $M \in D^b_{coh}(B)$, the complex $R\Gamma_{\frakm}(\DD_B(M))$ identifies with the Matlis dual $\RHom_B(M,E_B)$ of $M$.
\end{remark}

\begin{remark}[Compatibility with localization and the $d(-)$ function]
\label{DualizingComplexLocalRing}
Let $X$ be a finite type $A$-scheme. For any $x \in X$, the stalk $\omega^\bullet_{X,x}$ of $\omega^\bullet_X$ is a dualizing complex over the local ring $\calO_{X,x}$. Hence, the complex
\[
E_x \colonequals R\Gamma_{\{x\}}(\omega^\bullet_{X,x})[-d(x)]
\]
is concentrated in degree $0$ for a fixed integer $d(x) \in \ZZ$. As the formation of normalized dualizing complexes is compatible with passage to open subsets, the integer $d(x)$ can be calculated after replacing $X$ by any open subset containing $x$. The function $x \mapsto d(x)$ has good properties:

(1) For $X \colonequals \Spec B$ with $B$ finite over $A$, one has $d(x) = 0$ for any closed point $x \in \Spec B$. More generally, for any finite type $X/A$, if $i\colon \{x\} \to X$ is the inclusion of a closed point lying over the closed point of~$A$ (inducing a necessarily finite residue field extension $\kappa(x)/k$), then $d(x)=0$: indeed, we have
\begin{multline*}
\Hom_{\calO_{X,x}}(\kappa(x), E_x[d(x)]) = \RHom_{\calO_{X,x}}(\kappa(x), R\Gamma_x(\omega^\bullet_{X,x}))
= \RHom_X(i_* \calO_{\{x\}}, \omega^\bullet_X)\\
\simeq \RHom_A(\kappa(x), \omega^\bullet_A) \simeq \Hom_A(\kappa(x), E),
\end{multline*}
where the second equality is due to the fact that the first entry in the $\Hom$ is supported at $x$ and $\Hom$ and the dualizing complex behave well with respect to localization. The third isomorphism uses Grothendieck duality for $f\colon X \to \Spec A$ with respect to coherent sheaves on $X$ with proper support over $A$.

(2) If $x \rightsquigarrow y$ is an immediate specialization of points in $X$, then $d(x) = d(y)+1$: this is a general feature of dualizing complexes, see~\cite[\href{https://stacks.math.columbia.edu/tag/0A7Z}{Tag~0A7Z}]{stacks-project}.

(3) If $S \subset X$ is a locally closed subset, then we set
\[
d(S) = \max\{d(x) \mid x \in S\}.
\]
It follows from the observation in (2) that $d(S) = d(\overline{S})$ as each point of $\overline{S}$ is a specialization of a point of~$S$.

It follows from these properties that if $X$ is proper over $A$, then $d(X) = d(\eta)$ for some generic point~$\eta \in X$, and that $d(\eta) = \dim(\overline{\{\eta\}}) = \dim X$. In general, if $X$ admits a dense open immersion $j\colon X \into \overline{X}$ into a proper $A$-scheme, then $d(X) = \dim \overline{X}$.
\end{remark}

Using duality, one can define a perverse $t$-structure on the derived category of coherent sheaves:

\begin{definition}[Perverse coherent sheaves of $A$-schemes]
\label{DefPervCoh}
Let $X$ be a finite type $A$-scheme. Write
\[
\Perv_{coh}(X) \subset D^b_{coh}(X)
\]
for the essential image (i.e., the image closed under the isomorphisms) of $\mathrm{Coh}(X)$ under $\DD_X$; we refer to objects $M \in \Perv_{coh}(X)$ as \emph{perverse coherent sheaves} on $X$. More generally, applying Grothendieck duality $\DD_X(-)$ to the standard $t$-structure on $D^b_{coh}(X)$ induces the \emph{perverse $t$-structure} on $D^b_{coh}(X)$; the heart of the latter is the category $\Perv_{coh}(X)$ defined above. Thus, by construction, we have an equivalence $\mathrm{Coh}(X)^{op} \simeq \Perv_{coh}(X)$ given by the Grothendieck duality functor $\DD_X(-)$ in either direction.
\end{definition}

At least for $X \colonequals \Spec A$ itself, we shall later give a duality-free approach to the perverse $t$-structure (see Remark~\ref{PervCohLocalCoh}). The most basic example of a perverse sheaf is the following.

\begin{example}
\label{CMPerverseCoh}
The dualizing complex $\omega_X^\bullet \simeq \DD_X(\calO_X)$ is perverse. If $\omega_X^\bullet$ is concentrated in a single degree $d$ (e.g., if $X$ is Cohen-Macaulay), then $\calO_X[d]$ is perverse.
\end{example}

More examples arise by the following construction.

\begin{lemma}
\label{PervFinitePushforward}
If $f\colon X \to Y$ is a finite map of finite type $A$-schemes, then the pushforward functor $f_*$ carries perverse coherent sheaves to perverse coherent sheaves.
\end{lemma}

\begin{proof}
Fix $N \in \mathrm{Coh}(X)$ corresponding to $M \colonequals \DD_X(N) \in \Perv_{coh}(X)$. We must check that $f_* M \in \Perv_{coh}(Y)$ or equivalently that $\DD_Y(f_* M) \in D^b_{coh}(Y)$ actually lives in $\mathrm{Coh}(Y)$. But Grothendieck duality for $f$ shows that
\[
\DD_Y(f_* M) \simeq f_* \DD_X(M) \simeq f_* N,
\]
so the claim follows from the acyclicity of $f_*$ for finite maps.
\end{proof}

\begin{remark}
The definition of perverse coherent sheaves above works for any noetherian scheme equipped with a dualizing complex. The reason we insist on working relative to a fixed base scheme is that it permits us access to \emph{normalized} dualizing complexes via $!$-pullbacks, i.e., for a finite type map $f\colon X \to \Spec A$, we could define $\omega_X^\bullet \colonequals f^! \omega_A^\bullet$ in Notation~\ref{NormalizedNotation}. This normalization property is critical to ensure that the notion of perversity is preserved by finite pushforwards as in Lemma~\ref{PervFinitePushforward}; if we had not normalized our dualizing complexes, this would only be true up to shifts.
\end{remark}

Let us explain why Definition~\ref{DefPervCoh} naturally passes to localizations.

\begin{construction}[Perverse coherent sheaves on stalks]
\label{PervCohStalks}
It is convenient to extend Definition~\ref{DefPervCoh} to stalks. Thus, let $X$ be an $A$-scheme of finite type, and let $x \in X$. As explained in Remark~\ref{DualizingComplexLocalRing}, the complex $\omega^\bullet_{X,x}$ is a dualizing complex over the local ring $\calO_{X,x}$, and $R\Gamma_{\{x\}}(\omega^\bullet_{X,x})$ is concentrated in homological degree $d(x)$. In particular, we define \emph{perverse coherent sheaves} on $X_x \colonequals \Spec\calO_{X,x}$ as the image of $\mathrm{Coh}(X_x) \subset D^b_{coh}(X_x)$ under the duality functor
\[
\DD_{X_x}(-) \colonequals \RHom_{\calO_{X,x}}(-, \omega^\bullet_{X,x}).
\]
With this definition, the stalk functor $D^b_{coh}(X) \to D^b_{coh}(\calO_{X,x})$ is $t$-exact for both the standard and perverse $t$-structures. In particular, the perversity of an object in $D^b_{coh}(X)$ can be checked after passing to stalks. Moreover, for $M \in D^b_{coh}(\calO_{X,x})$, the following are equivalent:
\begin{enumerate}[\quad\rm(1)]
\item $M$ is perverse.

\item $R\Gamma_{\{x\}}(M)[-d(x)]$ is concentrated in degree $0$.
\end{enumerate}
Indeed, $(1)$ is equivalent to the Grothendieck dual $\DD_{X_x}(M)$ being concentrated in degree $0$. On the other hand, since $R\Gamma_{\{x\}}(\omega_{X,x}^\bullet)[-d(x)] \simeq E_x$ where $E_x$ is the injective hull of the residue field, $(2)$ above is equivalent to the Matlis dual of $M$ being concentrated in degree $0$. The equivalence of $(1)$ and $(2)$ now follows from the compatibility of Grothendieck and Matlis duality under local cohomology, as in Remark~\ref{DCNormalizedFinite}.

The same argument actually shows that $M \in\, ^p D^{\le 0}_{coh} (\calO_{X,x})$ if and only if $R\Gamma_{\{x\}}M[-d(x)] \in D^{\ge 0}_{coh}(\calO_{X,x})$, and likewise for $^p D^{\ge 0} (\calO_{X,x})$
\end{construction}

Using the previous construction, we have the following criterion for detecting perversity of a coherent complex in terms of the local cohomology of its stalks:

\begin{corollary}[Recognizing perversity via local cohomology]
\label{RecogPerv}
Fix a finite type $A$-scheme $X$ and some $M \in D^b_{coh}(X)$. Then the following are equivalent:
\begin{enumerate}[\quad\rm(1)]
\item $M \in D^b_{coh}(X)$ is perverse.

\item For each point $x \in X$, the stalk $M_x \in D^b_{coh}(\calO_{X,x})$ is perverse.

\item For each closed point $x \in X$, the stalk $M_x \in D^b_{coh}(\calO_{X,x})$ is perverse.

\item For each point $x \in X$, the $\calO_{X,x}$-complex $R\Gamma_{\{x\}}(M_x)$ is concentrated in homological degree $d(x)$.

\item For each closed point $x \in X$, the $\calO_{X,x}$-complex $R\Gamma_{\{x\}}(M_x)$ is concentrated in homological degree~$d(x)$.
\end{enumerate}
Analogous equivalences hold for the perverse $t$-structure itself, i.e., membership of $M$ in $^p D^{\le 0}_{coh} (X)$ (respectively $^p D^{\ge 0}_{coh} (X)$) can be checked on stalks at (closed) points and is equivalent to $R\Gamma_{\{x\}}M_x \in D^{\ge -d(x)}(\calO_{X,x})$ (respectively $D^{\le -d(x)}(\calO_{X,x})$) for all (closed) points $x \in X$.
\end{corollary}

\begin{proof}
The equivalence of $(1)$, $(2)$, and $(4)$ is clear from Construction~\ref{PervCohStalks}. For the rest, it is enough to observe that taking stalks is $t$-exact for the perverse $t$-structure and that $M \in D^b_{coh}(X)$ vanishes if and only if~$M_x = 0$ for each closed point $x$.
\end{proof}

\begin{remark}[Comparison with Gabber's approach]
\label{CompGabberCoh}
For a finite type $A$-scheme $X$, we can apply \cite[\S 2]{Gabbert} with the perversity function $p'(x) = -d(x)$ to obtain a perverse $t$-structure $p'$ on $D_{qc}(X)$, characterized by the requirement that $K \in D_{qc}(X)$ lies in ${}^{p'} D^{\ge 0}_{qc}(X)$ if and only if for each $x \in X$, we have $R\Gamma_x(K_x) \in D^{\ge p(x)}$. We claim that this $t$-structure extends the perverse one constructed in Definition~\ref{DefPervCoh} on $D^b_{coh}(X)$. Indeed, this follows from Gabber's results in \cite[\S 7]{Gabbert}. Alternately, one can argue directly using the material above as follows. First, using Corollary~\ref{RecogPerv} and filtering any object via its perverse cohomology sheaves, it follows that the coconnective objects of the perverse $t$-structure on $D^b_{coh}(X)$ studied in this section are coconnective in the $p'$-perverse $t$-structure on $D_{qc}(X)$. So it is enough to check that the same also holds true for connective objects; this follows again from Corollary~\ref{RecogPerv} since the local cohomological dimension of a finitely generated module over a noetherian ring equals its support dimension. 
\end{remark}

In particular, for $X = A$ itself, perversity of a coherent complex can be tested using a single local cohomology computation. This observation can be lifted to an equivalence of categories using the notion:

\begin{definition}[Cofinite modules]
An $A$-module $M$ is \emph{cofinite} if it is $\frakm^\infty$-torsion and $\Hom_A(k,M)$ is a finite dimensional $k$-vector space; write $\mathrm{Mod}_{\cof}(A)$ for the category of cofinite $A$-modules. Write $D^b_{\cof}(A) \subset D^b(A)$ for the full subcategory spanned by $K \in D^b(A)$ with $H^i(K) \in \mathrm{Mod}_{\cof}(A)$.
\end{definition}

Using Matlis duality, an $A$-module is cofinite precisely if it is artinian. In particular, $\mathrm{Mod}_{\cof}(A)$ is an abelian subcategory of $\mathrm{Mod}(A)$ closed under passage to subquotients and extensions. Under the completeness hypothesis, it turns out that the local cohomology functor identifies $\Perv_{coh}(A)$ with $\mathrm{Mod}_{\cof}(A)$.

\begin{proposition}
\label{PervCofDer}
Assume $A$ is $\frakm$-adically complete. The following triangulated categories are equivalent via $t$-exact functors:
\begin{enumerate}[\quad\rm(1)]
\item The opposite of $D^b_{coh}(A)$ equipped with the usual $t$-structure.
\item The category $D^b_{coh}(A)$ equipped with the perverse $t$-structure.
\item The category $D^b_{\cof}(A)$ equipped with the usual $t$-structure.
\end{enumerate}
The functor $\DD_A(-)$ relates $(1)$ and $(2)$ in both directions. The functors relating $(2)$ and $(3)$ are:
\[
M \in D^b_{coh}(A) \mapsto R\Gamma_{\frakm}(M), \qquad\text{and}\qquad N \in D^b_{\cof}(A) \mapsto \widehat{N} \colonequals R\lim (N \otimes_A^\LL A/\frakm^n),
\]
i.e., $\widehat{N}$ is the derived $\frakm$-adic completion.
The equivalence relating $(1)$ and $(3)$ obtained by composing the previous equivalences coincides with the one given by Matlis duality, i.e., $R\Gamma_{\frakm}(\DD_A(M))$ is the Matlis dual of $M$ for any $M \in D^b_{coh}(A)$.
\end{proposition}

\begin{proof}
The last statement follows from Remark~\ref{DCNormalizedFinite}. It is also clear from the definitions that $\DD_A(-)$ induces a $t$-exact equivalence between $(1)$ and $(2)$. For $(2)$ and $(3)$, recall that it is a general fact~\cite[\href{https://stacks.math.columbia.edu/tag/0A6X}{Tag~0A6X}]{stacks-project}, due to Dwyer-Greenlees-May, that the functors $M \mapsto R\Gamma_{\frakm}(M)$ and $N \mapsto \widehat{N}$ induce an equivalence between the full subcategories $D_{\frakm-comp}(A)$ and $D_{\frakm-nilp}(A)$ of $D(A)$ spanned by derived $\frakm$-complete and $\frakm^\infty$-torsion objects respectively. We shall check that these functors carry $\Perv_{coh}(A) \subset D_{\frakm-comp}(A)$ and $\mathrm{Mod}_{\cof}(A) \subset D_{\frakm-nilp}(A)$ to each other; this implies that they give a $t$-exact equivalence relating $(2)$ and $(3)$. By writing $A$ as a quotient of a regular local ring, we may assume $A$ is regular. The claim for $R\Gamma_{\frakm}(-)$ was already shown in Corollary~\ref{RecogPerv}. Conversely, if $N$ is a cofinite $A$-module, then Matlis duality implies that
\[
N \simeq \Hom_A(M,E) \simeq M^\vee \otimes^\LL_A E,
\]
where $M \simeq \Hom_A(N,E)$ is a finitely generated $A$-module and $M^\vee \colonequals \RHom(M,A) \in D(A)$ as before is the linear dual. As $R$ is regular and $M$ is finitely generated, it is quasi-isomorphic to a perfect $A$-complex. This allows to reduce the claimed isomorphism to the case $M=A$, where it is clear. Also, $M^\vee$ is a perfect $A$-complex and as tensoring with such commutes with completion, we get that
\[
\widehat{N} \simeq M^\vee \otimes^\LL_A \widehat{E} \simeq M^\vee \otimes^\LL_A \widehat{R\Gamma_A(\omega_A^\bullet)} \simeq M^\vee \otimes^\LL_A \omega_A^\bullet \simeq \DD_A(M)
\]
is a perverse coherent $A$-module. Note that the last equality can be reduced, as above, to the case $M = A$. 
\end{proof}

\begin{corollary}[Perverse coherent sheaves as cofinite modules]
\label{PervCof}
Assume $A$ is $\frakm$-adically complete. The functors in Proposition~\ref{PervCofDer} restrict to give equivalences of the following abelian categories:
\begin{enumerate}[\quad\rm(1)]
\item The opposite of the category $\mathrm{Mod}_{coh}(A)$ of finitely generated $A$-modules.

\item The category $\Perv_{coh}(A)$ of perverse coherent $A$-modules.

\item The category $\mathrm{Mod}_{\cof}(A)$ of cofinite $A$-modules.
\end{enumerate}
\end{corollary}

\begin{remark}[A formula for the perverse cohomology groups via derived completions]
\label{PervCohLocalCoh}
For future reference, we note the following consequence of Proposition~\ref{PervCofDer}, giving a duality-free formula for the perverse cohomology groups: when $A$ is $\frakm$-adically complete and $M \in D^b_{coh}(A)$, then ${}^p H^i(M)$ identifies with the derived $\frakm$-adic completion of $H^i_{\frakm}(M)$.
\end{remark}

\subsection{Recollection on the Riemann-Hilbert correspondence from \cite{BhattLurieRH}}
\label{RecallModpRH}

In this subsection, we recall some of the main results of \cite{BhattLurieRH} on the covariant Riemann-Hilbert correspondence for $\FFp$-\'etale sheaves on a scheme of characteristic $p$.

\begin{notation}
\label{FrobModNot}
Let $R$ be a ring of characteristic $p$. Write $R[F]$ for the associative ring freely generated by~$R$ and a formal symbol $F$ (for Frobenius) subject to the relation $a^p F = F a$ for any $a \in R$; thus, the category $\mathrm{Mod}(R[F])$ of left $R[F]$-modules is identified with the category of pairs $(M,\phi_M)$, where $M$ is an $R$-module, and $\phi_M\colon M \to F_* M$ is an $R$-linear map; we call such objects \emph{Frobenius modules} on $R$ and sometimes drop the Frobenius map $\phi_M$ from the notation. Write $D(R[F])$ for the derived category of Frobenius modules; tensor products over $R$ induces a natural symmetric monoidal structure on $\mathrm{Mod}(R[F])$ and $D(R[F])$.
\end{notation}

The main results of \cite{BhattLurieRH}, summarized in Theorems~\ref{RHBigVer} and~\ref{RHSmallVer} below, roughly state that Frobenius modules provide the algebraic counterpart to the theory of $\FFp$-\'etale sheaves on $\Spec R$. The connection between the two is provided by the following construction:

\begin{construction}[The $\Sol$-functor]
\label{ConsSolDef}
There is a natural exact functor
\[
\Sol \colon D(R[F]) \to D(\Spec(R)_{\et}; \FFp)
\]
given informally by the formula
\[
\Sol (M,\phi_M)(R \to S) \colonequals \mathrm{cone}\big(M \otimes_R S \xrightarrow{\ \phi_M-1\ } M \otimes_R S\big)[-1] \in D(\FFp)
\]
for every \'etale morphism $R \to S$. More formally, one may define
\[
\Sol \colonequals \underline{\RHom}_{D(R[F])}( (R,\phi_R),-),
\]
where $(R,\phi_R)$ denotes the Frobenius module $R \xrightarrow{\ \phi\ } \phi_* R$ determined by the Frobenius on $R$.
\end{construction}

An informal summary of some of the main results of \cite{BhattLurieRH} is as follows:

\begin{theorem}[The covariant Riemann-Hilbert correspondence for $\FFp$-sheaves]
\label{RHBigVer}
The $\Sol$-functor from Construction~\ref{ConsSolDef} admits a left adjoint $\RH\colon D^+(\Spec(R)_{\et}, \FFp) \to D^+(R[F])$ with the following properties:
\begin{enumerate}[\quad\rm(1)]
\item\label{RHBigVer:1} $\RH$ is fully faithful and $t$-exact for the standard $t$-structures on the source and target; see \cite[\S 6.4 and Theorem~12.1.5]{BhattLurieRH}.

\item\label{RHBigVer:2} The essential image of $\RH$ is given by the full subcategory $D_{alg}(R[F]) \subset D(R[F])$ of \emph{algebraic} Frobenius complexes, i.e., those $(M,\phi_M) \in D(R[F])$ such that each cohomology group
\[
(N,\phi_N) \colonequals (H^i(M), \phi_i \colonequals H^i(\phi_M))
\]
is \emph{algebraic} in the sense that it satisfies the conditions:
\begin{enumerate}[\rm(i)]
\item $(N,\phi_N)$ is \emph{perfect}, i.e., $\phi_N\colon N \to F_* N$ is an isomorphism.
\item Each $x \in N$ is annihilated by a monic polynomial in $\phi_N$.
\end{enumerate}
Thus, $\RH$ induces an equivalence
\[
\RH\colon D^+(\Spec(R)_{\et}, \FFp) \simeq D^+_{alg}(R[F]),
\]
see \cite[Theorem~12.1.5]{BhattLurieRH}.

\item\label{RHBigVer:3} The triangulated category $D^+_{alg}(R[F])$ is naturally symmetric monoidal, with $\otimes$-product given by perfecting the $\otimes$-product over $R$. The $\RH(-)$ equivalences is symmetric monoidal for the standard $\otimes$-product on the source and the preceding one on the target; see \cite[\S 8]{BhattLurieRH}.

\item\label{RHBigVer:4} The $\RH$ equivalence above commutes with proper pushforward (where pushforwards of algebraic Frobenius modules are defined to be compatible with pushforward of underlying quasi-coherent sheaves); see \cite[\S 10.5]{BhattLurieRH}.

\item\label{RHBigVer:5} The $\RH$-equivalence is compatible with arbitrary pullback (where pullbacks of algebraic Frobenius modules are defined to be perfections of the pullback of underlying quasi-coherent sheaves); see \cite[\S 6.2]{BhattLurieRH}.
\end{enumerate}
Moreover, if $R$ is noetherian of finite Krull dimension, we can replace $D^+$ with the unbounded derived category in all results above (see \cite[Remark~12.1.6]{BhattLurieRH}).
\end{theorem}

The fundamental examples of this construction are the following.

\begin{example}[The $\RH$-functor in some key examples]
\label{exRHshriekext}
We have $\RH(\FF_{\!p,R}) \simeq R_{\perf}$ with its natural Frobenius structure by symmetric monoidality of $\RH$ applied to the unit object. More generally, if $j\colon U \into \Spec R$ is an open immersion whose complement $i\colon Z \to X$ is defined by an ideal $I \subset R$, then $\RH$ carries the inclusion $j_! \FF_{\!p,U} \subset \FF_{\!p,R}$ to the inclusion
\[
I_{\perf} \colonequals \varinjlim(I \xrightarrow{\ \phi\ } F_* I \xrightarrow{\ \phi\ } F^2_* I \xrightarrow{\ \phi\ }\cdots) \subset R_{\perf}.
\]
Indeed, by applying the exact functor $\RH(-)$ to the exact sequence
\[
0 \to j_! \FF_{\!p,U} \to \FF_{\!p,R} \to i_* \FF_{\!p,Z} \to 0
\]
and using its proper pushforward compatibility we see that $RH j_!\FF_{\!p,U}$ is the kernel of the map $R_{\perf} \to (R/I)_{\perf}$. By writing the colimits computing these perfection we see that this kernel is precisely the one defining $I_{\perf}$ as above.

The above example can be twisted by an $\FF_p$-complex i.e., given $F \in D(\Spec(R)_{\et}, \FFp)$ with $M = \RH(F)$, we have $\RH(j_! (F|_U)) = I_{\perf} \otimes_{R_{\perf}}^\LL M$: this follows since $j_!(F|_U) \simeq j_! \FF_{\!p,U} \otimes_{\FF_{\!p,R}} F$ and because $\RH$ is symmetric monoidal.

In the setup of the previous paragraph, assume further $F$ is an $\FFp$-sheaf, i.e., in the heart of the standard $t$-structure on $D(\Spec(R)_{\et}, \FFp)$. Via the $t$-exactness of $\RH$ in Theorem~\ref{RHBigVer}.\ref{RHBigVer:1} applied to the injection $j_! (F|_U) \to F$ of $\FFp$-sheaves, we learn that the natural map $ I_{\perf} \otimes_{R_{\perf}}^\LL M \to M$ is an injection of $R[F]$-modules, and thus that $I_{\perf} \otimes_{R_{\perf}}^\LL M = I_{\perf} \otimes_{R_{\perf}} M = I_{\perf} M$. It follows that if $N$ is any $R[F]$-module with perfection $M$ and $J \subset R$ is any ideal cutting out $Z \subset X$ up to radicals, then, since $(JN)_{\perf} = I_{\perf}M$, we have $\mathrm{Sol}(JN) = \mathrm{Sol}(I_{\perf}M) = j_! (F|_U)$. 
\end{example}

In applications, it is convenient to restrict Theorem~\ref{RHBigVer} to ``finite'' objects on both sides. The natural finiteness condition on the \'etale side is that of constructibility. The corresponding finiteness conditions on the algebraic side is dubbed holonomicity, recalled next.

\begin{construction}[Perfect, nilpotent, and holonomic Frobenius modules]
A Frobenius module $(M, \phi_M)$ over $R$ is \emph{perfect} if $\phi_M$ is an isomorphism; note that restriction of scalars along $R \to R_{\perf}$ identifies the categories of perfect Frobenius modules on these rings, so we may often restrict to perfect rings when working with perfect Frobenius modules. Any Frobenius module $(M,\phi_M)$ admits a \emph{perfection} $(M,\phi_M)_{\perf}$ (or simply $M_{\perf}$, if $\phi_M$ is clear) given concretely as the direct limit
\[
M_{\perf} \colonequals \varinjlim\big(M \to F_* M \to F^2_* M \to \cdots \big).
\]
A Frobenius module $(M,\phi_M)$ is \emph{nilpotent} (resp. \emph{locally nilpotent}) if the $n$-fold composite $\phi_M^n\colon M \to F^n_* M$ is $0$ for some $n$ (resp. $M_{\perf} = 0$); if $M$ is finitely generated as an $R[F]$-module, these two conditions are equivalent.

A perfect Frobenius module is \emph{holonomic} if it arises as the perfection of a Frobenius module $(M,\phi_M)$ with $M$ finitely presented over $R$; write $\mathrm{Mod}_{hol}(R[F]) \subset \mathrm{Mod}(R[F])$ for the full subcategory spanned by such modules. This subcategory is closed under kernels, cokernels, and extensions. Consequently, the full subcategory $D^b_{hol}(R[F]) \subset D^b(R[F])$ spanned by complexes of Frobenius modules with holonomic cohomology groups is itself a triangulated category. Moreover, the $t$-structure on $D^b(R[F])$ induces one on $D^b_{hol}(R[F])$, and we have $D^b_{hol}(R[F]) \subset D^b_{alg}(R[F])$, i.e., holonomic Frobenius $R$-modules are algebraic.
\end{construction}

With this notion, Theorem~\ref{RHBigVer} restricts to the following:

\begin{theorem}[The Riemann-Hilbert equivalence for constructible sheaves]
\label{RHSmallVer}
The $\RH$-functor from Theorem~\ref{RHBigVer} restricts to an equivalence
\[
\RH\colon D^b_c(\Spec(R)_{\et}, \FFp) \simeq D^b_{hol}(R[F]).
\]
This equivalence is $t$-exact (resp. symmetric monoidal) for the standard $t$-structure (resp.\ $\otimes$-product) on the source and target. Moreover, it commutes with proper pushforward and arbitrary pullback. (See \cite[\S 12.1]{BhattLurieRH}.)
\end{theorem}

\begin{remark}[Holonomic Frobenius complexes as a Verdier quotient]
\label{HolFrobModVerdQuot}
For $R$ noetherian, one can also identify $D^b_{hol}(R[F])$ as a Verdier quotient: if $D^b_{coh}(R[F])$ denotes the full subcategory of $D^b(R[F])$ spanned by $R[F]$-complexes whose cohomology groups are finitely generated $R$-modules, and if $D^b_{coh,nil}(R[F]) \subset D^b_{coh}(R[F])$ denotes the full subcategory spanned by complexes whose cohomology groups are nilpotent, then the perfection functor induces an equivalence
\[
D^b_{coh}(R[F])/D^b_{coh,nil}(R[F]) \simeq D^b_{hol}(R[F]),
\]
see \cite[Remark~12.4.5]{BhattLurieRH}. The same discussion also applies if we replace $D^b_{coh}(R[F])$ with the full subcategory of $D^b(R[F])$ spanned by complexes that are perfect when regarded as $R$-complexes (and $D^b_{coh,nil}(R[F])$ must be replaced with the full subcategory of nilpotents); when formulated in this fashion, the description applies to all $\FFp$-algebras $R$.
\end{remark}

\subsection{A nilpotency result for Frobenius modules}

In this subsection, we record (with a new proof) a partially global variant of a result due to Hartshorne-Speiser, Lyubeznik, and Gabber (see Corollary~\ref{LyubeznikNilp}); it will be useful later in understanding perverse $\FFp$-sheaves on local rings.

\begin{proposition}
\label{HolonomicLocalCoh}
Let $A_0$ be a noetherian ring of characteristic $p$, and $\frakm_0 \subset A_0$ an ideal contained in the Jacobson radical. If $M \in D^b_{hol}(A_0[F])$ is a holonomic Frobenius complex with $R\Gamma_{\frakm_0}(M) \simeq 0$, then $M \simeq 0$.
\end{proposition}

We prove this by reduction to the case of a rank $1$ valuation ring with algebraically closed fraction field:

\begin{proof}
Let $A$ be the perfection of $A_0$, so $M$ is naturally an $A$-module. Write $\frakm\colonequals \frakm_0 A$, so $\frakm$ lies in the Jacobson radical of $A$. Assume $M \neq 0$. By \cite[Proposition~5.3.3]{BhattLurieRH}, there exists a prime ideal $\frakp\subset A$ such that $M \otimes_A^\LL \kappa(\frakp) \neq 0$. Since $R\Gamma_\frakm(M) = 0$, we also have $M \otimes_A^\LL \kappa(\frakq) = 0$ for any prime $\frakq$ containing $\frakm$, so $\frakm\not\subset \frakp$. As $\frakm$ is contained in the Jacobson radical, there exists some prime $\frakq$ containing $\frakm$ such that $\frakp\subset \frakq$. Since $A$ is the perfection of a noetherian ring, we may choose\footnote{Say $\frakp \rightsquigarrow \frakq$ is a specialization of primes in a noetherian ring $B$; we shall explain how to find a discrete valuation ring $V$ and a map $B \to V$ witnessing this specialization. (Replacing $V$ with its absolute integral closure then provides the valuation ring needed in the proof above.) To find such a $V$, we may replace $B$ with a maximal dimensional irreducible component of the $\frakq$-adic completion of $(B/\frakp)_\frakq$ and thus assume that $(B,\frakq)$ is a complete noetherian local domain and $\frakp = 0$. In this case, consider the normalization $X \to \mathrm{Bl}_\frakq(\Spec B)$ of the displayed blowup. This map is surjective, so there is a point $x \in X$ lying above a generic point of the exceptional divisor $E \subset \mathrm{Bl}_\frakq(\Spec B)$. The local ring $V=\calO_{X,x}$ then provides the desired discrete valuation ring over $B$.} an absolutely integrally closed valuation ring $V$ of rank~$1$ and a map $f\colon A \to V$ that witnesses the specialization $\frakp\rightsquigarrow \frakq$, i.e., $f^{-1}(0) = \frakp$ and $f^{-1}(\frakm_V) = \frakq$, where~$\frakm_V$ is the maximal ideal of $V$; let $K$ and $k$ denote the fraction and residue fields of $V$ respectively. Since $\frakm\not\subset \frakp$, the ideal $\frakm V \subset V$ is nonzero, so $\sqrt{\frakm V} = \sqrt{\frakm_V}$ as $V$ is a rank $1$ valuation ring. Our assumptions on $M$ then imply that if we set $M_V \colonequals M \otimes_A^\LL V$, then $R\Gamma_{\frakm_V}(M_V) \simeq 0$ and $M_V \otimes_V^\LL K \neq 0$. Consider the holonomic complex
\[
N \colonequals \frakm_V \otimes_V^\LL M_V \in D^b_{hol}(V[F]).
\]
The previous properties for $M_V$ show that $N \otimes_V^\LL K \neq 0$ and that $R\Gamma_{\frakm_V}(N) = 0$. But, under Riemann-Hilbert, the complex $N = \frakm_V \otimes_V^\LL M_V$ corresponds to $j_! \Sol(M_V \otimes_V^\LL K)$, where $j\colon \Spec K \subset \Spec V$ is the inclusion of the generic point (see Example~\ref{exRHshriekext}). As $K$ is algebraically closed, the complex
\[
\Sol(M_V \otimes_V^\LL K) \in D^b_c(\Spec(K)_{\et}, \FFp)
\]
is constant, i.e., a finite direct sum of copies of shifts of the constant sheaf $\underline\FF_p$. It follows that $N \in D^b_{hol}(V[F])$ is a finite direct sum of copies of shifts of $\RH(j_!\underline\FF_p) \simeq \frakm_V$; moreover, since $N \neq 0$, there is at least one nonzero direct summand. Our hypothesis $R\Gamma_{\frakm_V}(N) = 0$ then implies that $R\Gamma_{\frakm_V}(\frakm_V) \simeq 0$. But this means that $R\Gamma_{\frakm_V}(V) \simeq R\Gamma_{\frakm_V}(k) \simeq k$. Now it is easy to see that $R\Gamma_{\frakm_V}(V) \simeq K/V$, so we get a contradiction since~$K/V$ and $k$ are not isomorphic $V$-modules.
\end{proof}

\begin{remark}
Proposition~\ref{HolonomicLocalCoh} is false without some noetherianness assumptions on $A$, even if we require $\frakm_0$ to be finitely generated. For instance, say $V$ is a perfect valuation ring of rank $2$. Then the poset of prime ideals of $V$ is given by $0 \subsetneq \frakp\subsetneq \frakq$ with $\frakq$ maximal. Fix some $f \in \frakp - \{0\}$ and $g \in \frakq - \frakp$. As radical ideals in valuation rings are prime, we get that $\sqrt{fV} = \frakp$ and $\sqrt{gV} = \frakq$, so all primes are finitely generated up to radicals. Set $\frakm_0= gV$, so $\frakm_0 \subset V$ is a finitely generated ideal contained in the Jacobson radical (as its radical equals the maximal ideal of the local ring $V$). Consider the~$V$-module $M \colonequals \frakp$, viewed as a $V[F]$-module with its natural Frobenius. It is easy to see that $M$ is holonomic: it arises as the perfection of $M_0 = fV$ with its natural Frobenius. We claim that $R\Gamma_{\frakm_0}(M) = 0$, even though $M \neq 0$. To see this, note that
\[
R\Gamma_{\frakm_0}(N) \simeq \big(N[1/g]/N\big)[-1]
\]
for any $V$-module $N$ since $V[1/g] \simeq R\Gamma(\Spec V - \{\sqrt{\frakm_0}\}, \calO)$. Thus, to show that $R\Gamma_{\frakm_0}(\frakp) = 0$, it is enough to show that $g$ acts invertibly on $\frakp$; this is a general fact about valuation rings. Indeed, given $x \in \frakp$, we cannot have $g \in xV$ since $g \notin \frakp$, so we must have $x \in gV$ as $V$ is a valuation ring, whence $x/g \in V$; moreover, since $g(x/g) = x \in \frakp$ and $\frakp$ is prime with $g \notin \frakp$, we must have $x/g \in \frakp$, so multiplication by both $g$ and $1/g$ preserve $\frakp$, whence $g$ acts invertibly on $\frakp$.
\end{remark}

\begin{corollary}[Hartshorne-Speiser, Lyubeznik, Gabber]
\label{LyubeznikNilp}
Let $A$ be a noetherian local ring, and $M$ a cofinite $A$-module with an $A$-linear map $\phi_M\colon M \to F_* M$. If $(M,\phi_M)$ is locally nilpotent, then it is nilpotent.
\end{corollary}

For the original versions, see \cite[Proposition~1.11]{HartSpeis}, \cite[Proposition~4.4]{Lyubeznik:Fmod}, and \cite[Lemma~13.1]{Gabbert}.

\begin{proof}
Let $\widehat{M}$ be the perverse coherent $A$-module corresponding to $M$ under Corollary~\ref{PervCof}. Our hypothesis is that $M \simeq R\Gamma_{\frakm}(\widehat{M})$ has trivial perfection. As $R\Gamma_{\frakm}(-)$ commutes with filtered colimits, Proposition~\ref{HolonomicLocalCoh} implies that the perfection of $\widehat{M}$ is $0$. As $\widehat{M}$ is coherent, this implies that the Frobenius on $\widehat{M}$ is nilpotent, whence the same also holds true for $M \simeq R\Gamma_{\frakm}(\widehat{M})$.
\end{proof}

\subsection{Perverse $\FFp$-sheaves}
\label{ss:PervCohFrob}

In this subsection, we introduce the perverse $t$-structure on constructible $\FFp$-\'etale sheaves. Our approach is to define this $t$-structure as the image of the perverse $t$-structure on coherent sheaves under the Riemann-Hilbert equivalence (as well as the quotient description in Remark~\ref{HolFrobModVerdQuot}; this makes certain properties of this $t$-structure (e.g., preservation of constructibility under perverse truncations) obvious. In \S\ref{ss:GabberPervComp}, we shall show that our definition of the perverse $t$-structure agrees with the more classical one (in terms of support conditions for stalks and costalks) studied by Gabber \cite{Gabbert}.

\begin{notation}
As above, we work with schemes of finite type over a noetherian local $F$-finite ring $A$. Any such scheme $X$ comes endowed with a perverse $t$-structure on $D^b_{coh}(X)$ as in \S\ref{ss:PervCoh}.
\end{notation}

\begin{construction}[Perverse coherent Frobenius modules]
\label{PervCohFrob}
Fix a finite type $A$-scheme $X$, and let $\calO_X[F]$ be the twisted polynomial ring as in Notation~\ref{FrobModNot}, regarded as a quasi-coherent sheaf on $X$. Write $D_{qc}(X[F]) \subset D(X,\calO_X[F])$ be the full subcategory spanned by objects whose cohomology sheaves are quasi-coherent over $\calO_X$. Applying \cite[\S 2]{Gabbert} with the sheaf of rings $\calO_X[F]$ and perversity function $-d(x)$ shows that $D_{qc}(X, \calO_X[F])$ has a perverse $t$-structure, where the connectivity and coconnectivity are tested after forgetting down to $D_{qc}(X)$. Let $D^b_{coh}(X[F]) \subset D(X,\calO_X[F])$ be the full subcategory spanned by bounded objects whose cohomology sheaves are coherent over $\calO_X$. Using Remark~\ref{CompGabberCoh}, it follows that there is a natural $t$-structure on $D^b_{coh}(X[F])$ where connectivity and coconnectivity are tested after forgetting down to $D^b_{coh}(X)$. Write $\Perv_{coh}(X[F]) \subset D^b_{coh}(X[F])$ for the heart, so it is the full subcategory spanned by Frobenius complexes $M \to F_* M$ on $X$ with $M \in \Perv_{coh}(X)$. We refer to objects of $\Perv_{coh}(X[F])$ as \emph{perverse coherent Frobenius modules on $X$}.
\end{construction}

To study this notion, we need a couple of lemmas on the behavior of nilpotent Frobenius modules with respect to the perverse $t$-structure.

\begin{lemma}
\label{Nilp1}
Fix a finite type $A$-scheme $X$. Fix a Frobenius complex $M \in D^b_{coh}(X[F])$ which is nilpotent. Then each ${}^p H^i(M)$ is also nilpotent.
\end{lemma}

\begin{proof}
Fix a point $x \in X$. It is enough to show that each ${}^p H^i(M)$ is nilpotent in a neighborhood of $x$. Generally, since nilpotency of a Frobenius module $L$ is the vanishing of some fixed power of the structural map $\phi_L \colon L \to F^e_*L$, and since restriction commutes with $F^e_*$, nilpotency in some neighborhood of a point $x$ is equivalent to the nilpotency of the stalk at $x$. Hence it is enough to show that the stalk ${}^p H^i(M)_x$ is a nilpotent Frobenius complex over $\calO_{X,x}$. As taking stalks is perverse $t$-exact (see Construction~\ref{PervCohStalks}), we are reduced to the following local analog: if $N \in D^b_{coh}(\calO_{X,x}[F])$ is a nilpotent Frobenius complex, then each ${}^p H^i(N)$ is nilpotent. To see this, note that the assumption on $N$ implies that each $H^i_{\{x\}}(N)$ is a nilpotent cofinite Frobenius module. Corollary~\ref{PervCof} applied to the ring $\calO_{X,x}$ as well as Construction~\ref{PervCohStalks} show that the local cohomology functor $R\Gamma_{\{x\}}(-)$ carries the perverse $t$-structure on $D^b_{coh}(\calO_{X,x})$ faithfully to the standard $t$-structure on $D^b_{\cof}(\calO_{X,x})$, up to a shift of $d(x)$. Consequently, it follows that each ${}^p H^i(N)$ is also nilpotent, as wanted.
\end{proof}

\begin{lemma}
\label{Nilp2}
Fix a finite type $A$-scheme $X$. The nilpotent coherent perverse complexes
\[
\Perv_{coh,nil}(X[F]) \colonequals D^b_{coh,nil}(X[F]) \cap \Perv_{coh}(X[F])
\]
form an additive subcategory of $\Perv_{coh}(X[F])$ closed under subobjects, quotient objects, and extensions; in particular, $\Perv_{coh,nil}(X[F]) \subset \Perv_{coh}(X[F])$ is a Serre subcategory.
\end{lemma}

\begin{proof}
Say $0 \to K \to L \to M \to 0$ is an exact sequence in $\Perv_{coh}(X[F])$. It is enough to show that $L$ is nilpotent exactly when $K$ and $M$ are nilpotent. If $K$ and $M$ are nilpotent (meaning their perfection is zero), then applying the exact perfection functor to the above exact sequence, shows that the perfection of $L$ is also zero; hence $L$ is nilpotent. Conversely, assume $L$ is nilpotent. We must show that $K$ and $M$ are nilpotent. It is enough to check this for the stalks, so fix some $x \in X$. By functoriality, the nilpotence of $L$ implies that each $H^i_x(L_x)$ is a nilpotent cofinite Frobenius module over $\calO_{X,x}$ for all $i$. By Corollary~\ref{PervCof} and Construction~\ref{PervCohStalks}, we also know that for any perverse coherent $\calO_{X,x}$-module $N$, we have $H^i_x(N) = 0$ unless $i = -d(x)$. Consequently, the perversity of $K$, $L$, and~$M$ implies that the sequence
\[
0 \to H^{-d(x)}_x(K_x) \to H^{-d(x)}_x(L_x) \to H^{-d(x)}_x(M_x) \to 0
\]
is exact, and that the local cohomology modules for $K_x$, $L_x$, and $M_x$ vanish in all other degrees. Now the nilpotence of $H^{-d(x)}_x(L_x)$ implies that both $H^{-d(x)}_x(K_x)$ and $H^{-d(x)}_x(M_x)$ have trivial perfection (as passage to the perfection is exact). Since for $i \neq {-d(x)}$ these modules are zero they also have trivial perfection. Hence Proposition~\ref{HolonomicLocalCoh} implies that $K_x$ and $M_x$ also have trivial perfection, so they are indeed nilpotent.
\end{proof}

\begin{construction}[Perverse holonomic Frobenius modules and perverse $\FFp$-sheaves]
\label{ConsPervSheaf}
Denote by
\[
\Perv_{hol}(X[F]) \subset D^b_{hol}(X[F])
\]
the essential image of $\Perv_{coh}(X[F])$ under the perfection functor $D^b_{coh}(X[F]) \to D^b_{hol}(X[F])$; we refer to its objects as \emph{perverse holonomic Frobenius modules on $X$}. Write $\Perv_c(X;\FFp) \subset D^b_c(X;\FFp)$ for the essential image of $\Perv_{hol}(X[F])$ under the Riemann-Hilbert equivalence $D^b_{hol}(X[F]) \simeq D^b_c(X;\FFp)$; we refer to this as the category of \emph{perverse (constructible) sheaves on $X$}.
\end{construction}

We shall soon see that $\Perv_c(X;\FFp)$ is the heart of a $t$-structure on $D^b_c(X_{\et};\FFp)$.

\begin{example}
\label{CMConstPerv}
Say $X$ is a Cohen-Macaulay $A$-scheme of dimension $d$. Then $\calO_X[d]$ is a perverse coherent sheaf on $X$ (see Example~\ref{CMPerverseCoh}), and consequently $\FF_{\!p,X}[d]$ is a perverse $\FFp$-sheaf on $X$. Note that this is in contrast to what happens in characteristic $0$: the (shifted) constant sheaf is typically not perverse on a Cohen-Macaulay complex variety.
\end{example}

To identify the category of perverse holonomic Frobenius modules explicitly as a quotient, we shall need the following abstract lemma.

\begin{lemma}
\label{Triangtstructure}
Let $\calC$ be a triangulated category equipped with a bounded $t$-structure $(\calC^{\le 0}, \calC^{\ge 0})$. Say $\calD \subset \calC$ is a full triangulated subcategory satisfying the following:
\begin{enumerate}[\quad\rm(1)]
\item For $X \in \calD$, we have $H^i(X) \in \calD$ for all $i$. Consequently, $(\calC^{\le 0} \cap \calD, \calC^{\ge 0} \cap \calD)$ gives a $t$-structure on~$\calD$

\item The category $\calD^{\heartsuit} \colonequals \calC^{\heartsuit} \cap \calD$ is closed under passage to subobjects and quotient objects in $\calC^{\heartsuit}$, and consequently it is closed under extension in $\calC^{\heartsuit}$.
\end{enumerate}

Then the triangulated category $\calE \colonequals \calC/\calD$ admits a unique $t$-structure $(\calE^{\le 0}, \calE^{\ge 0})$ for which the quotient functor $q\colon \calC \to \calE$ is $t$-exact; explicitly, the category $\calE^{\le 0}$ (resp. $\calE^{\ge 0}$) is the essential image of $\calC^{\le 0}$ (resp.~$\calC^{\ge 0}$). Moreover, we have $\calE^{\heartsuit} \colonequals \calC^{\heartsuit}/\calD^{\heartsuit}$.
\end{lemma}

\begin{proof}
To obtain the $t$-structure on $\calE$, consider the pair $(\calE^{\le 0}, \calE^{\ge 0})$ as defined in the lemma. We shall check that this gives a $t$-structure on $\calE$; the uniqueness and $t$-exactness of $q$ then trivially follow. Stability under appropriate shifts and existence of triangles is clear. It is thus enough to show that for $x \in \calE^{\le 0}$ and $y \in \calE^{\ge 1} \colonequals \calE^{\ge 0}[-1]$, we have $\Hom_{\calE}(x,y) = 0$. Fix a map $f\colon x \to y$. Choose lifts $X \in \calC^{\le 0}$ and $Y \in \calC^{\ge 1}$ of $x$ and $y$. By definition of the Verdier quotient, we can represent the map $f$ by a diagram $(X \xleftarrow{\ s\ } Z \xrightarrow{\ g\ } Y)$, where $s\colon Z \to X$ is a map in $\calC$ whose cone lies in $\calD$. Using the long exact sequence on cohomology and properties $(1)$ and $(2)$ above, it is easy to see that $\tau^{\ge 1} Z \in \calD$, and hence the map $\tau^{\le 0} Z \to Z$ also has a cone in $\calD$. Composing this map with the diagram $(X \xleftarrow{\ s\ } Z \xrightarrow{\ g\ } Y)$, we learn that our map $f$ can then also be represented by the diagram $(X \xleftarrow{\ s'\ } \tau^{\le 0} Z \xrightarrow{\ g'\ } Y)$. But $Y \in \calC^{\ge 1}$, so the map $g'$ must be $0$ by the $t$-structure axioms for $\calC$. It follows that $f = 0$, as wanted.

We leave the identification of $\calE^{\heartsuit}$ as an exercise.
\end{proof}

We obtain a perverse $t$-structure on holonomic Frobenius complexes:

\begin{corollary}
\label{PervHolQuot}
Let $X$ be a finite type $A$-scheme. Then there is a unique $t$-structure on $D^b_{hol}(X[F])$ compatible with the perverse $t$-structure on $D^b_{coh}(X[F])$, i.e., such that the perfection functor
\[
D^b_{coh}(X[F]) \to D^b_{hol}(X[F])
\]
is $t$-exact. Under this $t$-structure, the heart $D^b_{hol}(X[F])^{\heartsuit}$ equals $\Perv_{hol}(X[F])$ and can be identified with the quotient $\Perv_{coh}(X[F])/\Perv_{coh,nil}(X[F])$.
\end{corollary}

\begin{proof}
Apply Lemma~\ref{Triangtstructure} to the equivalence $D^b_{coh}(X[F])/D^b_{coh,nil}(X[F]) \simeq D^b_{hol}(X[F])$ using Lemmas~\ref{Nilp1} and~\ref{Nilp2} to ensure that the hypotheses apply.
\end{proof}

\begin{definition}[The perverse $t$-structure]
\label{PervtDef}
Let $X$ be a finite type $A$-scheme. The \emph{(middle) perverse $t$-structure on $D^b_c(X_{\et};\FFp)$} is the $t$-structure obtained by transporting the $t$-structure on $D^b_{hol}(X[F])$ provided by Corollary~\ref{PervHolQuot} across the Riemann-Hilbert sequence $D^b_c(X_{\et};\FFp) \simeq D^b_{hol}(X[F])$. Its heart is the category $\Perv_c(X;\FFp)$ of perverse sheaves defined in Construction~\ref{ConsPervSheaf}.
\end{definition}

The following proposition is a key tool in extracting algebraic consequences of topological results:

\begin{proposition}
\label{RHtexactLocalCoh}
Let $X$ be a finite type $A$-scheme. For any $x \in X$, the composition
\[
D^b_c(X_{\et}, \FFp) \xrightarrow{\ \RH\ } D^b_{hol}(X[F]) \xrightarrow{\ R\Gamma_x( (-)_x)[-d(x)]\ } D^b(\calO_{X,x})
\]
is $t$-exact for the perverse $t$-structure on the source and the standard $t$-structure on the target.
\end{proposition}

\begin{proof}
The Riemann-Hilbert functor $\RH$ is $t$-exact for the perverse $t$-structures on the source and target essentially by definition (Definition~\ref{PervtDef}). Corollary~\ref{RecogPerv} ensures that $R\Gamma_x( (-)_x)[-d(x)]$ is $t$-exact for the perverse $t$-structure on the source and the standard $t$-structure on the target, so the claim follows.
\end{proof}

\begin{example}[Cohen-Macaulayness for Riemann-Hilbert partners of perverse sheaves]
\label{CMPerv}
Let $X$ be a finite type $A$-scheme. If $F \in \Perv_c(X_{\et}, \FFp)$ and $d=d(X)$, then $\RH(F)[-d]$ is Cohen-Macaulay on $X$ in the following sense: $R\Gamma_x(\RH(F)_x[-d])$ is concentrated in degree $d-d(x)$. Note that if $X=\Spec A$ and $A$ is equidimensional, then $d-d(x) = \dim(\calO_{X,x})$: indeed, $d(x) + \dim(\calO_{X,x})$ is the length of a maximal chain of specializations in $\Spec A$, and equidimensionality and locality of $A$ ensures that any such chain has length $d=\dim A$. This result was also observed by Cass \cite[Theorem~1.6]{CassPerv}.
\end{example}

Before specializing the above discussion to the setting of local rings, we recall two further notions of modules with an action of the Frobenius which fit into this framework:

\emph{Cartier modules} \cite{BliBoeCart}: Let $A$ be an $F$-finite ring. A \emph{Cartier module} is an $A$-module $M$, together with an $A$-linear map $\kappa\colon F_*M \to M$. Equivalently, analogous to the case of Frobenius modules, we might think of a Cartier module simple as a right $R[F]$-module. From this point of view, nilpotence and locally nilpotence are defined analogously to the case of Frobenius modules. A Cartier module is \emph{coherent} if its underlying $A$-module is coherent. We denote this category by $\mathrm{CohCart}(A[F])$.

\emph{Finitely generated unit Frobenius modules} \cite{BhattLurieRH,EK}: Let $R$ be a regular ring. A Frobenius module $M \to F_*M$ is a \emph{unit} Frobenius module if the adjoint structural map $F^*M \to M$ is an isomorphism. It is \emph{finitely generated unit (fgu)} if $M$ is unit and finitely generated as a left $R[F]$-module.

\begin{corollary}
\label{PerverseCharpLocal}
Let $A$ be a complete noetherian local $F$-finite ring. Write $\mathrm{Mod}_{\cof}(A[F]) \subset \mathrm{Mod}(A[F])$ for the full subcategory spanned by $A$-cofinite $A[F]$-modules. The following categories are equivalent:
\begin{enumerate}[\quad\rm(1)]
\item The quotient by nilpotents of $\mathrm{Mod}_{\cof}(A[F])$.

\item The essential image of $\mathrm{Mod}_{\cof}(A[F])$ in $\mathrm{Mod}_{\perf}(A[F])$ under the perfection functor.

\item The category $\Perv_{hol}(A[F])$ of perverse holonomic Frobenius modules on $\Spec A$.

\item The category $\Perv_c(\Spec A;\FFp)$ of perverse sheaves on $\Spec A$.

\item The (opposite of) the quotient by nilpotents of the category of Cartier modules $\mathrm{CohCart}(A[F])$.
\end{enumerate}
If we write $A = R/I$ as the quotient of an $F$-finite complete noetherian regular local ring $R$, then the above categories are also equivalent to the following:
\begin{enumerate}[\quad\rm(1)]
\setcounter{enumi}{5}
\item The (opposite of) the full subcategory $\mathrm{Mod}_{fgu,A}(R) \subset \mathrm{Mod}_{fgu}(R)$ spanned by images of pairs
\[
\big(N \in \mathrm{Mod}_{fg}(A),\ \psi_N\colon N \to F_R^* N\big)
\]
under the unitalization functor $\operatorname{colim}_{\psi} F_R^{e*}N$.
\end{enumerate}
\end{corollary}

\begin{proof}
The equivalence of $(3)$ and $(4)$ is the Riemann-Hilbert correspondence; the equivalence of $(1)$ and~$(2)$ comes from Corollary~\ref{LyubeznikNilp}. The equivalence of $(1)$ and $(3)$ comes from Corollary~\ref{PervCof} and Corollary~\ref{PervHolQuot}. The equivalence of $(4)$ and $(6)$ is the contravariant Riemann-Hilbert correspondence of \cite{BhattLurieRH}.

The equivalence between (4), (5), and (6) was first obtained by Schedlmeier in \cite{Schedlmeier} building on the smooth case treated in \cite{EK}. It also follows from \cite{BaudinDual} which was obtained independently. We sketch a fairly direct argument for the equivalence of (3) and (5) closer to the spirit of the present paper. We claim that Grothendieck duality induces an equivalence
\[
\mathrm{CohCart}(A[F]) \xrightarrow{\ \simeq\ } \Perv_{coh}(A[F])
\]
which is easily observed to preserve nilpotence on both sides, and hence induces the claimed equivalence of (4) and (5) using Corollary~\ref{PervHolQuot}. Since $A$ is the quotient of a regular ring with a $p$-basis \cite[Section 13]{Gabbert}, $A$ has a normalized dualizing complex which is equipped with a quasi-isomorphism $\eta\colon \omega_A \xrightarrow{\ \simeq\ } F^!\omega_A$.

Given a Cartier Module $\kappa\colon F_*M \to M$, we apply Grothendieck duality $\DD_A( \cdot )=\RHom_A(\cdot, \omega_A)$ to the structure map $\kappa$ to obtain
\[
\DD_A(M) \xrightarrow{\ \DD_A\kappa\ } \DD_A(F_*M) = \RHom_A(F_*M, \omega_A) \simeq F_*\RHom_A(M, F^!\omega_A) \xrightarrow{\ \eta^{-1}\ } F_*\DD_A(M),
\]
a map in the derived category of $A$-modules. By definition $\DD_A(M)$ and hence $F_*\DD_A(M)$ both lie in $\Perv_{coh}(A)$. Conversely, given such datum of $\phi\colon N \to F_*N$, where $N \in \Perv_{coh}(A)$ and $\phi$ is a map in $D^b_{coh}(A)$ the same calculation yields, after applying Grothendieck duality, the map
\[
F_*\DD_A(N) \simeq \DD_A(F_*N) \xrightarrow{\ \DD_A\phi\ } \DD_A(N)
\]
in the derived category of $A$-modules. Since by definition $\DD(N)$ and $F_*\DD(N)$ are actual $A$-modules, this is indeed a map of $A$-modules, hence the structure of a Cartier module on $\DD_A(N)$. This shows an equivalence between the categories
\[
\mathrm{CohCart}(A[F]) \simeq \left\{ (N,\phi) \mid N \in \Perv_{coh}(A),\ \phi\colon N \to F_*N \in D^b_{coh}(A)\right\} \equalscolon \Perv^{naive}_{coh}(A[F]),
\]
where the morphisms on the right are the obvious commutative diagrams in $D^b_{coh}(A)$. To finish our claim we have to show that the right hand side is equivalent to $\Perv_{coh}(A[F])$.

An object $M \in D^b_{coh}(A[F])$ is represented by a complex of Frobenius modules. In particular $M$ comes equipped with its structural map $\phi \colon M \to F_*M$ which is a map of complexes of $A$-modules (even $A[F]$-modules). By viewing this map of complexes as a map in $D^b_{coh}(A)$, we obtain a functor (after restriction to $\Perv_{coh}(A[F]) \subseteq D^b_{coh}(A[F])$) 
\[
\Perv_{coh}(A[F]) \to \Perv^{naive}_{coh}(A[F]).
\]
We show that this functor is an equivalence: For essential surjectivity, observe that for any map $N \to F_*N$ in $D^b_{coh}(A)$ we may replace $N$ by a quasi-isomorphic complex of projective $A$-modules. In this case, the map $\phi$ in the derived category is in fact represented by a (homotopy class of a) map of the complexes $N \to F_*N$. But this precisely means that each module in the complex $N$ is equipped with the structure of a Frobenius module, and that the differentials respect this structure; i.e., $N$ is a complex of Frobenius modules. Note that for the essential surjectivity we did not need to restrict to the perverse subcategory.

For the fully faithfulness let $M,N \in \Perv_{coh}(A[F])$. We may replace $M$ by a quasi-isomorphic complex consisting of free $A[F]$-modules; then the underlying $A$-modules are also free. Then we obtain a triangle in the derived category of abelian groups
\[
\RHom_{A[F]} (M,N) \to \RHom_{A} (M,N) \xrightarrow{\ \delta\ } \RHom_{A} (M,F_*N) \xrightarrow{\ +1\ }
\]
where the first map is induced by forgetting the Frobenius action and the second is given by
\[
f \mapsto F_*f \circ \phi_M - \phi_N \circ f
\]
on the level of $\Hom$-complexes if $M$ is $A$-projective, where $\phi_M\colon M \to F_*M$ (resp. $\phi_N$ ) are the maps induced by the Frobenius structure on $M$ (resp. $N$). The cohomology sequence in degree zero then yields the exact sequence
\[
\Ext^{-1}_{D(A)} (M,F_*N) \to \Hom_{D(A[F])}(M,N) \to \Hom_{D(A)}(M,N) \xrightarrow{\ \delta\ } \Hom_{D(A)}(M,F_*N).
\]
By definition the kernel of the map induced by $\delta$ are the morphisms in $\Perv^{naive}_{coh}(A[F])$. To see that this is equal to the morphisms in $\Perv_{coh}(A[F])$ --- which is the term $\Hom_{D(A[F])}(M,N)$ in the sequence --- it is enough to show the vanishing of the term $\Ext^{-1}_{D(A)} (M,F_*N)$. But by Grothendieck duality this is equal to $\Ext^{-1}_{D(A)}(F_*\DD_A(N),\DD_A(M))$ which, since $\DD_A(N)$ and $\DD_A(M)$ are both $A$-modules (since $M,N$ are assumed to be perverse), is equal to $\Ext^{-1}_A(F_*\DD_A(N),\DD_A(M))$ and hence equal to zero.
\end{proof}

The following example is fundamental to our later applications.

\begin{example}[Local cohomology under Riemann-Hilbert]
\label{ExPervFrobLocal}
Let $R \to A$ be a surjection of complete noetherian local rings with kernel $I \subset R$ and $R$ regular of dimension $d$. Fix an integer $i$. Under the equivalence in Corollary~\ref{PerverseCharpLocal}, the following objects match up:
\begin{enumerate}[\quad\rm(1)]
\item The image of the local cohomology module $H^i_\frakm(A)$ in $\mathrm{Mod}_{\cof}(A[F])/\text{nilpotents}$.

\item The local cohomology module $H^i_\frakm(A_{\perf}) \in \mathrm{Mod}_{\perf}(A[F])$.

\item The image of the derived $\frakm$-adic completion $\widehat{H^i_\frakm(A)}$ in $\Perv_{hol}(A[F])$.

\item The perverse cohomology sheaf ${}^p H^i(\underline\FF_{p, A}) \in \Perv_c(\Spec A;\FFp)$.

\item The Cartier module $(H^{-i}(\omega_A^\bullet), H^{-i}(\mathrm{tr}_F))$, where $\mathrm{tr}_F\colon F_* \omega_A^\bullet \to \omega_A^\bullet$ is the trace map for the Frobenius.

\item The local cohomology module $H^{d-i}_I(R) \in \mathrm{Mod}_{fgu,A}(R)$.
\end{enumerate}
In (1), (2), (3), and (6), the Frobenius actions are the natural ones.
\end{example}

In particular, the connection between (1) and (6), as explained in \cite[Example~4.8]{Lyubeznik:Fmod}, can be used together with Corollary~\ref{PerverseCharpLocal} to recover \cite[Theorem~1.1]{Lyubeznik:Compositio}:

\begin{corollary}[Lyubeznik]
$H^{d-i}_I(R) = 0$ if and only if the natural Frobenius action on $H^i_{\frakm}(A)$ is nilpotent.
\end{corollary}

Concretely, under the equivalences inducing (6), (4), and (2), the objects $R\Gamma_I(R)$, $\FF_{\!p,A}$, and $R\Gamma_\frakm(A_{\perf})$ correspond to one another. Denoting $i\colon \{x\} \to X$ the inclusion of the point corresponding to the maximal ideal $\frakm$, we have analogously that the functors $R\Gamma_\frakm(-)$, $i_*i^*(-)$, and $( - \otimes^\LL_A A/\frakm)_{\perf}$ match up. This has the following consequence for the invariants introduced in \cite[\S4]{Lyubeznik:Invent}, cf.~\cite{BlickleBondu}:
\[
\lambda_{j,d-i}(A) \colonequals \ell(H^j_{\frakm} (H^{d-i}_I(R))) = \ell(\Tor^A_j(H^i_{\frakm}(A), k)_{\perf}) = \dim_{\FFp}(H^{-j} \circ i^* \circ {}^p H^i(\underline\FF_{p,A})),
\]

\begin{remark}[Cohen-Macaulayness of $A$ and the perversity of {$\FFp[\dim A]$}]
\label{ConstantCMPerverse}
If $A$ is Cohen-Macaulay, then $\FFp[\dim A]$ is perverse: this follows as $A = \DD_A(\omega_A^\bullet)$ is perverse coherent. In fact, one has a stronger statement: $\FF_{\!p,A}[\dim A]$ is perverse if and only if $A$ is ``Cohen-Macaulay up to Frobenius nilpotents,'' i.e., if and only if $H^{< \dim A}_{\frakm}(A)$ is Frobenius nilpotent. Indeed, this follows from the equivalence of (1) and (4) in Corollary~\ref{PerverseCharpLocal}, applied using Example~\ref{ExPervFrobLocal}.
\end{remark}

It will be convenient to have a variant of Corollary~\ref{PerverseCharpLocal} and Example~\ref{ExPervFrobLocal} in the ``affine'' setting:

\begin{remark}
Say $k$ is an algebraically closed field of characteristic $p$. Let $X = \Spec A$ be a finite type affine $k$-scheme, and let $X \into Y = \Spec R$ be a closed immersion into a smooth $k$-scheme $Y$ of dimension~$d$. One can then show by mimicking Corollary~\ref{PerverseCharpLocal} that the following categories are equivalent:
\begin{enumerate}[\quad\rm(1)]
\item[(3')] The category $\Perv_{hol}(X[F])$ of perverse holonomic Frobenius modules on $X$.

\item[(4')] The category $\Perv_c(X;\FFp)$ of perverse sheaves on $X$.

\item[(5')] The (opposite of) the quotient by nilpotents of the category of Cartier modules on $X$.

\item[(6')] The (opposite of) the full subcategory $\mathrm{Mod}_{fgu,A}(R) \subset \mathrm{Mod}_{fgu}(R)$ spanned by images of pairs
\[
\big(N \in \mathrm{Mod}_{fg}(A), \psi_N\colon N \to F_R^* N\big)
\]
under the unitalization functor.
\end{enumerate}
Similarly as in the local case above we also obtain the an analogous description of Lyubeznik invariants
\[
\lambda_{j,d-i}(A) \colonequals \ell(H^j_{\frakm} (H^{d-i}_I(R))) = \ell(\Tor^A_j(H^i_{\frakm}(A), k)_{\perf}) = \dim_{\FFp}(H^{-j} \circ i^* \circ {}^p H^i(\underline\FF_{p,A})),
\]
where $i^*$ indicates the pullback to the residue field of $A$.
\end{remark}

\subsection{Relation to Gabber's perverse $t$-structure}
\label{ss:GabberPervComp}

We compare the perverse $t$-structure on constructible $\FFp$-complexes constructed in Corollary~\ref{PervHolQuot} (using the perverse $t$-structure on coherent sheaves) with the more classical one (in terms of support conditions on stalks and costalks) studied by Gabber \cite{Gabbert} which we recall first, see also \cite[Section 11.5.2]{EK}.
\begin{notation}
We work with schemes of finite type over a noetherian local $F$-finite ring $A$. Any such scheme $X$ has a dualizing complex and comes endowed with a perverse $t$-structure on $D^b_{coh}(X)$ as in \S\ref{ss:PervCoh}.
\end{notation}
For a point $x \in X$ denote by $i_x \colon x \to X$ its inclusion into $X$. Define the subcategories of $D^b_c(X_{\et}, \FFp)$
\begin{align*}
{}^pD^{\le 0} &= \{ M \in D^b_c(X_{\et}, \FFp) \mid H^j(i^*_xM) = 0 \text{ for all } x \in X\text{ and }j > -d(x)\} \\
{}^pD^{\ge 0} &= \{ M \in D^b_c(X_{\et}, \FFp) \mid H^j(i^!_xM) = 0 \text{ for all } x \in X\text{ and }j < -d(x)\}
\end{align*}
where $d(x)$ is as in Remark~\ref{DualizingComplexLocalRing}. It is shown in \cite[Theorem~10.3]{Gabbert} that these underlie a unique $t$-structure on $D^b_c(X_{\et}, \FFp)$. We shall need the following standard fact:
\begin{lemma}
\label{DualCohSupp}
Let $X$ be a finite type $A$-scheme. For $N \in D^{b, \ge 0}_{coh}(X)$ and $x \in X$, we have $\DD_X(N)_x \in D^{\le -d(x)}$.
\end{lemma}

\begin{proof}
It suffices to check the statement after completion at the closed point of $\calO_{X,x}$. By \cite[\href{https://stacks.math.columbia.edu/tag/0A6Y}{Tag~0A6Y}]{stacks-project} we then have
\[
\widehat{\DD_X(N)_x} = \RHom_{\widehat{\calO_{X,x}}}(\widehat{N_x},\widehat{\omega_{X,x}^\bullet}) \simeq \RHom_{\widehat{\calO_{X,x}}}(R\Gamma_x(N_x), R\Gamma_x(\omega_{X,x}^\bullet)) \simeq \RHom_{\widehat{\calO_{X,x}}}(R\Gamma_x(N_x), E_x[d(x)]),
\]
where $\widehat{\ }$ denotes completion at the closed point $x$, and $E_x \simeq R\Gamma_x(\omega_{X,x})[-d(x)]$ is an injective hull of the residue field at $x$. As $E_x$ is injective, $\Hom_{\widehat{\calO_{X,x}}}(-,E_x)$ is exact. Moreover, as $N \in D^{b,\ge 0}_{coh}(X)$, we also have $R\Gamma_x(N_x) \in D^{\ge 0}$. The above formula then shows that $\widehat{\DD_X(N)_x} \in D^{\le -d(x)}$, as wanted.
\end{proof}
We can now show the equivalence of the perverse $t$-structure on $D^b_c(X_{\et},\FFp)$ defined in \S\ref{ss:PervCoh} and the one defined by Gabber recalled above.

\begin{theorem}
\label{CompareGabber}
Let $X$ be a finite type $A$-scheme. The following two $t$-structures on $D^b_c(X_{\et}, \FFp)$ coincide:
\begin{enumerate}[\quad\rm(1)]
\item The perverse $t$-structure from Definition~\ref{PervtDef}.

\item The $t$-structure on $D^b_c(X_{\et},\FFp)$ constructed in \cite[Theorem~10.3]{Gabbert} using the middle perversity function $p(x) = -d(x)$, just recalled above.
\end{enumerate}
\end{theorem}

\begin{proof}
Since any $t$-structure is determined by its $D^{\le 0}$ part, it suffices to show that the $D^{\le 0}$ parts of both $t$-structures coincide. Unwinding definitions, this amounts to checking that the following full subcategories of $D^b_c(X_{\et},\FFp)$ are identical:

(1) The full subcategory $\calC_1$ spanned by $G \in D^b_c(X_{\et}, \FFp)$ such that $G$ can be written as $\Sol(M)$, where $(M,\varphi_M) \in D^b_{coh}(X[F])$ with $M \in \DD_X(D^{\ge 0}_{coh}(X))$. This is justified since $^p D^{\le 0}_c(X_{\et},\FFp)$ in the sense of Definition~\ref{PervtDef} corresponds via Riemann-Hilbert to $^p D^{\le 0}_{hol}(X[F])$ which, in turn,
corresponds to $^p D^{\le 0}_{coh}(X[F])$ via perfection, see Corollary~\ref{PervHolQuot}. By Definition~\ref{DefPervCoh} of the coherent $t$-structure, membership there can be checked after applying $\DD_X$.

(2) The full subcategory $\calC_2$ spanned by $G' \in D^b_c(X_{\et}, \FFp)$ such that for any $x \in X$ and chosen geometric point $\overline{x} \to X$ based at $x$, we have $G'_{\overline{x}} \in D^{\le p(x)}$. To see that this describes ${}^pD^{\le 0}$ in the sense of \cite[Theorem~10.3]{Gabbert}, use that for the \'etale topology the vanishing of $i_x^*G'$ is equivalent to the vanishing of the geometric stalk $G'_{\overline{x}}$.

\emph{Verifying $\calC_1 \subset \calC_2$}: Pick $G$ and $M$ as in the definition of $\calC_1$; we shall check the condition defining $\calC_2$. Using the exact triangle
\[
G \to M \xrightarrow{\ \varphi_M-1\ } M
\]
in $D(X_{\et}, \FFp)$, it suffices to show the following: for $x \in X$ and chosen geometric point $\overline{x} \to X$ based at $x$,
\begin{enumerate}[\rm(i)]
\item The stalk $M_{\overline{x}} \colonequals \DD_X(N)_{\overline{x}}$ lies in $D^{\le p(x)}$, and

\item The map $M_{\overline{x}} \to M_{\overline{x}}$ obtained by taking stalks at $\overline{x}$ of $\varphi_M-1$ is surjective on applying $H^{p(x)}(-)$.
\end{enumerate}

The first condition is verified by Lemma~\ref{DualCohSupp}, and the second is in fact true for all cohomology groups: for a finitely generated module $M'$ over a strictly henselian local $\FFp$-algebra $R$, and a $p$-linear map $\varphi\colon M' \to M'$, the map $\varphi-1$ is surjective; see \cite[Proposition~10.1~(3)]{Gabbert}.

\emph{Verifying $\calC_2 \subset \calC_1$}: The category $\calC_2$ is generated under direct summands, extensions and cones by objects of the form $j_! L[i]$, where $j\colon U \to X$ is a locally closed immersion with $U$ regular and equidimensional, $L$ is a locally constant sheaf of constant rank on $U$, and $i$ is an integer $\ge d(U)$\footnote{This is a standard fact for which we could not find a reference, so we give the argument. Let $\calC' \subset \calC_2$ be the subcategory generated by objects of the form $j_! L$ as above. To see that $\calC'=\calC_2$, fix $F \in \calC_2$. We shall show $F \in \calC'$ by induction on the dimension of its support. The dimension $0$ case is clear. In general, by filtering $F$ by its cohomology sheaves, it suffices to treat each cohomology group separately, so we may assume $F=G[i]$, where $G$ is a constructible sheaf supported on a closed set $Z \subset X$ with $\dim Z \le i$. There is then a dense open regular $j\colon U \subset Z$ such that $G|_U$ is a local system, whence $j_!(F|_U) \in \calC'$ (as we can write $U$ as a disjoint union of regular equidimensional schemes of dimension $\le i$). Now if $k\colon W=Z-U \to Z$ is the complement of $U$, then it is enough to show that $k_* (F|_{Z-U}) = \left (G/j_!(G|_U) \right)[i]$ lies in $\calC'$; this follows by induction since $\dim(Z-U) < \dim Z$.}. As $\calC_1$ is closed under direct summands, extensions and cones, it suffices to show that each generating object $j_! L[i]$ as before lies in $\calC_1$. By the finite pushforward compatibility of both Grothendieck duality and the $p(-)$ function, we may replace $X$ with $\overline{\Supp L}$ to assume that $U$ is a dense open subset of $X$, which we may assume reduced. As $L$ is a local system, we have $L=\Sol(M)$ for a vector bundle $M$ on $U$ equipped with a Frobenius structure $\phi_M\colon M \to F_* M$. Choose an extension $N \in \mathrm{Coh}(X)$ of $M$ with a Frobenius structure $\phi_N\colon N \to F_* N$ lifting $\phi_M$; for instance, after choosing a coherent extension $N_0$ of $M$, we can take $N = I_Z^n N_0$ for $n \gg 0$ to ensure that $\phi_M$ preserves $N$. Then the submodule $I_Z N \subset N$ is naturally a Frobenius submodule of $N$ and we must have $j_! L = \mathrm{Sol}(I_Z N)$ by Example~\ref{exRHshriekext}. Thus, to prove the claim, it is enough to show that $I_Z N[i] \in {}^p D^{\le 0}(X)$ or equivalently that $\DD_X(I_Z N)[-i] \in D^{\ge 0}$. Since $I_Z N$ is a honest coherent sheaf, noting that $\DD_X = \RHom_X(-,\omega_X^\bullet)$ and $i \ge d(U)=d(X)$, it suffices to know that $\omega_X^\bullet \in D^{\ge -d(X)}$. But this follows from our normalization of the dualizing complex in Remark~\ref{DualizingComplexLocalRing} and the completely general \cite[Lemma in \S 2]{Gabbert} (applied with $Z=X$).
\end{proof}

\begin{remark}
Fix a finite type $A$-scheme $X$. One consequence of Theorem~\ref{CompareGabber} and Gabber's work in \cite{Gabbert} is that the perverse $t$-structure on $D^b_c(X_{\et};\FFp)$ from Definition~\ref{PervtDef} admits a natural extension to a perverse $t$-structure on $D(X_{\et};\FFp)$, characterized by the support and cosupport conditions as in \cite[\S 2]{Gabbert} with $p(x)=-d(x)$. We also refer to this $t$-structure on $D(X_{\et};\FFp)$ as the perverse $t$-structure, and its heart is denoted by $\Perv(X;\FFp)$. Thus, we have
\[
\Perv(X;\FFp) \cap D^b_c(X_{\et};\FFp) = \Perv_c(X;\FFp)
\qquad \text{and} \qquad
\mathrm{Ind}\left(\Perv_c(X;\FFp)\right) \simeq \Perv(X;\FFp);
\]
see \cite[Remark~7.2]{Gabbert} for the last isomorphism.
\end{remark}

\begin{corollary}
Let $X$ be a finite type $A$-scheme. The category $\Perv_c(X;\FFp)$ is both noetherian and artinian.
\end{corollary}

\begin{proof}
This follows from Gabber's \cite[Corollary~12.4]{Gabbert}.
\end{proof}

\subsection{Reminders on compactly supported cohomology}

For any finite type separated maps of noetherian schemes, there is a well-developed notion of compactly supported pushforwards for torsion \'etale sheaves. In this subsection, we recall the definition and basic properties of this notion next.

\begin{notation}
In this subsection, all schemes are noetherian, and all morphisms are assumed to be separated and finite type. Fix a finite coefficient ring $\Lambda$.
\end{notation}

\begin{definition}[Compactly supported cohomology]
Given $f\colon X \to Y$, define
\[
Rf_!\colon D^b_c(X;\Lambda) \to D^b_c(Y;\Lambda)
\]
as $R\overline{f}_* \circ j_!$, where $X \xhookrightarrow{\ j\ } \overline{X} \xrightarrow{\ \overline{f}\ } Y$ is any factorization of $f$ as an open immersion $j$ followed by a proper morphism $f$. This functor is independent of the choice of $j$ and $\overline{f}$ by the proper base change theorem~\cite[\href{https://stacks.math.columbia.edu/tag/0F7I}{Tag~0F7I}]{stacks-project}. When $Y = \Spec k$ for an algebraically closed field $k$, we also write $R\Gamma_c(X;-)$ instead of $Rf_!$.
\end{definition}

\begin{lemma}[Excision]
Fix a map $f\colon X \to Y$, and an open $U \subset X$ with complement $Z$. Then for any $F \in D^b_c(X;\Lambda)$, we have a canonical exact triangle
\[
Rg_!(F|_U) \to Rf_! F \to Rh_!(F|_Z)
\]
in $D^b_c(Y; \Lambda)$, where $g\colon U \to Y$ and $h\colon Z \to Y$ are the maps induced by $f$. (We call this triangle the excision triangle associated to $f$ and the open immersion $U \subset X$.)
\end{lemma}

\begin{proof}
See~\cite[\href{https://stacks.math.columbia.edu/tag/0GKP}{Tag~0GKP}]{stacks-project}.
\end{proof}

\begin{proposition}[Proper base change]
Consider a cartesian diagram
\[
\xymatrix{X' \ar[r]^{g'} \ar[d]^{f'} & X \ar[d]^f \\
Y' \ar[r]^{g} & Y.}
\]
Then there is a canonical isomorphism
\[
g^* Rf_!(-) \simeq Rf'_! g'^*(-)
\]
of functors $D^b_c(X;\Lambda) \to D^b_c(Y';\Lambda)$.
\end{proposition}

\begin{proof}
See~\cite[\href{https://stacks.math.columbia.edu/tag/0F7L}{Tag~0F7L}]{stacks-project}.
\end{proof}

\begin{lemma}[Projection formula]
For a map $f\colon X \to Y$ and $F \in D^b_c(X;\Lambda)$, $G \in D^b_c(Y;\Lambda)$, there is a canonical isomorphism
\[
G \otimes Rf_! F \simeq Rf_!(f^* G \otimes F)
\]
in $D^b_c(Y;\Lambda)$.
\end{lemma}

\begin{proof}
See~\cite[\href{https://stacks.math.columbia.edu/tag/0GL5}{Tag~0GL5}]{stacks-project}.
\end{proof}

Let us discuss two calculations of compactly supported cohomology in the case of interest, i.e., with $\FFp$-coefficients in characteristic $p$.

\begin{example}[Affine space fibrations]
\label{AffineSpaceCoh}
Say $f\colon X \to Y$ is a map of $\FFp$-schemes whose fibers are affine spaces $\AA^n_k$ of positive dimension, and assume $\Lambda = \FFp$. Then $Rf_! \FFp \simeq 0$. To see this, by proper base change, it is enough to show that $R\Gamma_c(\AA^n_k; \FFp) = 0$ for $n > 0$ over any field $k$. Using the excision triangle for the compactification $\AA^n_k \subset \PP^n_k$ with boundary a hyperplane $\PP^{n-1}_k$, it suffices to show that
\[
R\Gamma(\PP^n; \FFp) \simeq R\Gamma(\PP^{n-1};\FFp) \qquad \text{for } n > 0.
\]
This follows from the Artin-Schreier exact sequence and the calculation that $k \simeq R\Gamma_c(\PP^m_k; \calO_{\PP^m_k})$ for $m \ge 0$.
\end{example}

\begin{example}[$\GG_m$-torsors]
\label{GmCohCompSupp}
Let $f\colon X \to Y$ be a $\GG_m$-torsor of $\FFp$-schemes, and assume that $\Lambda = \FFp$. Then~$Rf_! \FFp \simeq \FFp[-1]$. To see this, let $X \xhookrightarrow{\ i\ } \overline{X} \xrightarrow{\ \overline{f}\ } X$ be the standard partial compactification of $f$, so $\overline{f}$ is the total space of a line bundle over $Y$, and $Z \colonequals (\overline{X} - X)_{red}$ maps isomorphically to $Y$ via $\overline{f}$. Using Example~\ref{AffineSpaceCoh} and excision, we learn that
\[
Rf_! \FF_{\!p,X} \simeq R\overline{f}_! \FF_{\!p, Z}[-1] \simeq \FF_{\!p,Y}[-1].
\]
\end{example}

\section{Applications of the Riemann-Hilbert functor in characteristic $p$}

In the bulk of this section, we use the Riemann-Hilbert correspondence in characteristic $p$ as well as the theory of perverse sheaves as developed in \S\ref{sec:RHPervFp} to obtain the applications mentioned in \S\ref{section:introduction}. We point out that the results in \S\ref{ss:ICFRat} and \S\ref{ss:HHSm} have also been discovered independently by Cass and Louren{\c{c}}o \cite{CassLau}, and that of \S\ref{ss:ArtinVanishing} by Baudin \cite{BaudinDual}.

All functors are derived, unless otherwise specified.

\subsection{The perverse Artin vanishing theorem}
\label{ss:ArtinVanishing}

The goal of this subsection is to prove a surprisingly strong version of the Artin vanishing theorem (Theorem~\ref{PerverseArtin}) of $f_!$ for an affine map.

\begin{notation}
Fix a noetherian local $F$-finite $\FFp$-algebra $A$. We work with finite type $A$-schemes, so there is a well-defined perverse $t$-structure on $D^b_c(X;\FFp)$ for any scheme $X$ under consideration.
\end{notation}

Our first basic observation is that the perverse $t$-structure behaves well under smooth pullback, much like the classical setting, i.e., in the case of $\ZZ/\ell$-coefficients, where $\ell$ is invertible.

\begin{lemma}[Smooth pullbacks are perverse $t$-exact, up to a shift]
\label{smoothpullbackperverse}
Let $f\colon X \to Y$ be a smooth morphism of relative dimension $d$. Then $F \mapsto f^*F[d]$ is $t$-exact for the perverse $t$-structure.
\end{lemma}

\begin{proof}
Unwinding definitions, it suffices to show that $M \mapsto f^* M[d]$ is exact for the perverse $t$-structure on $D^b_{coh}(Y)$. In other words, for $N \in \mathrm{Coh}(Y)$, we need to check that $\DD_X\left(f^* \DD_Y(N)[d]\right)$ lies in $\mathrm{Coh}(X)$. But duality theory and the smoothness of $f$ show that $\DD_X\circ f^*\circ \DD_Y=f^!=f^* \otimes \omega_f[d]$, where $\omega_f$ is the relative dualizing line bundle of the smooth morphism $f$, see \cite[\S 2]{Conrad}. Hence the claim follows from flatness of $f$ since $f$ is smooth.
\end{proof}

Our key new observation is that the classical perverse Artin vanishing theorem acquires a stronger form in the case of $\FFp$-coefficients: we actually get $t$-exactness rather than merely half $t$-exactness.

\begin{theorem}[Perverse acyclicity for affine morphisms]
\label{PerverseArtin}
Let $f\colon X \to Y$ be an affine morphism between finite type $A$-schemes. Then $f_!\colon D^b_c(X;\FFp) \to D^b_c(Y;\FFp)$ is $t$-exact for the perverse $t$-structure.
\end{theorem}

\begin{proof}
We may assume that both $X$ and $Y$ are affine. Fix $F \in \Perv_c(X;\FFp)$. Under Riemann-Hilbert, such an~$F$ has the form $\Sol(M)$ for an $M \in D^b_{hol}(X[F])$ that arises as the perfection of Frobenius module $(M_0,\phi_{M_0})$ for some $M_0 \in \Perv_{coh}(X)$. Choose a diagram
\[
X \xrightarrow{\ j\ } \overline{X} \xrightarrow{\ \overline{f}\ } Y
\]
where $\overline{f}$ is a proper morphism, and $j$ is a dense affine open immersion with complement $H \colonequals (\overline{X} - X)_{red}$ being a Cartier divisor that is relatively ample over $Y$. Extend the pair $(M_0,\phi_{M_0})$ to a pair $(N_0,\phi_{N_0})$ with $N_0 \in \Perv_{coh}(\overline{X})$: one can simply pick any extension $(\overline{M}, \phi_{\overline{M}})$ with $\overline{M} \in D^b_{coh}(\overline{X})$ and then set $N_0 \colonequals {}^p H^0(\overline{M})$ (using Lemma~\ref{PervFinitePushforward} to ensure compatibility of formation of perverse cohomology sheaves with $F_*$).

If $I = \calO_{\overline{X}}(-H)$ denotes the ideal sheaf of the boundary, the pair $(I \otimes N_0, \phi \otimes \phi_{M_0})$ is a Frobenius module on $\overline{X}$ whose perfection corresponds to $j_! F$ under Riemann-Hilbert, see Example~\ref{exRHshriekext}. Thus, our task is to show $\overline{f}_*(I \otimes N_0)_{\perf}$ is a perverse holonomic Frobenius module on $Y$. Since we are allowed to replace $I$ with a power without affecting the perfection, it is enough to show that $\overline{f}_*(I^n \otimes N_0)_{\perf}$ is a perverse holonomic Frobenius module for $n \gg 0$; we shall check the stronger statement that for $n \gg 0$, we have
\[
{}^p H^i(\overline{f}_* (I^n \otimes N_0)) = 0 \qquad \text{for } i \neq 0.
\]
By duality, this is equivalent to showing that
\[
H^i(\overline{f}_*( \DD_{\overline{X}}(N_0) \otimes I^{-n})) = 0 \qquad \text{for } i \neq 0
\]
and $n \gg 0$. But the perversity of $N_0$ ensures that $ \DD_{\overline{X}}(N_0)$ is a coherent sheaf on $\overline{X}$; since $I^{-1} = \calO_{\overline{X}}(1)$ is ample, the claim now follows from Serre vanishing.
\end{proof}

\begin{corollary}[Strong Artin vanishing]
Let $X$ be an affine variety over a field $k$ of characteristic $p$. For any~$F \in \Perv_c(X;\FFp)$, we have $H^i_c(X; F) = 0$ for $i \neq 0$.
\end{corollary}

\begin{corollary}[Semiperversity of pushforwards]
\label{SemiperversePushforward}
For a separated morphism $f\colon X \to Y$ of finite type schemes over a noetherian $F$-finite ring $A$, the functor $f_!$ is right $t$-exact, i.e., it takes ${}^p D^{\le 0}(X)$ to ${}^p D^{\le 0}(Y)$.
\end{corollary}

\begin{proof}
If $X$ is affine, then the claim follows from Theorem~\ref{PerverseArtin}. In general, one picks a dense affine open set~$k\colon U \subset X$ (with closed complement $i \colon Z \to X$), and uses the associated excision triangle 
\[
k_!k^*F \to F \to i_!i^*F \to[+1].
\]
If $F$ lies in ${}^p D^{\le 0}(X)$ then so do the other two terms by the $t$-exactness of $k^*$ and $i_!$ and 
the right $t$-exactness of $k_!$ and $i^*$. Now apply $j_!$ to this triangle and use the affine case for the first, and Noetherian induction on the last term to see that these terms lie in ${}^p D^{\le 0}(Y)$. Hence the middle term $j_!F \in {}^p D^{\le 0}(X)$ as well.
\end{proof}

\begin{remark}
Theorem~\ref{PerverseArtin} as well as Corollary~\ref{SemiperversePushforward} are specific to our setting, i.e., to $\FFp$-sheaves in characteristic $p$, and fail completely in the classical setting. For example, if $X$ is a smooth affine complex variety of dimension $d$, then $H^{2d}_c(X; \CC) \simeq \CC \neq 0$.
\end{remark}

\begin{remark}
A similar vanishing has been observed for compactly supported Witt vector cohomology of smooth affine varieties with $\QQ_p$-coefficients by Berthelot-Bloch-Esnault \cite{BerthBlochEsn}. Neither result implies the other: ours is integral but only refers to cohomology spaces where Frobenius acts as the identity (so the slope is $0$), while theirs is rational but refers to cohomology spaces where Frobenius is allowed to act with slopes in $[0,1)$. It would be interesting to find a common generalization.
\end{remark}

\subsection{Intersection cohomology and $F$-rationality}
\label{ss:ICFRat}

The goal of this subsection is to explain what the intersection cohomology complex is, and why the property of being an intersection cohomology manifold is closely related to $F$-rationality.

\begin{notation}
Fix a noetherian local $F$-finite $\FFp$-algebra $A$. We work with finite type $A$-schemes, so there is a well-defined perverse $t$-structure on $D^b_c(X;\FFp)$, for any scheme $X$ under consideration.
\end{notation}

Let us begin with the following variant of the Artin vanishing for $*$-pushforwards for open affine immersions, complementing Theorem~\ref{PerverseArtin}.

\begin{proposition}
Let $j\colon U \to X$ be an affine open immersion of finite type $A$-schemes. Then
\[
j_*\colon D^b(U_{\et}, \FFp) \to D^b(X_{\et}, \FFp)
\]
is perverse $t$-exact.
\end{proposition}

We warn the reader that $j_*$ does not preserve constructibility.

\begin{proof}
As $i_x^! Rj_* = 0$ for all points $x \in X-U$ and $i_x^! Rj_* = i_x^!$ for all $x \in U$, it is immediate from Gabber's description of the $t$-structure that $Rj_*$ is left $t$-exact. For right $t$-exactness, we may filter to reduce to showing the following: given a constructible $\FFp$-sheaf $G$ on $U$ and an integer $i$ such that $i \ge d(u)$ for all $u \in \Supp G$ (whence $G[i] \in {}^p D^{\le 0}$), the pushforward $Rj_* G[i]$ lies in ${}^p D^{\le 0}$. As $j$ is an affine morphism, general results on the $p$-cohomological dimension of affine schemes show that $R^i j_* G = 0$ for $i \neq 0,1$; for example, \cite[Theorem~11.1]{BhattScholze} shows this statement for perfect $\FFp$-schemes, which implies the general case by the topological invariance of the étale site. Moreover, by construction, $R^0 j_* G$ is supported on $\overline{\Supp G}$ (with the closure computed in $X$) while~$R^1 j_* G$ is supported on $\overline{\Supp G}-U \subset X$. We must check that $(R^0 j_* G)[i]$ and $(R^1 j_* G)[i-1]$ lie in~${}^p D^{\le 0}$. The claim for $(R^0 j_* G)[i]$ is immediate from the fact that
\[
\max\{d(x) \mid x \in \overline{\Supp G}\} = \max\{d(u) \mid u \in \Supp G\} \le i.
\]
For $(R^1 j_*G)[i-1]$, it suffices to show that for any $x \in \overline{\Supp G}-U$, we have $i-1 \ge d(x)$. But this is clear: there is a non-trivial specialization $u \rightsquigarrow x$ for $u \in \Supp G$ as $\Supp G$ is dense $\overline{\Supp G}$, and the $d(-)$ function drops by at least $1$ as one moves up along a non-trivial specialization (see Remark~\ref{DualizingComplexLocalRing}).
\end{proof}

Next, we aim to discuss intermediate extensions. For this purpose, we shall need the following lemmas:

\begin{lemma}
\label{SubQuotPush}
Fix a finite type $A$-scheme $X$. For any open immersion $j\colon U \to X$ and $M \in \Perv_c(U;\FFp)$, the object ${}^p H^0(j_! M)$ (resp. ${}^p H^0(j_* M)$) admits no non-trivial quotients (resp.\ subobjects) supported on $X-U$.
\end{lemma}

\begin{proof}
This is immediate from adjunction as $i^* j_! = 0 $ and $i^! j_* = 0$, see \cite[Proposition 8.2.7]{HTTDmodules} where this argument is carried out.
\end{proof}

\begin{lemma}
\label{FRatEtale}
Let $(Y,y)$ be a strictly henselian local scheme of dimension $d$. If $Y$ is Cohen-Macaulay, then~$H^i_y(Y,\FFp) = 0$ for $i < d$. If $Y$ is regular and $d > 0$, then we also have $H^d_y(Y, \FFp) = 0$. More generally, the same holds true if $Y$ is $F$-rational.
\end{lemma}

\begin{proof}
The first claim is immediate from the Artin-Schreier sequence. For the second, by the same argument, it suffices to show the following: if $(R,\frakm)$ is a strictly henselian regular local ring of dimension $d > 0$, then the map $\varphi-1\colon H^d_{\frakm}(R) \to H^d_{\frakm}(R)$ is injective. This can be checked after completion, where it is easy from a grading argument. Finally, the last part follows similarly as $H^d_{\frakm}(R)$ contains no nonzero Frobenius fixed elements if $R$ is $F$-rational.
\end{proof}

We can now define the intersection cohomology complex:

\begin{proposition}
\label{IntermediateExt}
Let $X$ be a separated finite type $A$-scheme. Let $j\colon U \to X$ be a dense affine open immersion with $U_{red}$ regular; write $d$ for the (locally constant) function on $U$ given by sending each connected component $V \subset U$ to $d(V)$. Consider the object
\[
j_{!*}(\FF_{\!p,U}[d]) \colonequals \image\left(j_! (\FF_{\!p,U}[d]) \to j_* (\FF_{\!p,U}[d])\right) \in \Perv_c(X;\FFp).
\]
\begin{enumerate}[\quad\rm(1)]
\item $j_{!*} (\FF_{\!p,U}[d])$ is canonically independent of the choice of $U$.
\item If $U$ is connected, then $ j_{!*} (\FF_{\!p,U}[d])$ is simple in $\Perv_c(X;\FFp)$.
\end{enumerate}
\end{proposition}

The object $\mathrm{IC}_{X,\FFp} \colonequals j_{!*} (\FF_{\!p,U}[d]) \in \Perv_c(X;\FFp)$, defined by (1) above, is the \emph{intersection cohomology sheaf} on $X$.

\begin{proof}
Let us first show the claims when $X$ is itself regular and connected of constant normalized dimension~$d$. In this case, we must show the following:
\begin{enumerate}[\quad\rm(1)]
\item Given a dense affine open immersion $j\colon U \to X$, we have $j_{!*} \FF_{\!p,U}[d] \simeq \FF_{\!p,X}[d]$.

\item $\FF_{\!p,X}[d]$ is simple in $\Perv_c(X;\FFp)$.
\end{enumerate}
For (1), it suffices to show that the natural maps of perverse sheaves on $X$,
\[
\alpha\colon j_! \FF_{\!p,U}[d] \to \FF_{\!p,X}[d]
\qquad \text{and} \qquad
\beta\colon \FF_{\!p,X}[d] \to j_* \FF_{\!p,U}[d],
\]
are respectively surjective and injective. Note that
\[
\mathrm{fib}(\alpha) = i_* \FF_{\!p,Z}[d-1]
\]
where $i\colon Z \to X$ is the closed complement of $U$; as the normalized dimension of $Z$ is $\le d-1$, we have $i_* \FF_{\!p,Z}[d-1] \in {}^p D^{\le 0}$ (by definition of the perverse $t$-structure), showing that $\alpha$ is indeed a surjection of perverse sheaves. Similarly,
\[
\mathrm{cone}(\beta) = i_* i^! \FF_{\!p,X}[d+1],
\]
so it suffices to show that for any $x \in X-Z$, we have
\[
i_x^!(i_* i^! \FF_{\!p,X}[d+1]) \in D^{\ge -d(x)}.
\]
Simplifying, we must show that $i_x^! \FF_{\!p,X} \in D^{\ge d-d(x)+1}$, i.e., for a strict henselisation $X_x^{sh}$ of $X$ at $x$, we have
\[
R\Gamma_x(X_x^{sh}, \FFp) \in D^{\ge d-d(x)+1}.
\]
As $d-d(x) = \dim(\calO_{X,x})$, the claim follows from Lemma~\ref{FRatEtale}.

For (2), pick a nonzero subobject $N \subset \FFp[d]$ in $\Perv_c(X;\FFp)$. It follows from (1) and Lemma~\ref{SubQuotPush} that $N|_U \neq 0$ for any non-empty open subset of $U \subset X$. In particular, taking $j\colon U \subset X$ to be the dense open lisse locus of $N$, we learn that $N|_U$ is a non-trivial lisse subsheaf of $\FFp[d]$, so $N|_U = \FFp[d]$. But then the composition $j_!(N|_U) \to N \subset \FFp[d]$ is surjective by (1), so $N = \FFp[d]$, as wanted.

We now handle the general case. Given $j\colon U \to X$ as in the statement of the lemma, (1) follows from the analysis of the regular case above applied to a dense open immersion $V \subset U$, while (2) follows by the same analysis used in the regular case above.
\end{proof}

\begin{remark}[The canonical map from the constant sheaf to the $IC$-sheaf]
\label{NatMapToIC}
Fix notation as in Proposition~\ref{IntermediateExt}. We claim that the natural map $\FF_{\!p,X}[d] \to j_* \FF_{\!p,U}[d]$ factors uniquely over
\[
\mathrm{IC}_{X,\FFp} = j_{!*} \FF_{\!p,U}[d] \subset j_* \FF_{\!p,U}[d],
\]
thus giving a natural map $\FF_{\!p,X}[d] \to \mathrm{IC}_{X,\FFp}$ (uniquely determined by the requirement that it is the identity over $U$). To see this factorization and its uniqueness, observe that as $\FF_{\!p,X}[d] \in {}^p D^{\le 0}$ and
\[
j_{!*} \FF_{\!p,U}[d], j_* \FF_{\!p,U}[d] \in \Perv_c(X;\FFp),
\]
the uniqueness will be automatic from the injectivity of $j_{!*} \FF_{\!p,U}[d] \subset j_* \FF_{\!p,U}[d]$ once we have existence. For existence, observe that the quotient $ j_*\FF_{\!p,U}[d] / j_{!*} \FF_{\!p,U}[d]$ is supported on $i\colon Z \subset X$, so has the form $i_* Q$ for some perverse sheaf $Q$ on $Z$. Now any map $\FF_{\!p,X}[d] \to i_* Q$ on $X$ corresponds by adjunction to a map $\FF_{\!p,Z}[d] \to Q$ on $Z$. As $Q$ is perverse, this map factors uniquely over ${}^p H^0(\FF_{\!p,Z}[d])$. But the normalized dimension of $Z$ is $\le d-1$, so $\FF_{\!p,Z}[d] \in {}^p D^{\le -1}$, and thus ${}^p H^0(\FF_{\!p,Z}[d]) = 0$.
\end{remark}

\begin{remark}[The $\mathrm{IC}$-sheaf cannot be a summand of the constant sheaf except when they agree]
The canonical surjective map $\eta\colon {}^p H^0(\FF_{\!p,X}[d]) \to \mathrm{IC}_{X,\FFp}$ coming from Remark~\ref{NatMapToIC} is split surjective if and only if it is an isomorphism. In particular, $\mathrm{IC}$-sheaf $\mathrm{IC}_{X,\FFp}$ cannot be a summand of $\FF_{\!p,X}[d]$ except when the two are equal. To see this, assume that the map $\eta\colon {}^p H^0(\FF_{\!p,X}[d]) \to \mathrm{IC}_{X,\FFp}$ has a section. Then $\ker\eta$ is also a quotient of ${}^p H^0(\FF_{\!p,X}[d])$, say via a map $\tau\colon {}^p H^0(\FF_{\!p,X}[d]) \to \ker\eta$. By adjunction, $\tau$ is uniquely determined by the composition
\[
\mu\colon \FF_{\!p,X}[d] \to {}^p H^0(\FF_{\!p,X}[d]) \to \ker\eta.
\]
But $\ker\eta$ has the form $i_* K$ for a closed immersion $i\colon Y \into X$ with $d(Y) < d(X)$. By adjunction, $\mu$ is determined by the adjoint map $\FF_{\!p,Y}[d] \to K$, which must vanish as in Remark~\ref{NatMapToIC}: we have $\FF_{\!p,Y}[d] \in {}^p D^{\le -1}(Y)$ as $d(Y) < d$. This implies that $\tau$ also vanishes, and thus $\ker\eta= 0$ as $\tau$ is surjective.
\end{remark}

With notation as in Proposition~\ref{IntermediateExt}, if $X$ has $F$-rational singularities, then $\mathrm{IC}_{X,\FFp} \simeq \FF_{\!p,X}[d]$; this follows from the proof of Proposition~\ref{IntermediateExt}, and was also observed by Cass \cite[Theorem~1.7]{CassPerv}. A precursor of this result in terms of the intersection homology $D$-module can be found in \cite[\S 4]{BliIntHom}. More generally, we have:

\begin{corollary}[$IC$-manifolds and $F$-rationality up to nilpotents]
Let $d=\dim A$ and assume $A$ is a domain. The following statements are equivalent:
\begin{enumerate}[\quad\rm(1)]
\item The natural map $\FF_{\!p,A}[d] \to \mathrm{IC}_{A,\FFp}$ (Remark~\ref{NatMapToIC}) is an isomorphism.

\item $\FF_{\!p,A}[d]$ is perverse and simple.

\item $A$ is ``$F$-rational up to nilpotents,'' i.e., $H^{<d}_{\frakm}(A)$ is Frobenius nilpotent, and any proper Frobenius stable $A$-submodule of $H^d_{\frakm}(A)$ is also Frobenius nilpotent.
\end{enumerate}
\end{corollary}

\begin{proof}
The equivalence of (1) and (2) is clear from the simplicity of $\mathrm{IC}_{A,\FFp}$ as a perverse sheaf (see Proposition~\ref{IntermediateExt}) and the surjectivity of $\FF_{\!p,A}[d] \to \mathrm{IC}_{A,\FFp}[d]$. Moreover, Remark~\ref{ConstantCMPerverse} shows that $\FF_{\!p,A}[d]$ is perverse exactly when $H^{<d}_\frakm(A)$ is Frobenius nilpotent. The rest of the equivalence of (2) and (3) then follows from the equivalence of (1) and (4) in Corollary~\ref{PerverseCharpLocal}, applied using Example~\ref{ExPervFrobLocal} (1) and (4).
\end{proof}

\subsection{Embedding independence of Lyubeznik complexes}
\label{ss:EmbIndepLyuCharp}

The goal of this subsection is to prove a characteristic $p$ analog of the results in \S\ref{ss:EmbIndepLyuChar0}, recovering results of Zhang \cite{ZhangLyuIndep}. Surprisingly, the results here are much better than in characteristic $0$: we do not need smoothness assumptions.

\begin{notation}
Let $Z \subset \PP^n$ be a projective scheme over a perfect field $k$ of characteristic $p$. Let $Y$ be the affine cone over $Z \subset \PP^n$, let $k\colon \Spec k \into Y$ be the origin with complement $j\colon U \into Y$, let $\pi\colon \widetilde{Y} \to Y$ be the blowup at the origin. Then there is a map $f\colon \widetilde{Y} \to Z$ exhibiting $Z$ as the total space of the ample line bundle $\calO_Z(-1)$.
\end{notation}

\begin{proposition}
\label{GradedGR}
The functor $D^b_c(Z;\FFp) \to D^b_c(Y; \FFp)$ given by $M \mapsto \pi_* f^* M[1]$ is $t$-exact for the perverse $t$-structure. Moreover, for any $M \in D^b_c(Z;\FFp)$, we have $R\Gamma_c(Y, \pi_* f^* M) = 0$.
\end{proposition}

\begin{proof}
The second assertion follows from the string of equalities
\[
R\Gamma_c(Y, \pi_* f^* M[1]) \simeq R\Gamma_c(\widetilde{Y}, f^* M[1]) \simeq R\Gamma(Z, M[1] \otimes f_! \FFp) = 0,
\]
where the second equality uses the projection formula, while last equality follows as $f_! \FFp = 0$ by Example~\ref{AffineSpaceCoh}. For the first assertion, say $M$ is perverse on $Z$. By Lemma~\ref{smoothpullbackperverse} $f^* M[1]$ is a perverse sheaf on $\widetilde{Y}$ as $f$ is smooth of relative dimension $1$. Since $\pi\colon \widetilde{Y} \to Y$ is an isomorphism over the open $U$, the complex $(\pi_* f^* M[1])|_U$ is a perverse sheaf on $U$. Thus, ${}^p H^i( \pi_* f^* M[1])$ is supported at the origin for $i \neq 0$. Our task is to show that these sheaves are actually $0$. Since these sheaves are supported at a point, it is enough to show that
\[
R\Gamma_c(Y, {}^p H^i(\pi_* f^*M[1])) = 0
\]
for $i \neq 0$. But $R\Gamma_c(Y,-)$ is perverse $t$-exact as $Y$ is affine by Theorem~\ref{PerverseArtin}, so
\[
R\Gamma_c(Y, {}^p H^i(\pi_* f^*M[1]) \simeq H^i_c(Y, \pi_* f^* M[1]),
\]
whence we are reduced to the vanishing $R\Gamma_c(Y, \pi_* f^* M ) = 0$ that was already proven.
\end{proof}

As a corollary to the above observation we obtain the following statement.

\begin{proposition}
Let $\pi\colon \widetilde{Y} \to Y$ be the blow-up of the vertex of the affine cone $Y$ over a projective $Z \subseteq \PP^n$ of dimension $d$. If $\FF_{\!p,Z}[d]$ is perverse, then so is $\pi_*\FF_{\!p,\widetilde{Y}}[d+1]$. In particular this holds when $Z$ is smooth.
\end{proposition}

\begin{proof}
We apply the preceding proposition in the case of $M=\FF_{\!p,Z}[d]$, which is perverse by assumption. Since $f$ is smooth of relative dimension 1, we get that $f^*(\FF_{\!p,Z}[d])[1] = \FF_{\!p,\widetilde{Y}}[d+1]$ is also perverse by Lemma~\ref{smoothpullbackperverse}. And hence so is $\pi_*\FF_{\!p,\widetilde{Y}}[d+1] = \pi_*f^*(\FF_{\!p,Z}[d])[1]$ due to Proposition~\ref{GradedGR}.
\end{proof}

This statement can be viewed as the constructible $\FFp$-sheaf version of a Frobenius stable Grauert-Riemen\-schneider vanishing in the cone-situation described above. It states that if $Z$ is smooth, and $\omega^{\bullet}_{\widetilde{Y}}$ is a normalized dualizing complex of the smooth~$\widetilde{Y}$ viewed as a Cartier module placed in degree $-\dim \widetilde{Y}$, then the complex $R\pi_*\omega^{\bullet}_{\widetilde{Y}}$ of Cartier modules is, up to nilpotent Frobenius actions, concentrated in the single degree $-\dim Y$. In other words, on the higher derived functors $R^i\pi_*\omega_{\widetilde{Y}}$ for $i \neq 0$ of the dualizing module $\omega_{\widetilde{Y}} = h^{-\dim \widetilde{Y}}\omega^{\bullet}_{\widetilde{Y}}$, the induced right action of the Frobenius is nilpotent. Note that by a recent result of \cite{BBK} this statement is false beyond the cone case. 

Before preceding to the main result of this section we record the following:

\begin{corollary}
\label{CompactSupportCohCone}
We have
\[
R\Gamma_c(Y; \FFp) \simeq (\tau^{> 0} R\Gamma(Z;\FFp))[-1]
\]
and
\[
R\Gamma_c(Y; {}^p H^j(\FF_{\!p,Y})) \simeq H^{j-1}(\tau^{> 0} R\Gamma(Z; \FFp))[0].
\]
\end{corollary}

\begin{proof}
The second statement follows from the first as $R\Gamma_c(Y;-)$ is perverse $t$-exact by Theorem~\ref{PerverseArtin}. For the first, we use the~triangle
\[
\FF_{\!p,Y} \to \pi_* \FF_{\!p,\widetilde{Y}} \to k_* \tau^{> 0} R\Gamma(Z;\FFp),
\]
where the identification of the third term comes from proper base change as well as the fact that $\pi$ is an isomorphism outside the origin. Applying $R\Gamma_c(Y;-)$ and using that $R\Gamma_c(Y; \pi_* \FF_{\!p,\widetilde{Y}}) = 0$ by Proposition~\ref{GradedGR}, we get the result.
\end{proof}

\begin{theorem}
\label{EmbedIndepCharp}
For each $j$, the complex $k^* {}^p H^j(\FF_{\!p,Y})$ depends functorially on $Z$ and not on the projective embedding.
\end{theorem}

\begin{remark}
The immediate observable from Theorem~\ref{EmbedIndepCharp} is that each of the individual cohomology groups $H^i k^* {}^p H^j(\FF_{\!p,Y})$ depends only on $Z$ and not on the projective embedding. In particular, the Lyubeznik numbers $\lambda_{i,j}(A) $ (where $A$ is the coordinate ring of $Z$) depends only on $Z$ and not on the projective embedding; this consequence was conjectured by Lyubeznik \cite{Lyubeznik:survey}, proven earlier by Zhang \cite{ZhangLyuIndep}, and was the primary motivation for our result. Theorem~\ref{EmbedIndepCharp} is a bit stronger as we prove a statement at the level of complexes. For example, say $Z$ is equipped with an action of a group $G$ that extends to $\PP^n$. Then each $H^i k^* {}^p H^j(\FF_{\!p,Y})$ is naturally an $\FFp$-representation of $G$ that depends only on $Z$ by the version of Theorem~\ref{EmbedIndepCharp} at the level of cohomology groups. However, Theorem~\ref{EmbedIndepCharp} itself yields a finer statement: the entire complex $k^* {}^p H^j(\FF_{\!p,Y})$, viewed as an object of $D(\mathrm{Rep}_{\FFp}(G))$ of the derived category $\FFp$-representations of $G$, depends only on $Z$. Such an object has ``secondary'' invariants beyond its individual cohomology groups; for example, if $M$ is an $\FFp$-representation of $G$, then our theorem implies that the hypercohomology spectral sequence for $\RHom_G(k^* \FF_{\!p,Y}, M) \simeq M$
\[
E_2^{a,b}\colon \Ext^a_G(k^* {}^p H^{-b}(\FF_{\!p,Y}), M) \Rightarrow M
\]
depends only on $Z$, and not on the embedding.
\end{remark}

\begin{remark}
The characteristic $0$ analogue of Theorem~\ref{EmbedIndepCharp} does not hold without further hypotheses on $Z$; see \cite{RSW1}. In contrast, in characteristic $p$, the statement holds true without any restrictions on the singularities of $Z$. Tracing through the proof, the source of this difference is essentially Theorem~\ref{PerverseArtin}, especially the special case explained in in Example~\ref{AffineSpaceCoh}.
\end{remark}

\begin{proof}[Proof of Theorem~\ref{EmbedIndepCharp}]
Fix an integer $j$. The excision triangle for $U \subset Y$ with respect to the sheaf ${}^p H^j(\FF_{\!p,Y})$ gives
\begin{equation}
\label{independencetriangle}
R\Gamma_c(U; {}^p H^j(\FF_{\!p,U})) \xrightarrow{\ \alpha\ } R\Gamma_c(Y; {}^p H^j(\FF_{\!p,Y})) \to k^* {}^p H^j(\FF_{\!p,Y}).
\end{equation}
It is enough to show that the first two terms of this triangle, as well as the map $\alpha$, are independent of the embedding. For the left term of~\eqref{independencetriangle}, use the excision triangle for $U \subset \widetilde{Y}$ and the sheaf ${}^p H^j(\FF_{\!p,\widetilde{Y}})$ to~obtain
\[
R\Gamma_c(U; {}^p H^j(\FF_{\!p,U})) \to R\Gamma_c(\widetilde{Y}; {}^p H^j(\FF_{\!p,\widetilde{Y}})) \to R\Gamma_c(Z, ({}^p H^j(\FF_{\!p,\widetilde{Y}}))|_Z).
\]
Now the middle term is zero by Proposition~\ref{GradedGR} and the $t$-exactness of $R\Gamma_c$ due to Theorem~\ref{PerverseArtin}. Using Lemma~\ref{smoothpullbackperverse} applied to the map $f$ we compute the argument in the right term as
\[
({}^p H^j(\FF_{\!p,\widetilde{Y}}))|_Z \simeq ({}^p H^{j-1}(f^*\FF_{\!p,Z}[1]))|_Z\simeq (f^*({}^p H^{j-1}(\FF_{\!p,Z}))[1])|_Z \simeq {}^p H^{j-1}(\FF_{\!p,Z})[1]
\]
It follows that
\[
R\Gamma_c(U; {}^p H^j(\FF_{\!p,U})) \simeq R\Gamma(Z, {}^p H^{j-1}(\FF_{\!p,Z})[1])[-1]
\]
is independent of the embedding.

Corollary~\ref{CompactSupportCohCone} identifies the middle term of~\eqref{independencetriangle} as
\[
R\Gamma_c(Y; {}^p H^j(\FF_{\!p,Y})) \simeq H^{j-1}(\tau^{> 0} R\Gamma(Z; \FF_{\!p,Z})),
\]
which is evidently independent of the embedding as well.

Finally, we check that the map $\alpha$ in~\eqref{independencetriangle} is independent of the embedding. Note that the target of $\alpha$ lives in degree $0$ as explained above, while the source of $\alpha$ lives in $D^{\le 0}$ by Corollary~\ref{SemiperversePushforward}. So it remains to show that the map
\[
H^0(\alpha)\colon H^0_c(U, {}^p H^j(\FF_{\!p,U})) \to H^0_c(Y, {}^p H^j(\FF_{\!p,Y}))
\]
is independent of the embedding with respect to the identifications above. We do this by unwinding all the identifications involved.

Using the naturality of the excision triangle for the inclusion $U \subset \widetilde{Y}$ with respect to the natural map ${}^p \tau^{\le j} \FF_{\!p,\widetilde{Y}} \to \FF_{\!p,\widetilde{Y}}$ gives a map of exact triangles
\[
\xymatrix{ R\Gamma_c(U, {}^p \tau^{\le j} \FF_{\!p,U}) \ar[r] \ar[d] & R\Gamma_c(\widetilde{Y}, {}^p \tau^{\le j} \FF_{\!p,\widetilde{Y}}) \simeq 0 \ar[r] \ar[d] & R\Gamma_c(Z, ({}^p \tau^{\le j} \FF_{\!p,\widetilde{Y}})|_Z) \simeq R\Gamma_c(Z, {}^p \tau^{\le j-1} \FF_{\!p,Z}) \ar[d] \\
R\Gamma_c(U, \FF_{\!p,U}) \ar[r] & R\Gamma_c(\widetilde{Y}, \FF_{\!p,\widetilde{Y}}) \simeq 0 \ar[r] & R\Gamma_c(Z, \FF_{\!p,Z}),}
\]
where the middle terms in both rows are zero by Proposition~\ref{GradedGR}. Passing to boundary maps gives a commutative square
\begin{equation}
\label{indep1}
\xymatrix{ R\Gamma_c(Z, {}^p \tau^{\le j-1} \FF_{\!p,Z})[-1] \ar[r]^<<<<{\simeq} \ar[d] & R\Gamma_c(U, {}^p \tau^{\le j} \FF_{\!p,U}) \ar[d] \\
R\Gamma_c(Z, \FF_{\!p,Z})[-1] \ar[r]^{\simeq} & R\Gamma_c(U, \FF_{\!p,U})}
\end{equation}
where the left vertical map is obvious one (and thus independent of the embedding).

Next, comparing the excision triangles for the constant sheaf for $U \subset \widetilde{Y}$ and $U \subset Y$, we get a map of exact triangles
\[
\xymatrix{ R\Gamma_c(U, \FF_{\!p,U}) \ar[r] \ar@{=}[d] & R\Gamma_c(Y, \FF_{\!p,Y}) \ar[r] \ar[d] & R\Gamma(\Spec k, \FFp) \simeq \FFp \ar[d] \\
R\Gamma_c(U, \FF_{\!p,U}) \ar[r] & R\Gamma_c(\widetilde{Y}, \FF_{\!p,\widetilde{Y}}) \simeq 0\ar[r] & R\Gamma_c(Z, \FF_{\!p,Z}),}
\]
where the second middle term is zero by Proposition~\ref{GradedGR}. By direct computation the cone of the right vertical map is equal to $\tau^{> 0} R\Gamma_c(Z, \FF_{\!p,Z})$. A direct application of the octahedral axiom shows that this cone is canonically isomorphic to $R\Gamma_c(Y,\FF_{\!p,Y})[1]$ such that we obtain from this a commutative square 
\begin{equation}
\label{indep2}
\xymatrix{R\Gamma_c(Z, \FF_{\!p,Z})[-1] \ar[r]^{\simeq} \ar[d] & R\Gamma_c(U, \FF_{\!p,U}) \ar[d] \\
(\tau^{> 0} R\Gamma_c(Z, \FF_{\!p,Z}))[-1] \ar[r]^<<<<{\simeq} & R\Gamma_c(Y, \FF_{\!p,Y}),}
\end{equation}
where the left vertical map is the natural map (and thus independent of the embedding).

Combining squares~\eqref{indep1} and~\eqref{indep2} gives a commutative diagram
\[
\xymatrix{ R\Gamma_c(Z, {}^p \tau^{\le j-1} \FF_{\!p,Z})[-1] \ar[r]^<<<<{\simeq} \ar[d] & R\Gamma_c(U, {}^p \tau^{\le j} \FF_{\!p,U}) \ar[d] \\
(\tau^{> 0} R\Gamma_c(Z, \FF_{\!p,Z}))[-1] \ar[r]^<<<<<{\simeq} & R\Gamma_c(Y, \FF_{\!p,Y}),}
\]
where the left vertical map depends only on $Z$. But the map on $H^j$ induced by the right vertical map identifies~with
\[
H^0(\alpha)\colon H^0_c(U, {}^p H^j(\FF_{\!p,U})) \to H^0_c(Y, {}^p H^j(\FF_{\!p,Y}))
\]
appearing above: this follows from Theorem~\ref{PerverseArtin} for the target, and Corollary~\ref{SemiperversePushforward} for the source. This completes the proof.
\end{proof}

\subsection{Bass numbers}
\label{ss:BassCharp}

The goal of this subsection is to describe Bass numbers topologically, giving characteristic $p$ analogs of the results in \S\ref{ss:BassChar0}.

\begin{remark}
\label{rmk:cofinite:unit:Frob:module}
Let $(R,\frakm,\kappa)$ be a regular local ring of positive characteristic $p$ and $M$ be a cofinite unit Frobenius module. To understand the Bass numbers $\mu^j(M)\colonequals\dim_{\kappa}\Ext^j_R(\kappa,M)$ of $M$, which are not affected by field extensions, we may assume that $\kappa$ is algebraically closed. Then, since $M$ is supported only in $\{\frakm\}$, it can be identified with a perverse sheaf on $\{\frakm\}$, i.e., with $\FFp$-vectors spaces. Consequently $M\simeq E^{\oplus\mu}$ with~$E=H^d_\frakm(R)$. Since $M$ is cofinite, it follows that
\[
\mu^j(M)=\begin{cases} \mu<\infty & \text{if } j=0,\\ 0 & \text{if } j\neq 0.\end{cases}
\]
Put differently, $M$ is an injective $R$-module with finite Bass numbers. This recovers \cite[Corollary~3.6]{HS:TAMS}.
\end{remark}

When $M$ is a local cohomology module, we can describe the Bass numbers topologically in term of compactly supported cohomology. To this end, we fix some notation.

\begin{notation}
Let $X$ be a smooth affine variety of dimension $d$ over an algebraically closed field $k$ of characteristic $p$. Fix a closed subscheme $i\colon Z \into X$ with open complement $j\colon U \to X$ as well as a closed point $x \in Z$. Assume $Z \neq X$, so $\dim Z < d$.
\end{notation}

The compactly supported cohomology of perverse truncations of the constant sheaf on $Z$ can be effectively computed in terms of $U$ and $X$:

\begin{lemma}
\label{BassCharp}
One has the following:
\begin{enumerate}[\quad\rm(1)]
\item \label{BassCharp:1} Each $R\Gamma_c(Z, {}^p H^{j} (\FF_{\!p,Z}[d]))$ lives in degree $0$ and vanishes for $j \ge 0$.

\item \label{BassCharp:2} For $k < -1$, we have a natural isomorphism
\[
R\Gamma_c(Z, {}^p H^k(\FF_{\!p,Z}[d])) \simeq H^{k+1}_c(U, \FF_{\!p,U}[d]).
\]

\item \label{BassCharp:3} There is an exact sequence
\[
0 \to R\Gamma_c(Z, {}^p H^{-1} (\FF_{\!p,Z}[d])) \to H^0_c(U, \FF_{\!p,U}[d]) \to H^0_c(X,\FF_{\!p,X}[d]) \to 0.
\]

\item \label{BassCharp:4} If $H^d_c(X, \FF_{\!p,X}) = 0$, e.g., if $X = \AA^d$ is affine space, then $(2)$ extends to all $k$.
\end{enumerate}
\end{lemma}

\begin{proof}
The first part of (1) follows immediately from Theorem~\ref{PerverseArtin}. For the second, it is enough to show that ${}^p H^k(\FF_{\!p,Z}[d]) = 0$ for $k \ge 0$. For this, it is enough to show that $\FF_{\!p,Z}[d-1] \in {}^p D^{\le 0}(Z)$, but this follows as $d-1 \ge \dim Z$ by our assumptions.

For (2), consider the excision triangle for $U \subset X$ given by
\[
j_!(\FF_{\!p,U}[d]) \to \FF_{\!p,X}[d] \to i_* \FF_{\!p,Z}[d].
\]
All terms lie in ${}^p D^{\le 0}$: the middle term is perverse as $X$ is smooth of dimension $d$, the left term lies in ${}^p D^{\le 0}$ by Corollary~\ref{SemiperversePushforward}, and the rest follows by the long exact sequence for perverse cohomology. Since the middle term is perverse, the long exact sequence gives isomorphisms
\[
i_* {}^p H^k (\FF_{\!p,Z}[d]) \simeq {}^p H^{k+1}(j_! \FF_{\!p,U}[d])
\]
for $k < -1$, where we use that $i_*$ is $t$-exact. Applying the $t$-exact functor $R\Gamma_c(X,-)$ now gives the claim.

For (3), we take the long exact sequence associated to the triangle used in (2) together with the vanishing of ${}^p H^0(\FF_{\!p,Z}[d])=0$ proven in (1) and the vanishing of ${}^p H^{-1}(\FF_{\!p,X}[d])=0$ since $X$ is smooth, to get a short exact sequence
\[
0 \to i_* {}^p H^{-1}(\FF_{\!p,Z}[d]) \to {}^p H^0(j_! \FF_{\!p,U}[d]) \to {}^p H^0(\FF_{\!p,X}[d]) \simeq \FF_{\!p,X}[d] \to 0
\]
of perverse sheaves on $X$. Applying the $t$-exact functor $R\Gamma_c(X,-)$ now gives the desired exact sequence.

Part (4) follows immediately from (3) as well as the vanishing proven in (1).
\end{proof}

This calculation has the following consequence for local cohomology, parallel to the results in \S\ref{ss:BassChar0}.

\begin{theorem}
Let $(R,\frakm)$ be a regular local ring of characteristic $p$ whose residue field is algebraically closed and $I$ be an ideal of $R$. Let $x$ denote the closed point. Assume that $H^k_I(R) \neq 0$ and $\Supp(H^k_I(R))=\{x\}$.
\begin{enumerate}[\quad\rm(1)]
\item If $k > 1$, then $H^k_I(R) \simeq E \otimes_{\FFp} H^{d-k+1}_c(U; \FFp)^\vee$ where $E$ denotes the injective hull of the residue field.

\item Set $X=\Spec R$. If $H^d_c(X, \FF_{\!p,X}) = 0$, then $(1)$ extends to all $k$.
\end{enumerate}
\end{theorem}

\begin{proof}
We shall use the calculations in Lemma~\ref{BassCharp}.

The dual Riemann-Hilbert correspondence carries $H^k_I(R)$ to $i_* {}^p H^{-k}(\FF_{\!p,Z}[d])$, so, to prove part (1), it suffices to identify the latter with a skyscraper sheaf at $x$ with fiber $H^{d-k+1}_c(U; \FFp)^\vee$ as long as $k > 1$. The assumption that $H^k_I(R)$ is supported at $x$ immediately implies that $i_* {}^p H^{-k}(\FF_{\!p,Z}[d])$ is a skyscraper sheaf. As applying $R\Gamma_c(X;-)$ carries a skyscraper sheaf to its fiber, it suffices to show that
\[
R\Gamma_c(X; i_* {}^p H^{-k}(\FF_{\!p,Z}[d])) \simeq H^{d-k+1}_c(U; \FFp)
\]
as long as $k > 1$. As $i_*$ is $t$-exact, this follows from Lemma~\ref{BassCharp}.\ref{BassCharp:2}.

When $k=0$ the hypothesis on support is vacuous. When $k = 1$ the hypothesis on support implies that $\dim X=1$. Then the conclusion follows in the exact same way as above using Lemma~\ref{BassCharp}.\ref{BassCharp:3} and~\ref{BassCharp}.\ref{BassCharp:4}.
\end{proof}

\subsection{Gradings}
\label{ss:GradingCharp}

In this subsection, we record the characteristic $p$ analog of the result in \S\ref{ss:GradingChar0}, recovering the main result of \cite{YiZhang} (for a generalization to Eulerian graded $\calD$-modules in characteristic $p$ we refer the interested reader to \cite{Ma-Zhang}); as in \S\ref{ss:GradingChar0}, the key is again the connectedness of $\GG_m$.

\begin{proposition}
Fix an algebraically closed field $k$ of characteristic $p$. Let $X = \Spec R$ be a smooth affine $k$-scheme of dimension $d$. Let $M$ be a fgu $R[F]$-module that vanishes away from a closed point $x \in X$, corresponding to a maximal ideal $\frakm$. Then $M \simeq E^{\oplus \mu}$ as fgu $R[F]$-modules, where $E = H^d_{\frakm}(R)$.

Assume moreover that $X$ admits a $\GG_m$-action that fixes $x$ and such that $M$ is $\GG_m$-equivariant. Then any isomorphism $M \simeq E^{\oplus \mu}$ of fgu $R[F]$-modules is automatically $\GG_m$-equivariant.
\end{proposition}

\begin{proof}
Translating via the dual Riemann-Hilbert correspondence, the first statement follows as the category of perverse sheaves on $X$ supported at $\{x\}$ identifies with the category of perverse sheaves on $\{x\}$, i.e., with the $\FFp$-vector spaces. For the second statement, it is enough to prove that forgetting the $\GG_m$-action identifies the category of $\GG_m$-equivariant $\FFp$-sheaves on $\{x\}$ with the category of $\FFp$-sheaves on $X$; this follows because~$\GG_m$ is connected.
\end{proof}

\subsection{Cohen-Macaulayness and $F$-rationality up to finite covers}
\label{ss:HHSm}

The goal of this subsection is to recover fundamental results of Hochster-Huneke \cite{HHRPlusCM} and Smith \cite{SmRPlusFRat} on the Cohen-Macaulayness and $F$-rationality properties of the absolute integral closure of (non-graded) rings; the graded analogs will be established in \S\ref{ss:HHSmGr}.

\begin{remark}
\label{mixedcharinspire}
The arguments in this section and \S\ref{ss:HHSmGr} are analogs of those in \cite{BhattRPlusCM,BhattLurieRHtors} (which concerns the mixed characteristic case); \cite{BhattRPlusCM}, in turn, was partially inspired by previous arguments in this paper.
\end{remark}

\begin{notation}
Let $R$ be an excellent noetherian normal local domain of characteristic $p$ and dimension $d$ and let $R^+$ denote the absolute integral closure of $R$. Assume $R$ is $F$-finite; thus, $R$ admits a normalized dualizing complex, and we have a theory of perverse $\FFp$-sheaves on $R$-schemes. Fix an absolute integral closure $R \to R^+$ of $R$. Let $\calP$ be the poset of all $R$-finite $R$-subalgebras $S \subset R^+$, so $\varinjlim_{S \in \calP} S \simeq R^+$. Any \'etale sheaf on $\Spec S$ for $S \in \calP$ or $\Spec (R^+)$ will be regarded as an \'etale sheaf on $\Spec R$ via pushforward.
\end{notation}

Our basic vanishing result is the following, borrowed from \cite{BhattRPlusCM,BhattLurieRHtors}.

\begin{lemma}
\label{AICVanEtCoh}
Let $Y$ be an integral normal scheme with algebraically closed function field. Then $R\Gamma(Y,\FFp) \simeq \FFp$. Moreover, for any nonempty open subset $U$ and the immersion $j\colon U\hookrightarrow Y$, we have $\FF_{\!p,Y}=j_*\FF_{\!p,U}$.
\end{lemma}

\begin{proof}
It is easy to see that the \'etale and Zariski topologies of $Y$ coincide. This implies that $R\Gamma(Y,\FFp) \simeq \FFp$ by Grothendieck's theorem~\cite[\href{https://stacks.math.columbia.edu/tag/02UW}{Tag~02UW}]{stacks-project} on the vanishing of the cohomology of constant sheaves on irreducible topological spaces. Since each non-empty open subset of $Y$ is also an integral normal scheme with algebraically closed function field, it follows from the first part that the morphism $\FF_{\!p,Y}\to j_*\FF_{\!p,U}$ is an isomorphism, since the sections on each irreducible open affine subset are in both cases $\FFp$, by the first part. 
\end{proof}

\begin{theorem}[$\Spec(R^+)$ is an $IC$-manifold]
\label{PervRPlus}
The natural map
\[
\{\FF_{\!p,S}[d]\}_{S \in \calP} \to \{\mathrm{IC}_{S, \FFp} \}_{S \in \calP}
\]
is an isomorphism of ind-objects. In particular, taking the colimit gives an isomorphism
\[
\FF_{\!p,R^+}[d] \simeq \varinjlim_{S \in \calP} \FF_{\!p,S}[d] \simeq \varinjlim_{S \in \calP} \mathrm{IC}_{S,\FFp}.
\]
Thus, $\FF_{\!p,R^+}[d]$ is perverse.
\end{theorem}

\begin{proof}
Consider the natural maps
\[
\{\FF_{\!p,S}[d]\}_{S \in \calP} \xrightarrow{\ \alpha\ } \{{}^p H^0(\FF_{\!p,S}[d])\}_{S \in \calP} \xrightarrow{\ \beta\ } \{ \mathrm{IC}_{S, \FFp} \}_{S \in \calP}
\]
of ind-objects. We shall show that $\alpha$ and $\beta$ are both isomorphisms of ind-objects, which implies the theorem.

Isomorphy of $\alpha$: By general nonsense with compactness properties of constructible sheaves, it suffices to show that $\alpha$ induces an isomorphism on taking colimits; equivalently, it suffices to show that
\[
\varinjlim_{S \in \calS} \FF_{\!p,S}[d] \simeq \FF_{\!p,R^+}[d]
\]
is perverse. The containment $\FF_{\!p,R^+}[d] \in {}^p D^{\le 0}$ is immediate from the definition of the perverse $t$-structure recalled in \S\ref{ss:GabberPervComp}, so it suffices to show $i_x^! \FF_{\!p,R^+}[d] \in D^{\ge -d(x)}$ for all $x \in \Spec R$. When $x$ is the generic point, this is clear as $d(x)=d$. For non-generic $x$, we claim the stronger statement that $i_x^! \FF_{\!p,R^+} = 0$. To show this, it suffices to check that $\FF_{\!p,R^+}$ is $*$-extended from any non-empty affine open $U \subset \Spec R$, which follows from Lemma~\ref{AICVanEtCoh}.

Isomorphy of $\beta$: as each component map of $\beta$ is a surjection of perverse sheaves, the colimit is also surjective, so it suffices to prove that for any $S \in \calP$, there exists a larger $S \to T$ in $\calP$ such that the kernel of
\[
\beta_S\colon {}^p H^0(\FF_{\!p,S}[d]) \to \mathrm{IC}_{S,\FFp}
\]
is annihilated by the map ${}^p H^0(\FF_{\!p,S}[d]) \to {}^p H^0(\FF_{\!p,T}[d])$. By constructibility of ${}^p H^0(\FF_{\!p,S}[d])$, this can be checked after taking the colimit over all possible $T$, so it suffices to show that $\ker\beta_S$ is annihilated by
\[
{}^p H^0(\FF_{\!p,S}[d]) \to {}^p H^0( \FF_{\!p,R^+}[d]).
\]
But the previous paragraph shows that ${}^p H^0( \FF_{\!p,R^+}[d])$ agrees with $\FF_{\!p,R^+}[d]$ and is $*$-extended from any non-empty open subset of $\Spec R$. Now any perverse sheaf $*$-extended from an open subset cannot contain perverse subsheaves supported on the closed complement. In particular, any perverse subsheaf of ${}^p H^0( \FF_{\!p,R^+}[d])$ cannot be supported on a proper closed subset of $\Spec R$. But $\ker\beta_S$ is supported on a proper closed subset of $\Spec R$: it vanishes after restriction to any dense open $U \subset \Spec R$ such that the preimage of $U$ in $S$ is Cohen-Macaulay (Example~\ref{CMConstPerv}). Then $\ker\beta_S$ must indeed die in ${}^p H^0( \FF_{\!p,R^+}[d])$, as wanted.
\end{proof}

\begin{corollary}
\label{CMRPlus}
One has the following:
\begin{enumerate}[\quad\rm(1)]
\item (Hochster-Huneke) $R^+$ is Cohen-Macaulay.

\item (Smith) For any maximal ideal $\frakm\subset R$, any Frobenius stable $R$-submodule $N \subset H^d_{\frakm}(R)$ is either all of~$H^d_{\frakm}(R)$, or is annihilated by the map $H^d_{\frakm}(R) \to H^d_{\frakm}(S)$ for sufficiently large $S \in \calP$.

\item For any maximal ideal $\frakm\subset R$ and nonzero $g \in \frakm$, the $R^+$-module $H^d_{\frakm}(R^+)$ has no nonzero $g$-almost zero elements; that is, if $\alpha\in H^d_{\frakm}(R^+)$ satisfies $g^{1/p^n}\alpha =0$ for all $n\ge 1$, then $\alpha=0$.
\end{enumerate}
\end{corollary}

\begin{proof}
(1) For any $S \in \calP$, the Riemann-Hilbert functor $\RH\colon D(\Spec(R)_{\et}, \FFp) \to D(R[F])$ carries $\FF_{\!p,S}$ to $S_{\perf}$. Passing to inductive limits and noting that $R^+$ is already perfect, we learn that $\RH(\FF_{\!p,R^+}) = R^+$. Applying Theorem~\ref{PervRPlus}, we then learn that $R^+$ can also be written as
\[
R^+ \simeq \varinjlim_{S \in \calP} \RH(\mathrm{IC}_{S,\FFp})[-d].
\]
The desired claim now follows from Theorem~\ref{PervRPlus}, the perversity of $\mathrm{IC}_{S,\FFp}$, and Example~\ref{CMPerv}.

(2) We may assume $R$ is $\frakm$-adically complete. Assume $N \neq H^d_{\frakm}(R)$ as there is nothing to show otherwise. Recall from Proposition~\ref{PervCofDer} that the composition
\[
D^b_{coh}(R) \xrightarrow{\ \DD_R(-)\ } D^b_{coh}(R) \xrightarrow{\ R\Gamma_\frakm(-)\ } D^b_{\cof}(R)
\]
is an equivalence and $t$-exact for the standard $t$-structures on the source and target. This composition carries the dualizing sheaf $H^{-d}(\omega_R^\bullet)$ to $H^d_{\frakm}(R)$. Consequently, the $R$-module $N \subset H^d_{\frakm}(R)$ arises from a quotient $Q$ of $H^{-d}(\omega_R^\bullet)$. As $N \neq H^d_{\frakm}(R)$ and the dualizing sheaf is torsion-free, we must have $gQ =0$ for some $0 \neq g \in R$, whence $gN = 0$ as well.

Write $N_{\perf} \subset H^d_{\frakm}(R_{\perf})$ for the $R_{\perf}$-span of the image of $N$. Recall that the exactness result in Proposition~\ref{RHtexactLocalCoh} lets us write
\[
H^d_{\frakm}(R_{\perf}) = R\Gamma_{\frakm}(\RH({}^p H^0(\FF_{\!p,R}[d]))).
\]
Using the equivalence of (2) and (4) in Corollary~\ref{PerverseCharpLocal}, we can find a perverse subsheaf $F \subset {}^p H^0(\FF_{\!p,R}[d])$ such that the map
\[
R\Gamma_{\frakm}(\RH(F)) \to R\Gamma_{\frakm}(\RH({}^p H^0(\FF_{\!p,R}[d])))
\]
equals the map
\[
N_{\perf} \to H^d_{\frakm}(R_{\perf}).
\]
In fact, explicitly, the perverse holonomic Frobenius $R$-module $\RH(F)$ is the perfection of the derived $\frakm$-completion of $N$. In particular, as $gN = 0$, we learn that $F|_{\Spec(R[1/g])} = 0$, so $F$ is supported on the closed set $\Spec(R/g) \subset \Spec R$. But then the composition
\[
F \to {}^p H^0(\FF_{\!p,R}[d]) \to \mathrm{IC}_{R,\FFp}
\]
must be $0$: if not, then the map must be surjective by simplicity of the target, but this is not possible as the target is supported everywhere, while the source vanishes on a non-empty open subset. Theorem~\ref{PervRPlus} then implies that there exists some $S \in \calP$ such that the map ${}^p H^0(\FF_{\!p,R}[d]) \to {}^p H^0(\FF_{\!p,S}[d])$ factors over ${}^p H^0(\FF_{\!p,R}[d]) \to \mathrm{IC}_{R,\FFp}$, which then implies that image of $N$ in
\[
H^d_{\frakm}(S_{\perf}) = R\Gamma_{\frakm}(\RH({}^p H^0(\FF_{\!p,S}[d])))
\]
vanishes. By Corollary~\ref{LyubeznikNilp}, we conclude that the image of $N$ in $H^d_{\frakm}(S)$ generates a Frobenius nilpotent $S$-module; passing from $S$ to $S \to S^{1/p^n}$ for $n \gg 0$ then solves the problem.

(3) Pick a class $\alpha \in H^d_{\frakm}(R^+)$ that is $g$-almost zero; we shall show that $\alpha = 0$. At the expense of enlarging~$R$, we may assume that $\alpha$ comes from a class in $H^d_{\frakm}(R)$ and hence from a class (abusively called~$\alpha$) in $H^d_{\frakm}(R_{\perf})$. If we write $K \colonequals \ker(H^d_{\frakm}(R) \to H^d_{\frakm}(R^+))$, then the assumption that $\alpha$ is $g$-almost zero in~$H^d_{\frakm}(R^+)$ implies that
\[
g^{1/p^n} \alpha \in K_{\perf} \colonequals \ker(H^d_{\frakm}(R_{\perf}) \to H^d_{\frakm}(R^+))
\]
for all $n$. Also, we know that the map $H^d_{\frakm}(R) \to H^d_{\frakm}(R^+)$ is not zero (e.g., by the direct summand conjecture applied to a Noether normalization of $\widehat{R}$), so $K \neq H^d_{\frakm}(R)$. Part (2) then allows us to find a finite extension $R \to S$ in $\calP$ such that $K$ maps to $0$ in $H^d_{\frakm}(S)$, and thus $K_{\perf}$ maps to $0$ in $H^d_{\frakm}(S_{\perf})$. Thus, after replacing $R$ with $S$, we may assume that $g^{1/p^n} \alpha = 0$ in $H^d_{\frakm}(R_{\perf})$ for all $n$.

Consider the $R_{\perf}$-submodule $N \subset H^d_{\frakm}(R_{\perf})$ spanned by all $g$-almost zero elements; this is a Frobenius submodule. It suffices to show $N$ maps to $0$ in $H^d_{\frakm}(R^+)$. We can write $N = \varinjlim_n N_n$, where $N_n \subset H^d_{\frakm}(R^{1/p^n})$ is the preimage of $N$ (and thus a Frobenius submodule of $H^d_{\frakm}(R^{1/p^n})$) for all $n$. We shall show that $N_n$ maps to $0$ in $H^d_{\frakm}(R^+)$ for each $n$ separately. Using (2), it suffices to show $N_n \neq H^d_{\frakm}(R^{1/p^n})$ for each $n$. But if we had $N_n = H^d_{\frakm}(R^{1/p^n})$ for some $n$, then we would learn that $N = H^d_{\frakm}(R_{\perf})$ by passing to the perfection, so $H^d_{\frakm}(R_{\perf})$ would itself be $g$-almost zero and in particular killed by $g$. Now if $H^d_{\frakm}(R_{\perf})$ is killed by~$g$, then it is in fact $0$ as it is also $g$-divisible by virtue of being a top local cohomology module. But then Corollary~\ref{PerverseCharpLocal} shows that ${}^p H^0(\FF_{\!p,R} [d]) = 0$, whence the constant sheaf vanishes on the dense open subset of $\Spec(R)_{reg} \subset \Spec R$, which is absurd.
\end{proof}

\begin{remark}
\cite{SannaiSingh} have improved Corollary~\ref{CMRPlus} (1) by showing that $R^+$ can be replaced with $R^{+,\text{sep}}$, which is the integral closure of $R$ in a separable closure of $\mathrm{Frac}(R)$. As $R^{+,\text{sep}}$ is not perfect, this generalization is inaccessible to our topological techniques: our primary tool is the covariant Riemann-Hilbert correspondence, which naturally takes values in perfect Frobenius modules.
\end{remark}

\subsection{Kodaira vanishing up to finite covers}
\label{ss:HHSmGr}

In this subsection, we prove the graded main theorem of \cite{HHRPlusCM} as well the analog for top local cohomology in \cite{SmRPlusFRatGr}.

\begin{notation}
Let $X$ be a projective variety of dimension $d$ over a perfect field $k$ of characteristic $p$ equipped with an ample line bundle $L$. Fix an absolute integral closure $\pi\colon X^+ \to X$, and let $\calP$ denote the poset of finite covers of $X$ dominated by $X^+$, i.e., the poset of all factorizations $X^+ \to Y \xrightarrow{\ f_Y\ } X$ of $\pi$ with $f_Y$ finite surjective and $Y$ integral. For any $Y \in \calP$, define the following objects:
\begin{align*}
\text{The total space of \ } L^{-1} & \colonequals T(Y,L) = \underline{\Spec}_Y(\bigoplus_{n \ge 0} f_Y^* L^n), \\
\text{The } \GG_m\text{-torsor attached to } L^{-1} & \colonequals U(Y,L) = \underline{\Spec}_Y(\bigoplus_{n \in \ZZ} f_Y^* L^n), \\
\text{The homogeneous coordinate ring of } (Y,L) & \colonequals R(Y,L) = H^0(T(Y,L), \calO_{T(Y,L)}) = \bigoplus_{n \ge 0} H^0(Y, f_Y^* L^n), \\
\text{The affine cone over } (Y,L) & \colonequals C(Y,L) = \Spec(R(Y,L)).
\end{align*}
Thus, we have natural maps
\[
Y \xleftarrow{\ g_Y\ } T(Y,L) \xrightarrow{\ h_Y\ } C(Y,L)
\qquad \text{and} \qquad
j_Y\colon U(Y,L) \into T(Y,L)
\]
where $g_Y$ is a line bundle, $h_Y$ is the affinization map and is proper and birational, and $j_Y$ is an affine open immersion with reduced complement the $0$-section $Y \into T(Y,L)$. Moreover, the $0$-section $Y \subset T(Y,L)$ is contracted to the origin in $C(Y,L)$ by $h_Y$ and is in fact the exceptional locus of $h_Y$, so $h_Y$ identifies $U(Y,L)$ with the complement of the origin in $C(Y,L)$. As $Y \in \calP$ varies, one has compatible maps for all the varieties defined above. Passing to the limit, we set
\[
T^+ = \varprojlim_{Y \in \calP} T(Y,L), \quad U^+ = \varprojlim_{Y \in \calP} U(Y,L), \quad C^+ = \varprojlim_{Y \in \calP} C(Y,L),
\quad \text{and} \quad
R^{+,gr} = \varinjlim_{Y \in \calP} R(Y,L) \simeq H^0(C^+,\calO_{C^+}).
\]
Finally, any \'etale sheaf (or complex) on any scheme over $C(X,L)$ is regarded as a sheaf on $C(X,L)$ via derived pushforward unless otherwise specified.
\end{notation}

The main result of this section is the following, giving a graded analog of Theorem~\ref{PervRPlus}. The proof in fact relies crucially on Theorem~\ref{PervRPlus} to solve the problem away from the cone point; the solution at the cone point is obtained using Proposition~\ref{GradedGR}.

\begin{theorem}[$C^+$ is an $IC$-manifold]
\label{PervRPlusGr}
The maps
\[
\xymatrix{
\{\FF_{\!p,C(Y,L)}[d+1]\}_{Y \in \calP} \ar[r]^-{\alpha} \ar[d]^-{\beta} & \{\mathrm{IC}_{C(Y,L), \FFp}\}_{Y \in \calP} \ar[d]^-{\gamma} \\
\{\FF_{\!p,T(Y,L)}[d+1]\}_{Y \in \calP} \ar[r]^-{\delta} & \{\mathrm{IC}_{T(Y,L), \FFp}\}_{Y \in \calP}}
\]
are all isomorphisms of ind-objects. In particular,
\[
\FF_{\!p,R^{+,gr}}[d+1] \simeq \varinjlim_{Y \in \calP} \mathrm{IC}_{C(Y,L), \FFp}
\]
is perverse.
\end{theorem}

\begin{proof}
Note that since $T(Y,L) \to C(Y,L)$ is proper, all objects involved above are constructible. By general nonsense, it suffices to show that all maps induce isomorphisms after taking colimits.

We apply Theorem~\ref{PervRPlus} to affine opens in $X$ and then pull back along the maps $T(Y,L) \to Y$. As these maps are smooth with connected fibers, the (shifted) pullback from Lemma~\ref{smoothpullbackperverse} preserves $IC$-sheaves, see \cite{Staebler}. Hence we learn that $\varinjlim \delta$ is an isomorphism.

The map $\varinjlim \beta$ identifies with the map $\FF_{\!p,C^+} \to \FF_{\!p,T^+}$. It is thus clearly an isomorphism away from the cone point. To show the claim at the cone point, by proper base change, as the fiber of $T^+ \to C^+$ at the cone point is $X^+$, it suffices to show that $\FFp \simeq R\Gamma(X^+, \FFp)$, which follows by Lemma~\ref{AICVanEtCoh}.

It follows from the first half of the statement of Proposition~\ref{GradedGR} that each $\FF_{\!p,T(Y,L)}[d+1]$ is perverse on $C(Y,L)$ which implies that $\FF_{\!p,T^+}[d+1]$ is perverse on $C^+$. Then the isomorphy of $\beta$ implies that $\FF_{\!p,C^+}[d+1]$ is also perverse.

It remains to prove that
\[
\varinjlim \alpha\colon \FF_{\!p,C^+}[d+1] \to \varinjlim_{Y \in \calP} \mathrm{IC}_{C(Y,L),\FFp}
\]
is an isomorphism. By the isomorphy of $\beta$ and $\delta$ and simplicity of $\mathrm{IC}$-sheaves, we already know that $\varinjlim \alpha$ is a surjection of perverse sheaves on $C(X,L)$, giving a short exact sequence
\[
0 \to K \to \FF_{\!p,C^+} \xrightarrow{\ \varinjlim \alpha\ } \varinjlim_{Y \in \calP} \mathrm{IC}_{C(Y,L),\FFp} \to 0
\]
of perverse sheaves on $C(X,L)$. The map $\varinjlim \alpha$ is an isomorphism away from the cone point: indeed, this holds true for $\gamma$ as $T(Y,L) \to C(Y,L)$ is an isomorphism outside the cone point, so the claim follows from the known statements for $\beta$ and $\delta$. Thus, the kernel $K$ is supported at the cone point, so it is enough to show that $R\Gamma_c(C(X,L), K) = 0$. Applying $R\Gamma_c(C(X,L),-)$ to the above short exact sequence gives a triangle
\[
R\Gamma_c(C(X,L),K) \to R\Gamma_c(C(X,L), \FF_{\!p,C^+}) \to R\Gamma_c(C(X,L), \varinjlim_{Y \in \calP} \mathrm{IC}_{C(Y,L),\FFp}).
\]
In fact, as $C(X,L)$ is affine, the perverse Artin vanishing theorem (Theorem~\ref{PerverseArtin}) shows that the preceding triangle is a short exact sequence of complexes in degree $0$. Thus, it suffices to show the middle term vanishes. But we can rewrite the middle term as
\[
R\Gamma_c(C(X,L), \FF_{\!p,C^+}) \simeq \varinjlim_{Y \in \calP} R\Gamma_c(C(Y,L), \FF_{\!p,C(Y,L)}).
\]
Since $\varinjlim \beta$ is an isomorphism, we may replace each $R\Gamma_c(C(Y,L), \FF_{\!p,C(Y,L)})$ with $R\Gamma_c(C(Y,L), {h_Y}_*\FF_{\!p,T(Y,L)})$ which vanishes by the second part of Proposition~\ref{GradedGR} as $\FF_{\!p,T(Y,L)}\simeq g^*_Y\FF_{\!p,Y}$. So the direct limit above also vanishes, as wanted.
\end{proof}

\begin{corollary}
\label{CMRPlusGr}
One has:
\begin{enumerate}[\quad\rm(1)]
\item (Hochster-Huneke) $R^{+,gr}_{\perf}$ is Cohen-Macaulay.

\item (Smith) Let $\frakm\subset R(X,L)$ be the homogeneous maximal ideal corresponding to the origin in $C(Y,L)$. Then any Frobenius stable $R$-submodule $N \subset H^d_{\frakm}(R(X,L))$ is either all of $H^d_{\frakm}(R(X,L))$ or is annihilated by the map $H^d_{\frakm}(R(X,L)) \to H^d_{\frakm}(R(Y,L)^{1/p^n})$ for sufficiently large $Y \in \calP$ and $n \gg 0$.
\end{enumerate}
\end{corollary}

\begin{proof}
This follows as in Corollary~\ref{CMRPlus} using Theorem~\ref{PervRPlusGr} instead of Theorem~\ref{PervRPlus}.
\end{proof}

\begin{remark}
Some variants of Corollary~\ref{CMRPlusGr} that have been established in the literature can also be deduced by our topological methods. We discuss two instances:

(1) The papers \cite{HHRPlusCM, SmRPlusFRatGr} prove analogs of Corollary~\ref{CMRPlusGr} for a larger ring $R^{+,GR}$ in lieu of $R^{+,gr}$ obtained by allowing $\QQ$-gradings. Geometrically, this corresponds to choosing a compatible system $\{L^{1/n}\}$ of $n$-th roots of the pullback of $L$ to $X^+$, and working with the $\QQ$-graded object $C^+_\infty = \Spec(\bigoplus_{n \in \QQ_{\ge 0}} H^0(X^+, L^n))$ instead of the corresponding $\NN$-graded object $C^+$ as in Corollary~\ref{CMRPlusGr}. To deduce this generalization, one first proves the analog for $C^+_m = \Spec(\bigoplus_{n \in \ZZ[1/m]_{\ge 0}} H^0(X^+, L^n))$ by deducing it from Corollary~\ref{CMRPlusGr} applied to a finite cover of $X$ where $L$ has a chosen $m$-th root; taking a limit over $m$ then yields the result for $C^+_\infty$.

(2) The paper \cite{BhattDerSplinPos} proves an analog of Corollary~\ref{CMRPlusGr} (formulated geometrically) when $L$ is only assumed to be semiample and big, rather than ample. This version can be proved similarly to Corollary~\ref{CMRPlusGr} by replacing the category $\calP$ of all finite covers of $X$ dominated by $X^+$ with the category $\calP_{alt}$ of all alterations of $X$ dominated by a generic point of $X^+$. The key input in carrying out this argument is the following result, whose proof critically relies on the proper pushforward compatibility of the Riemann-Hilbert functor as well as Lemma~\ref{AICVanEtCoh}:
\end{remark}

\begin{proposition}[Annihilating cohomology of proper maps]
Let $f\colon X \to S$ be a proper map of noetherian schemes of characteristic $p$, with $S$ affine. Then there exists a finite surjective map $Y \to X$ such that $H^i(X, \calO_X) \to H^i(Y, \calO_Y)$ is the $0$ map for $i > 0$.
\end{proposition}

\begin{proof}
By standard reductions, we may assume $X$ and $S$ are integral. Choose an absolute integral closure $X^+ \to X$. As $X^+$ is a cofiltered limit of finite covers of $X$, a limit argument reduces us to showing that $H^i(X^+,\calO_{X^+}) = 0$ for $i > 0$. Let $S^+$ denote the integral closure of $S$ in the composition $X^+ \to X \to S$. Then $S^+$ is an absolutely integrally closed domain and we have an induced map $h\colon X^+ \to S^+$ which is pro-proper. We observe that $\FFp \simeq Rh_* \FFp$: As the \'etale and Zariski topologies of $S^+$ coincide (by Lemma~\ref{AICVanEtCoh}), this can be checked after taking sections over a non-empty affine open $U \subset S^+$, i.e., we must show that $R\Gamma(U, \FFp) \simeq R\Gamma(U \times_{S^+} X^+, \FFp)$. But both $U$ and $U \times_{S^+} X^+$ are integral normal schemes with algebraically closed function fields, so Lemma~\ref{AICVanEtCoh} implies both complexes are identified with $\FFp[0]$, and are thus isomorphic. Then the proper pushforward compatibility of the Riemann-Hilbert functor shows that
\begin{align*}
R\Gamma(X^+, \calO_{X+}) &= R\Gamma(X^+, \RH(F_p)) = R\Gamma(\Spec R^+, Rh_*\RH(F_p))= R\Gamma(\Spec R^+, \RH(Rh_*F_p)) \\ &= R\Gamma(\Spec R^+, \RH(F_p)) = R\Gamma(\Spec R^+, \calO_{R^+}) = R^+
\end{align*}
which shows that $H^i(X^+, \calO_{X^+})=0$ for $i > 0$ and completes the proof.
\end{proof}

The paper \cite{BhattRPlusCM} contains similar arguments in the more involved mixed characteristic setting; we refer the reader there for more details.



\begin{thebibliography}{BBDG}

\bibitem[AB]{AB}
D.~Arinkin and R.~Bezrukavnikov, \emph{Perverse coherent sheaves}, Mosc. Math. J.~\textbf{10} (2010), 3--29.

\bibitem[Ba]{BaudinDual}
J.~Baudin, \emph{Duality between Cartier crystals and perverse $\mathbf{F}_p$-sheaves, and application to generic vanishing},\newline
\url{https://arxiv.org/abs/2306.05378}.

\bibitem[BBK]{BBK}
J.~Baudin, F.~Bernasconi, and T.~Kawakami, \emph{The Frobenius--stable version of the Grauert--Riemenschneider vanishing theorem fails}, \url{https://arxiv.org/abs/2312.13456}.

\bibitem[BBDG]{BBDG}
A.~Beilinson, J.~Bernstein, P.~Deligne, and O.~Gabber, \emph{Faisceaux pervers}, Ast{\'e}risque~\textbf{100}, Soci{\'e}t{\'e} Math{\'e}matique de France, Paris, 2018.

\bibitem[BeL]{BernsteinLunts}
J.~Bernstein and V.~Lunts, \emph{Equivariant sheaves and functors}, Lect. Notes Math.~\textbf{1578}, Springer-Verlag, Berlin, 1994.

\bibitem[BBE]{BerthBlochEsn}
P.~Berthelot, S.~Bloch, and H.~Esnault, \emph{On Witt vector cohomology for singular varieties}, Compos. Math.~\textbf{143} (2007), 363--392.

\bibitem[Bh1]{BhattDerSplinPos}
B.~Bhatt, \emph{Derived splinters in positive characteristic}, Compos. Math.~\textbf{148} (2012), 1757--1786.

\bibitem[Bh2]{Bhatt}
B.~Bhatt, \emph{Perverse sheaves}, \url{https://www.math.purdue.edu/~murayama/MATH731.pdf}.

\bibitem[Bh3]{BhattRPlusCM}
B.~Bhatt, \emph{Cohen-Macaulayness of absolute integral closures}, \url{https://arxiv.org/abs/2008.08070}.

\bibitem[BL1]{BhattLurieRH}
B.~Bhatt and J.~Lurie, \emph{A Riemann-Hilbert correspondence in positive characteristic}, Camb. J. Math.~\textbf{7} (2019), 71--217.

\bibitem[BL2]{BhattLurieRHtors}
B.~Bhatt and J.~Lurie, \emph{A $p$-adic Riemann-Hilbert functor: $\mathbf{Z}/p^n$-coefficients}, in preparation.

\bibitem[BS]{BhattScholze}
B.~Bhatt and P.~Scholze, \emph{Prisms and prismatic cohomology}, Ann.~of Math.~(2)~\textbf{196}, 1135--1275.

\bibitem[Bl]{BliIntHom}
M.~Blickle, \emph{The intersection homology $D$-module in finite characteristic}, Math. Ann.~\textbf{328} (2004), 425--450.

\bibitem[BBG]{BliBoeCart}
M.~Blickle and G.~B\"{o}ckle, \emph{Cartier modules: finiteness results}, J. Reine Angew. Math.~\textbf{661} (2011), 85--123.

\bibitem[BBR]{BlickleBondu}
M.~Blickle and R.~Bondu, \emph{Local cohomology multiplicities in terms of \'etale cohomology}, Ann. Inst. Fourier (Grenoble)~\textbf{55} (2005), 2239--2256.

\bibitem[BP]{BoPi}
G.~B{\"o}ckle and R.~Pink, \emph{Cohomological theory of crystals over function fields}, EMS Tracts Math.~\textbf{9}, European Mathematical Society, Z{\"u}rich, 2009.

\bibitem[Bri]{Bridgland}
N.~Bridgland, \emph{On the de~Rham homology of affine varieties in characteristic $0$}, J. Algebra~\textbf{643} (2024), 203--223.

\bibitem[Bry]{Brylinski}
J.~L.~Brylinski, \emph{Transformations canoniques, dualit\'{e} projective, th\'{e}orie de Lefschetz, transformations de Fourier et sommes trigonom\'{e}triques}, Ast\'{e}risque~\textbf{140-141} (1986), 3--134.

\bibitem[Ca]{CassPerv}
R.~Cass, \emph{Perverse $\mathbf{F}_p$-sheaves on the affine Grassmannian}, J. Reine Angew. Math.~\textbf{785} (2022), 219--272.

\bibitem[CL]{CassLau}
R.~Cass and J.~Louren{\c{c}}o, \emph{Mod $p$ sheaves on Witt flags}, preprint, 2023.

\bibitem[Co]{Conrad}
B.~Conrad, \emph{Grothendieck duality and base change}, Lect. Notes Math.~\textbf{1750}, Springer, Berlin, 2000.

\bibitem[EK]{EK}
M.~Emerton and M.~Kisin, \emph{The Riemann-Hilbert correspondence for unit $F$-crystals}, Ast{\'e}risque~\textbf{293}, Soci{\'e}t{\'e} Math{\'e}\-matique de France, Paris, 2004.

\bibitem[Ga]{Gabbert}
O.~Gabber, \emph{Notes on some $t$-structures}, in: Geometric aspects of Dwork theory, II, pp.~711--734, Walter de Gruyter GmbH \& Co. KG, Berlin, 2004.

\bibitem[GLS]{GarciaSabbah}
R.~Garc{\'{\i}}a L{\'o}pez and C.~Sabbah, \emph{Topological computation of local cohomology multiplicities}, Collect. Math.~\textbf{49} (1998), 317--324.

\bibitem[GM1]{GM1}
M.~Goresky and R.~MacPherson, \emph{Intersection homology theory}, Topology~\textbf{19} (1980), 135--162.

\bibitem[GM2]{GM2}
M.~Goresky and R.~MacPherson, \emph{Intersection homology. II}, Invent. Math.~\textbf{72} (1983), 77--129.

\bibitem[Ha]{Hartshorne:AlgdR}
R.~Hartshorne, \emph{On the de~Rham cohomology of algebraic varieties}, Inst. Hautes \'Etudes Sci. Publ. Math.~\textbf{45} (1975), 5--99.

\bibitem[HaS]{HartSpeis}
R.~Hartshorne and R.~Speiser, \emph{Local cohomological dimension in characteristic $p$}, Ann.~of Math.~(2)~\textbf{105} (1977), 45--79.

\bibitem[HH]{HHRPlusCM}
M.~Hochster and C.~Huneke, \emph{Infinite integral extensions and big Cohen-Macaulay algebras}, Ann.~of Math.~(2)~\textbf{135} (1992), 53--89.

\bibitem[HTT]{HTTDmodules}
R.~Hotta, K.~Takeuchi, and T.~Tanisaki, \emph{$D$-modules, perverse sheaves, and representation theory}, Prog. Math.~\textbf{236}, Birkh\"{a}user, Boston, 2008.

\bibitem[HuS]{HS:TAMS}
C.~Huneke and R.~Sharp, \emph{Bass numbers of local cohomology modules}, Trans. Amer. Math. Soc.~\textbf{339} (1993), 765--779.

\bibitem[Ku]{Kuperberg}
G.~Kuperberg, (user \href{https://mathoverflow.net/users/1450/greg-kuperberg}{1450}), \emph{Diffeomorphic K\"ahler manifolds with different Hodge numbers}, \href{https://mathoverflow.net/q/42746}{Mathoverflow}, 2010.

\bibitem[Ly1]{Lyubeznik:Invent}
G.~Lyubeznik, \emph{Finiteness properties of local cohomology modules (an application of $D$-modules to commutative algebra)}, Invent. Math.~\textbf{113} (1993), 41--55.

\bibitem[Ly2]{Lyubeznik:Fmod}
G.~Lyubeznik, \emph{$F$-modules: applications to local cohomology and $D$-modules in characteristic $p>0$}, J. Reine Angew. Math.~\textbf{491} (1997), 65--130.

\bibitem[Ly3]{Lyubeznik:survey}
G.~Lyubeznik, \emph{A partial survey of local cohomology}, in: Local cohomology and its applications, (Guanajuato, 1999), Lecture Notes in Pure and Appl. Math.~\textbf{226}, pp.~121--154, Dekker, New York, 2002.

\bibitem[Ly4]{Lyubeznik:Compositio}
G.~Lyubeznik, \emph{On the vanishing of local cohomology in characteristic $p>0$}, Compos. Math.~\textbf{142} (2006), 207--221.

\bibitem[LSW]{LSW}
G.~Lyubeznik, A.~K.~Singh, and U.~Walther, \emph{Local cohomology modules supported at determinantal ideals}, J. Eur. Math. Soc.~\textbf{18} (2016), 2545--2578.

\bibitem[MZ]{Ma-Zhang}
L.~Ma and W.~Zhang, \emph{Eulerian graded $\mathcal D$-modules}, Math. Res. Lett.~\textbf{21} (2014), 149--167.

\bibitem[MP]{MuPo}
M.~Musta\c t\u a and M.~Popa, \emph{Hodge filtration on local cohomology, Du Bois complex, and local cohomological dimension}, Forum Math. Pi~\textbf{10} (2022), Paper No.~e22, 58~pp.

\bibitem[Og]{OgusLocalCoh}
A.~Ogus, \emph{Local cohomological dimension of algebraic varieties}, Ann.~of Math.~(2)~\textbf{98} (1973), 327--365.

\bibitem[RSW1]{RSW1}
T.~Reichelt, M.~Saito, and U.~Walther, \emph{Dependence of Lyubeznik numbers of cones of projective schemes on projective embeddings}, Selecta Math.~\textbf{27} (2021), Paper No. 6, 22~pp.

\bibitem[RSW2]{RSW2}
T.~Reichelt, M.~Saito, and U.~Walther, \emph{Topological calculation of local cohomological dimension}, J. Singul.~\textbf{26} (2023), 13--22.

\bibitem[SS]{SannaiSingh}
A.~Sannai and A.~K.~Singh, \emph{Galois extensions, plus closure, and maps on local cohomology},
Adv. Math.~\textbf{229} (2012), 1847--1861.

\bibitem[Sc]{Schedlmeier}
T.~Schedlmeier, \emph{A Riemann-Hilbert correspondence for Cartier crystals}, \url{https://arxiv.org/abs/1812.00256}.

\bibitem[Sm1]{SmRPlusFRat}
K.~E.~Smith, \emph{Tight closure of parameter ideals}, Invent. Math.~\textbf{115} (1994), 41--60.

\bibitem[Sm2]{SmRPlusFRatGr}
K.~E.~Smith, \emph{Tight closure and graded integral extensions}, J. Algebra~\textbf{175} (1995), 568--574.

\bibitem[Sta]{Staebler}
A.~St\"abler, \emph{Intermediate extensions of perverse constructible {$\mathbb{F}_p$}-sheaves commute with smooth pullbacks}, J. Number Theory~\textbf{197} (2019), 317--340.

\bibitem[SP]{stacks-project}
The Stacks Project Authors, \emph{Stacks Project}, \url{https://stacks.math.columbia.edu}.

\bibitem[Sw1]{Switala}
N.~Switala, \emph{Lyubeznik numbers for nonsingular projective varieties}, Bull. Lond. Math. Soc.~\textbf{47} (2015), 1--6.

\bibitem[Sw2]{SwitaladR}
N.~Switala, \emph{On the de~Rham homology and cohomology of a complete local ring in equicharacteristic zero}, Compos. Math.~\textbf{153} (2017), 2075--2146.

\bibitem[Wa]{Wang}
B.~Wang, \emph{Lyubeznik numbers of irreducible projective varieties depend on the embedding}, Proc. Amer. Math. Soc.~\textbf{148} (2020), 2091--2096.

\bibitem[ZhW]{ZhangLyuIndep}
W.~Zhang, \emph{Lyubeznik numbers of projective schemes}, Adv. Math.~\textbf{228} (2011), 575--616.

\bibitem[ZhY]{YiZhang}
Y.~Zhang, \emph{Graded $F$-modules and local cohomology}, Bull. Lond. Math. Soc.~\textbf{44} (2012), 758--762.

\end{thebibliography}
\end{document}